\documentclass[oneside]{amsart}

   \usepackage{amsthm}
   \usepackage{amsmath}
   \usepackage{amssymb}
   \usepackage{amsfonts}
   \usepackage{amscd}
   \usepackage{mathrsfs}  
   \usepackage{bm}   
   \usepackage[all]{xy}
   \makeatletter
    
    \@addtoreset{equation}{section}
  \makeatother   
   
  \newtheorem{theorem}{Theorem}[section]
  \newtheorem{definition}[theorem]{Definition}
  \newtheorem{remark}[theorem]{Remark} 
  \newtheorem{lemma}[theorem]{Lemma}
  \newtheorem{proposition}[theorem]{Proposition}
  \newtheorem{corollary}[theorem]{Corollary}
  
  \newtheorem{example}[theorem]{Example}

  \def\R{\mathbb{R}}
  \def\Z{\mathbb{Z}}
  \def\Q{\mathbb{Q}}
  \def\C{\mathbb{C}}
  \newcommand{\id}{\mathrm{id}}
  \newcommand{\simeqto}{\xrightarrow{\sim}}
  \newcommand{\into}{\hookrightarrow}
   
  \newcommand{\End}{\mathrm{End}}
  \newcommand{\gr}{\mathrm{Gr}}
  \newcommand{\Hom}{\mathrm{Hom}}
  
  \def\Image{\mathrm{Im}}
  \def\del{\partial}
  \def\delbar{{\overline{\del}}} 
  \def\T{\mathbb{T}}
  \def\H{\mathbb{H}}
  \def\P{\mathbb{P}}    
  
  \def\L{\mathcal{M}}
  \def\w{{\sf w}}
  
  \def\dmod{\mathcal{D}}

  \newcommand{\rmod}{\lambda}
  \newcommand{\para}{\tau} 
  \newcommand{\scale}{\theta} 
  \def\m{\sigma}
  \def\HH{\mathrm{HH}}
  \def\id{\mathrm{id}}
  \def\Res{\mathrm{Res}}
  \newcommand{\pole}{D}
  
  \newcommand{\amb}{X}
  \newcommand{\tot}{s(\widetilde{A}^{\bullet,\bullet})}
  \newcommand{\rtot}{s(\widetilde{C}^{\bullet,\bullet})}
  \def\i{{\tt i}}
  \newcommand{\disk}{{\Delta_s}}
  \newcommand{\kah}{\omega_{\rm K\ddot{a}h}}
  \newcommand{\poly}{{\sf P}}
  \newcommand{\facet}{{\sf F}}
  \newcommand{\face}{{\sf Q}}
  \newcommand{\cone}{\sigma}
  \newcommand{\ray}{\rho}
  \newcommand\fano{{\mathrm{F}}}
  
\begin{document}
  \title{Hodge-Tate conditions for Landau-Ginzburg models}
  \author{Yota Shamoto}
  \footnote[0]{Kavli Institute for the Physics and Mathematics of the Universe (WPI), University of Tokyo, 5-1-5 Kashiwanoha, Kashiwa, Chiba 277-8583, Japan
  Math. Subject Classification: 14J33. Communicated by T. Mochizuki. Received September 19, 2017. Revised December 27, 2017}
\address{Kavli Institute for the Physics and Mathematics of the Universe (WPI), University of Tokyo, 5-1-5 Kashiwanoha, Kashiwa, Chiba 277-8583, Japan}
\email{yota.shamoto@ipmu.jp} 
            \date{}
            \begin{abstract}            
             We give a sufficient condition for a class of 
             tame compactified Landau-Ginzburg models  
             in the sense of Katzarkov-Kontsevich-Pantev 
             to satisfy some versions of their conjectures. 
             We also give examples which satisfy the condition.
             The relations to the quantum D-modules of Fano manifolds and 
             the original conjectures are explained in Appendices.
            \end{abstract}
          \maketitle
          \setcounter{tocdepth}{1}
          \tableofcontents
  \section{Introduction}
        Let $\amb$ be a smooth projective variety over $\C$ 
        with a Zariski open subset $Y$.
        We assume that $\pole:=\amb\setminus Y$ is a simple normal crossing hypersurface.
        Let  $f:X\to \P^1$ be a flat projective morphism such that the restriction 
        $\w:=f_{|Y}$ is a regular function. 
        In general, the meromorphic
        flat connection $(\mathcal{O}_\amb(*\pole), d+df)$ has irregular singularities
        along $\pole$.
        Let $H_{\rm dR}^\bullet(Y, \w)$ denote the 
        de Rham cohomology group of $(\mathcal{O}_\amb(*\pole), d+df)$.        
        It has been studied from the viewpoint of generalized Hodge theories. 
        (See twistor $\dmod$-modules \cite{moctwi}, \cite{mocgkz},
         irregular Hodge structures \cite{delirr}, \cite{esy},  \cite{sabirr}, \cite{sabont},
         non-commutative Hodge structures \cite{kkp08}, \cite{kkp14}, 
        TERP-structures \cite{hernil}, and so on.)

         In some cases, $(Y, \w)$ can be considered as a ^^ mirror dual' of a 
         smooth projective Fano variety $\fano$ called a sigma model.
         In that case, $(Y, \w)$ is called a Landau-Ginzburg model, and 
         it is predicted that some categories associated to $(Y,\w)$ are 
         equivalent to the corresponding categories associated to $\fano$.
         This prediction is called a Homological Mirror Symmetry  conjecture (HMS).
         Some parts of HMS are proved in some cases \cite{aurmir}, \cite{aurwei}, \cite{uedhom}.
         
         From this point of view, 
         Katzarkov-Kontsevich-Pantev \cite{kkp14}
         proposed some conjectures 
         as conjectural
         consequences of HMS.
         As emphasized in \cite{kkp14}, 
         some of their conjectures can be seen as 
         ^^ ^^ purely algebro-geometric'' conjectures
         on the generalized Hodge theory of $H_{\rm dR}^\bullet(Y, \w)$.
         Such conjectures are the main subjects of this paper.
         
         As an introduction, we survey some versions of the conjectures in 
         \S \ref{int num} and \S \ref{int spe}. 
         (The relations to the original ones are explained in Appendix \ref{app kkp}.)
         Then, we explain our main result in \S \ref{int main}. 
         In this paper, we always assume that the pole divisor
         $(f)_\infty$ of $f$ is reduced and the support 
         $|(f)_\infty|$ is equal to $\pole$, although 
         this assumption is more restrictive than that of \cite{kkp14}.

         \subsection{Hodge numbers}\label{int num}
         The cohomology group $H_{\rm dR}^\bullet(Y, \w)$ is 
         given by taking the hypercohomology 
         of the complex $(\Omega^\bullet_\amb(* \pole),d+df\wedge)$.
         There are $\mathcal{O}_X$-coherent subsheaves $\Omega_f^k$ of
         $\Omega_\amb^k(*\pole)$ which give a subcomplex
         $(\Omega^\bullet_f, d+df\wedge)$ 
         (see \S \ref{ss kon }).
         It is known that the inclusion 
         $(\Omega^\bullet_f, d+df\wedge)
         \hookrightarrow (\Omega^\bullet_\amb(* \pole),d+df\wedge)$
         is a quasi-isomorphism (see \cite[Corollary 1.4.3]{esy}). 
         The Hodge number $f^{p, q}(Y, \w)$ is defined by 
         \begin{align*}
         f^{p, q}(Y, \w):=\dim H^q(X, \Omega_f^p).
         \end{align*}
         It is proved by Esnault-Sabbah-Yu, Kontsevich, and M. Saito \cite{esy}
          that we have $\dim H^k(Y, \w)=\sum_{p+q=k}f^{p, q}(Y, \w)$,
         which can be considered as a consequence of 
         $E_1$-degeneration property of the ^^ ^^ Hodge filtration''.
         
         Take sufficiently small holomorphic disk $\Delta$ in $\P^1$ centered at infinity 
         so that $Y_b:=f^{-1}(b)$ is smooth for any $b\in \Delta\setminus\{\infty\}$.
         It is proved in \cite{kkp14} (see also \cite{dimont})
         that we have the following equality:
         \begin{align*}
          \dim H^k_{\mathrm{dR}}(Y, \w)=\dim H^k(Y, Y_b),
         \end{align*}
         where $b\in \Delta\setminus \{\infty\}$, 
         and $ H^k(Y, Y_b)$ denotes the relative cohomology with $\C$-coefficient.
         In our situation, the monodromy $T_k$ at infinity is known to be unipotent (\cite[Theorem I']{lanont}).
         Let ${}^kW$ be the monodromy weight filtration of 
         $N_k:=\log T_k$ on $H^k(Y, Y_b)$
         centered at $k$ (see (\ref{dwf}), (\ref{dwf2})). 
         The number $h^{p, q}(Y, \w)$ 
         is defined by 
         \begin{align*} 
         h^{p, q}(Y, \w):=\dim \gr^{{}^kW}_{2p}H^k(Y, Y_b),\ \ \  (k=p+q).
         \end{align*}
         
         By a HMS consideration, 
         Katzarkov-Kontsevich-Pantev \cite{kkp14} conjectured: 
         \begin{align}\label{con h}
          f^{p, q}(Y, \w)=h^{p, q}(Y, \w).
         \end{align}
         It is easy to observe that the conjecture (\ref{con h}) does not hold 
         if the fiber $D$ at infinity is smooth 
         and $f^{p, q}(Y, \w)$ are not  zero for two different pairs $(p, q)$ and $(p',q')$
         with $p+q=p'+q'$.
         Actually, such example is given in \cite{lunlan}.         
         However, in loc. cit., there are examples of $(X, f)$ 
         which satisfy (\ref{con h}).
         There remains a question  
         when the equality (\ref{con h}) holds.
         The counter-example suggests that 
         we need to impose some conditions on the degeneration property 
         of $Y_b$ as $b\to \infty$.

         \subsection{Speciality}\label{int spe}        
         Let $(\rmod, \para)$ be a pair of complex numbers. 
         The dimension of the hypercohomology $\H^\bullet(\amb;(\Omega_f^\bullet,\rmod d+\para df\wedge))$
         is known to be independent of the choice of $(\rmod, \para)$ (\cite{esy}, \cite{moctwi}). 
          Let $\C_\rmod$, $\C_\para$ be complex planes with coordinate 
         $\rmod$ and $\para$ respectively.
         Put $\P_\rmod^1:=\C_\rmod\cup \{\infty\}$ and 
         $S:=\P^1_\rmod \times \C_\para$.
         It follows that we have a locally free $\Z/2\Z$-graded $\mathcal{O}_S(*(\rmod)_\infty)$-module
         ${}^{\mathfrak{b}}H$ whose fiber at $(\rmod,\para)$ 
         is $\H^\bullet(\amb;(\Omega_f^\bullet,\rmod d+\para df\wedge))$.
         The $\mathcal{O}_S(*(\rmod)_\infty)$-module ${}^{\mathfrak{b}}H$
         is equipped with a grade-preserving meromorphic flat connection:
         $${}^{\mathfrak{b}}\nabla:
         {}^{\mathfrak{b}}H\to {}^{\mathfrak{b}}H\otimes_{\mathcal{O}_S}\Omega^1_S(\log \rmod \para)((\rmod)_0),$$
         where $\Omega^1_S(\log \rmod \para)((\rmod)_0)$ denotes the 
         $\mathcal{O}_S$-module locally generated by 
         $\rmod^{-1}\para^{-1}d\para$ and $\rmod^{-2}d\rmod$.
                  
         For a smooth projective Fano variety $\fano$, 
         the quantum $\dmod$-module for the quantum parameters 
         $c_1(\fano)\log \para \in H^2(\fano)$
         gives a similar pair $({}^\mathfrak{a}H,{}^{\mathfrak{a}}\nabla)$.
         These pairs are considered as one parameter variation of
         non-commutative Hodge structures 
         $({}^{\sf A}H,{}^{\sf A}\nabla):=({}^\mathfrak{a}H,{}^{\mathfrak{a}}\nabla)_{|\para=1}$, 
         and $({}^{\sf B}H,{}^{\sf B}\nabla):=({}^\mathfrak{b}H,{}^{\mathfrak{b}}\nabla)_{|\para=1}$.
         It is conjectured \cite[Conjecture 3.11]{kkp14} that homological mirror correspondences for a pair $\fano\mid(Y, \w)$         
         should induce an isomorphism 
         $({}^\mathfrak{a}H,{}^{\mathfrak{a}}\nabla)\simeq ({}^\mathfrak{b}H,{}^{\mathfrak{b}}\nabla)$ 
         (more precisely, we need to fix more data to determine the mirror pair).
         
         On the one hand,  $({}^{\sf A}H,{}^{\sf A}\nabla)$ has a trivial logarithmic extension to $\rmod=\infty$. 
         On the other hand,
         it is a non-trivial problem
         to construct a logarithmic extension of $({}^{\sf B}H,{}^{\sf B}\nabla)$
         such that the induced vector bundle on $\P^1_\rmod$ is trivial.
         The problem is called Birkhoff problem (see e.g.\cite{sabiso}),
         and the solution to the problem  for $({}^{\sf B}H,{}^{\sf B}\nabla)$
         plays a key roll in the construction of primitive forms
         \cite{sabgau}, \cite{sabhyp}. 
         
         Katzarkov-Kontsevich-Pantev observed that
         the trivial solution of the Birkhoff problem for 
         the connection $({}^{\sf A}H,{}^{\sf A}\nabla)$
         can be described in terms of the 
         Deligne's canonical extension and 
         the weight filtration for the nilpotent part of 
         the residue endomorphism along $\{\rmod=\infty\}$.         
         An extension given in a similar way 
         is called a skewed canonical extension in \cite{kkp14}. 
         The skewed canonical extension can be defined for more general objects including 
         $({}^{\sf B}H,{}^{\sf B}\nabla)$.
         The property that the skewed canonical extension 
         gives a solution to the Birkhoff problem 
         is called ^^ ^^ speciality" (see \cite[Definition 3.21]{kkp14}, 
         or Definition \ref{special} for details).

         From the point of view of the conjecture 
         $({}^\mathfrak{a}H,{}^{\mathfrak{a}}\nabla)
         \simeq ({}^\mathfrak{b}H,{}^{\mathfrak{b}}\nabla)$, 
         they conjectured 
         that $({}^{\sf B}H,{}^{\sf B}\nabla)$ is special (\cite[Conjecture 3.22 (a)]{kkp14}).
         Combining it with their unobstructedness result on 
         the versal deformation of $(Y, \w)$, 
         they also conjectured
         the existence           
         of a version of a primitive form under the assumption that $\Omega^{\dim\amb}_\amb(\pole)$
         is trivial
         (\cite[Conjecture 3.22 (b)]{kkp14}). 
         \subsection{Rescaling structures and Hodge-Tate conditions}\label{int main}
         To treat the conjectures in \S \ref{int num} and \S \ref{int spe} simultaneously, 
         we introduce a notion of rescaling structure (See \S \ref{se res} for details).
         Let $\m: \C^*_\scale\times S\to S$ be a $\C^*_\scale$ 
         be the action of $\C^*_\scale$ defined by 
         $(\scale,\rmod,\para)\mapsto (\scale\rmod,\scale\para)$.
         Let $p_2:\C^*_\scale\times S\to S$ denote the projection.
         A rescaling structure is a triple $(\mathcal{H},\nabla,\chi)$ of
         $\Z$-graded locally free $\mathcal{O}_S(*(\rmod)_\infty)$-module 
         $\mathcal{H}$, 
         a grade-preserving meromorphic flat connection 
         $$\nabla:\mathcal{H}\to 
         \mathcal{H}\otimes \Omega_S^1(\log \rmod\para)((\rmod)_0),$$
         and 
         an isomorphism
         $\chi: p_2^*\mathcal{H}\simeqto\m^*\mathcal{H}$ 
         with some conditions 
         (see Definition \ref{definition of rescaling structure}).
         
         For a rescaling structure $(\mathcal{H},\nabla,\chi)$,
         take a fiber $V$ of $\mathcal{H}$ at $(\rmod,\para)=(1,0)$. 
         Under an assumption, 
         we associate two filtrations $F$ and $W$ on $V$,  
         where $F$ is called Hodge filtration and 
         $W$ is called weight filtration of 
         $\mathcal{H}$ (\S \ref{reshod}). 
         We also define an abstract version of Hodge numbers 
         $f^{p, q}(\mathcal{H})$
         and $h^{p, q}(\mathcal{H})$.
         
         The rescaling structure is said to satisfy the Hodge-Tate condition 
         if these two filtration behave like a Hodge filtration and a weight filtration of 
         a mixed Hodge structure of Hodge-Tate type
         in the sense of Deligne \cite{delloc} 
         (see Definition \ref{def ht} for details). 
         If $(\mathcal{H},\nabla,\chi)$ satisfies Hodge-Tate condition, 
         we have $f^{p, q}(\mathcal{H})=h^{p, q}(\mathcal{H})$, 
         and we also have that $\mathcal{H}_{|\para=1}$ is special.
         
         In Appendix \ref{app A}, we show that
         a ^^ ^^ Tate twisted" version 
         $\mathcal{H}_\fano$ of ${}^{\mathfrak{a}}H$ 
         comes equipped with a rescaling structure 
         for any smooth projective Fano variety $\fano$.
         The rescaling structure $\mathcal{H}_\fano$
         satisfies the Hodge-Tate condition, 
         and we have 
         $$f^{p, q}(\mathcal{H}_\fano)=h^{p, q}(\mathcal{H}_\fano)
         =\dim H^q(\fano,\Omega_\fano^{n-p}).$$
         
         For the pair $(X, f)$, we also have a version 
         $\mathcal{H}_f$ of ${}^\mathfrak{b}H$, 
         which comes equipped with a rescaling structure
         (See \S \ref{section LG}. 
         The relation between $\mathcal{H}_f$ and ${}^\mathfrak{b}H$
         is given  in Appendix \ref{app kkp}).
          The main result of this paper is the following:
         \begin{theorem}[Theorem \ref{main theorem}]\label{thm intro}
         Let $\mathcal{H}_f$ be the rescaling structure for $(X, f)$.
          \begin{enumerate} 
           \item If $\mathcal{H}_f$ satisfies the Hodge-Tate condition, 
                    then the equation $(\ref{con h})$ holds 
                    and $\mathcal{H}_{f|\para=1}$ is special.
           \item The rescaling structure $\mathcal{H}_f$ satisfies Hodge-Tate condition
                    if and only if the mixed Hodge structure
                    $(H^k(Y,Y_\infty;\Q), F, W)$ is Hodge-Tate for every $k\in\Z$.
           \end{enumerate}
         \end{theorem}
         The definition of the mixed Hodge structure $(H^k(Y,Y_\infty;\Q), F, W)$
         is given in \S \ref{MHC}.
         In \S \ref{section Ex}, we also give some examples
         such that $\mathcal{H}_f$ satisfies the Hodge-Tate condition in the case where
         the dimension of $X$ is two or three.
  \section{Rescaling structures}\label{se res}
      \subsection{Holomorphic extensions and filtrations}
          Let $\C$ denote a complex plane.  
          Set $\C^*:=\C\setminus \{0\}$. 
          Let $H$  be a finitely generated locally free $\mathcal{O}_\C(*\{0\})$-module.
          Let $V$ denote the fiber of $H$ at $1\in\C$. 
          Assume that we are given an increasing filtration 
          $G_\bullet V=(G_mV\mid m\in \Z)$ on $V$ such that 
          \begin{align}\label{finie}
          G_mV:=\begin{cases}
                       0& (m\ll 0) \\
                       V& (m\gg 0).
                       \end{cases}
          \end{align}
          We shall recall some methods to construct an extension 
          of $H$ to an $\mathcal{O}_\C$-module
          by using $G_\bullet V$.
          Here, by an extension of $H$, 
          we mean a locally free $\mathcal{O}_\C$-submodule $L$ of $H$
          such that $L\otimes \mathcal{O}_\C(*\{0\})=H$. 
      
     \subsubsection{Construction using $\C^*$-actions}\label{action}
         Let $\mathrm{m}:\C^*\times \C^*\to \C^*$ 
         and $\m:\C^*\times \C\to \C$ denote the multiplications.
         Let $p_2:\C^*\times\C\to \C$ be the projection.      
         Assume that $H$ is $\C^*$-equivariant with 
         respect to $\m$.
         Namely, 
         we have an isomorphism 
         $\chi: p_2^*H\simeqto \m^*H$ 
         with the cocycle condition: 
         $$ (\mathrm{m}\times \id_\C)^*\chi=(\id_{\C^*}\times \m)^*\chi\circ p_{23}^*\chi,
         $$
         where $p_{23}:\C^*\times\C^*\times\C\to \C^*\times\C$ is given by $p_{23}(t_1,t_2,z):=(t_2,z)$.
          This case is considered in \cite[Lemma 19]{simnon}, for example.
          For any vector $v\in V$, 
          there is a unique invariant section 
          $\phi_v\in \Gamma(\C,H)$ with $\phi_v(1)=v$.
          There exists a unique extension $L_1$ such that 
          $v\in G_mV$ if and only if $\phi_v\in L_1(m\{0\})$. 
           The extension $L_1$ is isomorphic to the extension 
           $\sum_m G_mV \otimes \mathcal{O}_\C(-m\{0\})$ 
           of 
           $V\otimes\mathcal{O}(*\{0\})$. 
           This construction gives a one to one correspondence between the 
           sets of increasing filtrations on $V$ with (\ref{finie}) 
           and $\C^*$-equivariant holomorphic extensions of $H$.
           \begin{example}\label{exR}
           Let $V$ be a finite dimensional $\C$-vector space with a decomposition 
           $V=\bigoplus_{p\in \Z}V_p$. Put $H:=\mathcal{O}_\C(*\{0\})\otimes_\C V$.
           Remark that 
           $p_2^*H\simeq \mathcal{O}_{\C^*\times \C}(*\C^*\times \{0\})\otimes V
           \simeq\m^*H$.
           Define $\chi: p_2^*H\simeqto \m^*H $
           by $\chi_{|\mathcal{O}_{\C^*\times\C}(*\C^*\times \{0\})\otimes V_p}(t, z)
           :=t^p\otimes\id_{V_p}$.
            Consider $V$ as the fiber of $H$ at $1\in \C$.
            Then the trivial extension $L_1:=\mathcal{O}_\C\otimes V$
            corresponds to the following filtration: 
            $$G_{m}V=\bigoplus_{-p\leq m} V_p.$$  
            Indeed, 
            for $v\in V_p$, the invariant section $\phi_v$ is given by 
            $\phi_v(z)=z^{p}v\in L_1(-p\{0\})$.
           \end{example}
      \subsubsection{Double complex}
        Let $(C^{\bullet,\bullet},\delta_1,\delta_2)$ be a double complex of $\C$-vector spaces
        where $\delta_1:C^{p, q}\to C^{p+1,q}$ and 
        $\delta_2:C^{p, q}\to C^{p, q+1}$ are the differentials.
        We assume that $C^{p, q}=0$ if $p<0$ or $q<0$, 
        and that 
        the total complex $(C^\bullet,\delta)$ has finite dimensional cohomology.
        Here, we put $C^\ell:=\bigoplus_{p+q=\ell}C^{p, q}$ and $\delta:=\delta_1+\delta_2$.
        Let $F$ be the filtration on $(C^\bullet,\delta)$
        given by $F_mC^\ell:=\bigoplus_{p+q=\ell, -p\leq m}C^{p, q}$.
        We also assume that 
        the morphisms
        $  H^k(F_m(C^\bullet, \delta))\to H^k(C^\bullet, \delta)$
        are injective for all $k$ and $m$.
        
        Put $\mathcal{C}^{p, q}:=\mathcal{O}_\C\otimes C^{p, q}$ and 
        $\mathcal{C}^\ell:=\bigoplus_{p+q=\ell}\mathcal{C}^{p, q}$.
        We have a complex 
        $(\mathcal{C}^\bullet, z\delta_1+\delta_2).$
        Let $L_1$ be the $k$-th cohomology group of this complex. 
        By the assumption, $L_1$
        is a finitely generated locally free $\mathcal{O}_\C$-module.
        Put $H:=L_1\otimes \mathcal{O}_\C(*\{0\})$ and 
        consider $L_1$ as an extension of $H$.
        Define $\chi_p: p_2^*\mathcal{C}^{p, q}\simeqto \sigma^*\mathcal{C}^{p, q}$ by 
        $\chi_p(t, z):=t^p\otimes \id$. 
        This induces an isomorphism $\chi: p_2^*H\simeqto \m^*H$
        with the cocycle condition.
        \begin{lemma}\label{doucpx}
         Consider the $k$-th cohomology ${H}^k(C^\bullet,\delta)$ as 
         the fiber of $H$ at $1\in \C$.
         Then the extension $L_1$ corresponds to the following filtration:
         \begin{align*}
         G_mH^k(C^\bullet,\delta):=\Image(H^k(F_m(C^\bullet, \delta))\to H^k(C^\bullet, \delta)).
         \end{align*} 
         \end{lemma}
         \begin{proof}
          Put $F_m\mathcal{C}^k:=\bigoplus_{p+q=k, p\geq -m}\mathcal{C}^{p, q}$.
          It induces a filtration on the complex $(\mathcal{C}^\bullet, z\delta_1+\delta_2)$, 
          which is also denoted by $F$.
          The induced filtration on $L_1$ is also denoted by $F$.
          By the assumption, 
          we have
          $\gr^F_\ell L_1\simeq H^k(\gr^F_\ell(\mathcal{C}^\bullet)).$
          Hence it reduces to the case 
          where there exists a $p_0\in\Z$ such that $\mathcal{C}^{p, q}=0$ for $p\neq p_0$.
          In this case,  we have $L_1\simeq H^{k-p_0}(C^{p_0,\bullet},\delta_2)\otimes \mathcal{O}_\C$,
          and we obtain the conclusion by Example \ref{exR}.
         \end{proof}
      \subsubsection{Construction using flat connections with regular singularities}
      \label{connection}
          Assume that $H$ is equipped with 
          a flat connection $\nabla$ with a regular singularity at $\{0\}$.   
          We also assume that each $G_kV$ is invariant with respect to 
          the monodromy of $\nabla$. 
          This case is considered in 
          \cite{kkp08}, \cite{kkp14}, \cite{sabiso} for example.
          We have the flat subbundles $G_\bullet H$ on $H$
          such that the fiber of $G_kH$ at $1$ is $G_kV$.
          For any $t\in\C^*$, let $V_t$ be the fiber of $H$ at $t$.
          Let $G_\bullet V_t$ denote the induced filtration on $V_t$.
          Set $I_t:=\{st\mid 0<s\leq 1\}$.
          For any vector $v\in V_t$, we have the flat section 
          $\psi_{v, t}\in \Gamma(I_t,H)$
          with $\psi_{v, t}(t)=v$.
          There exists a unique logarithmic lattice $L_2$ with the following property:
          Fix a frame of $L_2$ near 0, 
          and let $||*||_{L_2}$ be the Hermitian metric on $L_2$ near $0$
          so that the frame is a orthogonal with respect to $||*||_{L_2}$.
          A vector $v\in V_t$ is contained in $ G_m V_t$ if and only if $\psi_{v, t}$ satisfies 
           \begin{align*}
            || \psi_{v, t}(r\cdot t)||_{L_2}\leq C|r|^{-m}(-\log r)^N &\ \ \ \  (0<r \ll 1)
          \end{align*}
          for some positive constants $C$ and $N$.
          This construction also gives a one to one correspondence between 
          the logarithmic extension of $H$ and 
          monodromy invariant filtrations on $V$ with (\ref{finie}). 
          \subsubsection{Characterization by using the Deligne lattice}
          The extension $L_2$ can  be characterized 
          by using the Deligne lattice of $(H,\nabla)$.
          Let $L'$ be the Deligne lattice of $(H,\nabla)$, which means that 
          $L'$ is the logarithmic at {0} and 
          the residue with eigenvalues whose real parts are contained in $(-1,0]$.
          The flat subbundles $G_mH$ extend to $\{0\}$ and give subbundles of $L'$.
          Let $G_mL'$ denote the subbundles of $L'$.
          \begin{lemma}[{\cite[\S 3.3.1]{kkp14}}]
           The extension $L_2$ is given by 
           \begin{align*}
            L_2=\sum_{m\in \Z}G_mL'(-m\{0\})
           \end{align*}
           as a submodule of $L'(*\{0\})$.
           \begin{proof}
           It is enough to show that $L_2=L'$ if $G_{\bullet}V$ is given 
           by $G_{-1}V=0$ and $G_0V=V$. 
           Let $\mathrm{rk}L'$ be the rank of $L'$.
           We have an isomorphism of logarithmic connections 
           $(L',\nabla)\simeq(\mathcal{O}_\C^{\oplus\mathrm{rk}L'},\nabla')$, 
           where 
           $\nabla'=d-\mathcal{U}t^{-1}d t$ for a matrix
           $\mathcal{U}\in \End(\C^{\oplus\mathrm{rk}L'})$
           with eigenvalues whose real parts are contained in $[0,1)$. 
           Take the standard frame $v_1,\dots, v_{\mathrm{rk}L'}$ of 
           $\mathcal{O}_\C^{\oplus\mathrm{rk}L'}$.
           It induces a Hermitian metric $||*||_{L'}$.
           For fixed $t\in \C^*$, take $\alpha\in \C$ with $\exp \alpha=t$.  
           We have the flat section 
           $\psi_i(r \cdot t):=\exp(\alpha \log r{\mathcal{U}})v_i(rt)$ on 
           $I_t$ for all $i=1,\dots, \mathrm{rk}L'$. 
           Since the flat sections on $I_t$ are $\C$-linear combinations of $\psi_i$, 
           we obtain the conclusion.
           \end{proof}
          \end{lemma}
   \subsubsection{Relation between two constructions}
       Assume that $H$ is $\C^*$-equivariant and 
       equipped with a flat connection $\nabla$.
       We also assume the compatibility of the action and flat connection.
       In other words, for all $t\in\C^*$, 
       the action of $t$ on $H$ is assumed to be 
       equal to the parallel transport of $\nabla$.
       Then we have the following.
       \begin{lemma}\label{relation of construction}
        The connection $\nabla$ is regular singular at $\{0\}$.
        The extensions $L_1$ and $L_2$
        constructed in $\S \ref{action}$ and 
        $\S \ref{connection}$ coincide.
        \begin{proof}
         By the compatibility of the action and the connection,
         the invariant section $\phi_v$ for $v\in V$ is $\nabla$-flat.
         Since $L_2$ is generated by 
         $t^{m}\phi_v$ $(v\in G_m V,\ t \text{ is a coordinate on } \C)$, 
         it gives a logarithmic extension of $H$. 
         This shows that the connection $\nabla$ is regular singular at $\{0\}$. 
         Fix a trivialization of $L_1$ around $\{0\}$ and 
         let $||*||_{L_1}$ be the induced Hermitian metric 
         on $L_1$ around $\{0\}$.
         For $v\in G_mV$, $\phi_v$ is in $L_1(m\{0\})$,
         which implies 
         \begin{align*}
           ||\phi_v(t)||_{L_1}\leq C|t|^{-m} \ \ \ (t\in \C^*)
         \end{align*}
         for some positive constant $C$.
         This shows the conclusion: $L_1=L_2$.
        \end{proof}
       \end{lemma}
  \subsection{Definition of rescaling structure}
  \label{def and eg of scaling}
      Let $\C_\rmod$, $\C_\para$ be complex planes with coordinate 
      $\rmod$ and $\para$ respectively.
      Put $\P_\rmod^1:=\C_\rmod\cup \{\infty\}$ and 
      $S:=\P^1_\rmod \times \C_\para$.
      Let 
      $\m: \C^*_\scale \times S\to S$
      be the action of $\C^*_\scale$ defined by 
      $\m(\scale,\rmod, \para):=(\scale \rmod, \scale \para)$.
      For a meromorphic function $h$ on a variety, 
      $(h)_0$ and $(h)_\infty$  denote 
      the zero divisor of $h$ and the pole divisor of $h$, respectively.
      The supports of these divisors are denoted by $|(h)_0|$ and $|(h)_\infty|$,
      respectively.
      Let $p_2:\C^*_\scale \times S\to S$ be the projection.
      We define the notion of rescaling structure as follows. 
      \begin{definition}\label{definition of rescaling structure}
      A rescaling structure is a triple $(\mathcal{H},\nabla, \chi)$ of 
      a $\Z$-graded locally free $\mathcal{O}_S(*(\rmod)_\infty)$-module 
      $\mathcal{H}$, 
       a grade-preserving meromorphic flat connection 
       $$\nabla: \mathcal{H}\to \mathcal{H}\otimes 
       \Omega_S^1\left(*\big{(}|(\rmod)_\infty|\cup |(\rmod \para)_0|\big{)}\right),$$
       and an grade-preserving isomorphism 
       $\chi: p_2^*\mathcal{H}\simeqto \m^*\mathcal{H}$
       with the following properties:
      \begin{enumerate}
       \item We have 
                $\nabla_{\rmod\para\del_\para}(\mathcal{H})\subset \mathcal{H}$
                and $\nabla_{\rmod^2\del_\rmod} \mathcal{H}\subset \mathcal{H}$.
       \item On $\C^*_\rmod \times \C^*_\para$, 
                $\chi$ is flat with respect to $p_2^*\nabla$ and $\sigma^*\nabla$.
       \item The isomorphism $\chi$ satisfies the cocycle condition. 
                In other words,
                we have
                $$(\mathrm{m}\times \id_S)^*\chi
                =(\id_{\C^*_\scale}\times\m)^*\chi\circ p_{23}^*\chi,$$
                where $\mathrm{m}:\C^*_\scale\times \C^*_\scale\to\C^*_\scale$
                denotes the multiplication and 
                $p_{23}:\C^*_\scale\times \C^*_\scale\times S\to \C^*_\scale\times S$
                denotes the projection given by 
                $p_{23}(\scale_1,\scale_2,(\rmod,\para))
                =(\scale_2,(\rmod,\para))$.
      \end{enumerate}  
      We often omit $\nabla$ and $\chi$ if there is no confusion.
      The $k$-th graded piece of $\mathcal{H}$ is denoted by $\mathcal{H}^k$.
      We assume $\sum_k \mathrm{rank}\ \mathcal{H}^k<\infty$ in this paper.
      \end{definition}
      We note that we introduce the notion of rescaling structure only for convenience
      for the later use. 
      Similar structures have been studied in
      \cite{hernil}, \cite{kkp14}, \cite{moctwi}, \cite{sabirr, sabont}, for example.     
      Operations acting on $\mathcal{H}$ is often assumed to preserve the grading 
      without a mention. If $\mathcal{H}$ and $\mathcal{H}'$ are rescaling structures, 
   we can naturally define the tensor product $\mathcal{H}\otimes\mathcal{H}'$
   which is also a rescaling structure. 
   The dual $\mathcal{H}^\vee$ can also be defined canonically.
    \begin{example}\label{Tate twist}
  Set $\T :=\mathcal{O}_S(*(\rmod)_\infty)v$ where $v$ is a global section, and 
   $\deg v=2$. The connection $\nabla$ is defined by 
   $\nabla v:= -v \rmod^{-1}d\rmod$.
   The isomorphism
   $\chi: p_2^*\T\simeqto \sigma^*\T$ is given by 
   $\chi (p^*_2v):=\scale \sigma^*v$.
   Then the tuple $\T(-1):=(\T,\nabla,\chi)$ is a rescaling structure. 
   We define 
   \begin{align*}
    \T(-k):=\begin{cases}
                  \T(-1)^{\otimes k} & \text{ if } k\in \Z_{\geq0} \\
                  (\T(-1)^\vee)^{\otimes -k} &\text{ if } k\in \Z_{< 0}.
                  \end{cases}
   \end{align*}
   For a rescaling structure $\mathcal{H}$, 
   we define $\mathcal{H}(k):=\mathcal{H}\otimes\T(k)$.
   \end{example}
%
    %
    \subsection{Hodge numbers and Hodge-Tate condition for rescaling structures}\label{reshod}   
        \subsubsection{Hodge filtrations for rescaling structures}\label{section hodge}
           Let us consider the restriction 
           $\mathcal{H}_{|\para=0}:=\mathcal{H}/\para\mathcal{H}$. 
           It admits $\C^*_\scale$ action, and hence we can 
           apply the correspondence of \S\ref{action}
           to get the filtration 
           $F_\bullet V$ on $V:=\mathcal{H}\mid_{\rmod=1,\para=0}$
           corresponding to the lattice at $\rmod=0$. 
           \begin{definition} 
           Let $(\mathcal{H},\nabla,\chi)$ be a rescaling structure.
           Then we define 
           \begin{align*} 
             f^{p, q}(\mathcal{H}):=\dim\gr^{F}_{-p}V^{p+q},
           \end{align*}
           where $V^k$ is the $k$-th graded part of $V$.
           \end{definition}
       \subsubsection{Weight filtrations for nilpotent rescaling structures}
       \label{weight filtration}
       We consider the following condition on rescaling structures.
          \begin{definition}
          A rescaling structure $(\mathcal{H},\nabla,\chi)$ is called nilpotent if
          the residue endomorphism $\Res_{\{\para=0\}} \nabla$ on 
          $\mathcal{H}_{|\para=0}$ is nilpotent.
          \end{definition}
          By definition, we have the following: 
          \begin{lemma}\label{lemlat}
          $\mathcal{H}(*(\rmod)_0)$ is
          the Deligne lattice of the meromorphic connection $\mathcal{H}(*(\rmod\para)_0)$
          along the divisor $|(\para)_0|$.\qed
          \end{lemma}
          We have a nilpotent endomorphism 
          $N:=(\Res_{\{\para=0\}} \nabla)_{|\rmod=1}$ on $V$, 
          where $V$ is the fiber of $\mathcal{H}$ at $(\rmod,\para)=(1,0)$.
          Let $V^k$ be the fiber of $\mathcal{H}^k$ at $(\rmod,\para)=(1,0)$.
          The graded piece of $N$ on $V^k$ is denoted by $N_k$.
          Let ${}^kW$ denote the weight filtration of $N_k$ centered at $k$, 
          i.e.,  ${}^kW$ is the unique filtration on $V^k$ with the following properties: 
          \begin{align}\label{dwf}
           &N_k \left({}^kW_i\right) \subset {}^kW_{i-2}  & \text{for all } i\in \Z, \\ \label{dwf2}
           &N_k^{j}:\gr^{{}^kW}_{k+j}V^k \simeqto
            \gr^{^{k}W}_{k-j}V^k & \text{for all } j \in \Z.
          \end{align}
          The induced filtration on $V$ is simply denoted by $W$.
          
          \begin{definition}
           Let $(\mathcal{H},\nabla,\chi)$ be a nilpotent rescaling structure.
           We define 
           \begin{align*}
           h^{p, q}(\mathcal{H}):=\dim\gr^W_{2p}V^{p+q}.  
           \end{align*}
          \end{definition}
       \subsubsection{Hodge-Tate condition}\label{SHT}
       In \cite{delloc}, a mixed ($\Q$-)Hodge structure $(V_\Q,F,W)$ is called 
       Hodge-Tate
       if the Hodge filtration $F$ on $V:=V_\Q\otimes_\Q \C$ and 
       the weight filtration $W$ satisfy the following:
       \begin{align}
             &W_{2i+1}=W_{2i} &\text{ for all }i\in \Z,\label{Wei}\\ 
             &F_{-j}\oplus W_{2j+2}\simeqto V &\text{ for all } j\in \Z.\label{fw}
       \end{align}
       We use the same notation in this paper.
       Imitating this notion, we define the following:
           \begin{definition}\label{def ht}
            Let $(\mathcal{H},\nabla,\chi)$ be a nilpotent rescaling structure.
            Let $F$ and $W$ be the filtrations on $V:=\mathcal{H}_{|(\rmod,\para)=(1,0)}$ defined in
             \S $\ref{section hodge}$
            and \S $\ref{weight filtration}$.
            Then $(\mathcal{H},\nabla,\chi)$ is said to satisfy the Hodge-Tate condition 
            if $(V,F,W)$ satisfies $(\ref{Wei})$ and $(\ref{fw})$. 
            A rescaling structure is called of Hodge-Tate type 
            if it satisfies the Hodge-Tate condition. 
           \end{definition}
           The following is trivial by definition. 
           \begin{lemma}\label{2.8}
             If a rescaling structure $(\mathcal{H},\nabla,\chi)$ 
             satisfies Hodge-Tate condition, 
             then $f^{p, q}(\mathcal{H})=h^{p, q}(\mathcal{H})$ for all $p, q$.\qed
           \end{lemma}
    \subsection{Hodge-Tate condition implies the speciality}\label{htis}
        Let $H=\bigoplus_k H^k$ be a $\Z$-graded finitely generated locally free $\mathcal{O}_{\P^1_\rmod}(*\infty)$ module
        with a grade-preserving meromorphic  flat connection $\nabla$. We assume that $\nabla$ has 
        singularity at most at $\{\rmod=0\}$ in $\C_\rmod$ and $\nabla_{\rmod^2\del_\rmod}(H)\subset H$.
        We also assume that $\nabla$ is regular singular at infinity. 
        Take the Deligne lattice $U_0H$ at $\lambda=\infty$.
        Let $N$ be the nilpotent part of $\Res_{\{\rmod=\infty\}}\nabla$.
        Define ${}^kW_\bullet (U_0H^k_{|\rmod=\infty})$ as the weight filtration of $N$ centered at $k$.
        It induces a filtration $W_\bullet (U_0H_{|\rmod\neq 0})$ of $\Z$-graded logarithmic 
        subbundles of $U_0H_{|\rmod\neq 0}$.
        \begin{definition}[{\cite[Definition 3.21]{kkp14}}]\label{special}
         Let $H$, $\nabla$, $U_0H$, and  $W_\bullet( U_0H_{|\rmod\neq0})$ be as above. 
         We define a vector bundle $\hat{H}$ on $\P^1_\rmod$ 
         by  
         \begin{align*}
         \hat{H}_{|\rmod\neq 0}:=\Image 
         \left\{ \bigoplus_\ell W_{2\ell}( U_0H)\otimes \mathcal{O}_{\P^1_\rmod}(-\ell\cdot\infty)
         \to U_0H(*\infty)\right\},
         \end{align*}
         and $\hat{H}_{|\rmod\neq \infty}:=H$.  
         We call $\hat{H}$ a skewed canonical extension of $H$.
         The $\Z$-graded flat bundle $(H,\nabla)$ is called special if 
         $\hat{H}$
         is isomorphic to a trivial bundle over ${\P^1_\rmod}$.
        \end{definition}
        \begin{remark}
        Our definition of speciality is slightly different from that of \cite{kkp14}.
        This construction of $\hat{H}$  is the same as in \S $\ref{connection}$
        if we take the filtration $G_\bullet V$ to be $G_\ell:=W_{2\ell}$.
        \end{remark} 
      \begin{proposition}\label{HT special}
       Let  $(\mathcal{H},\nabla, \chi)$ be a rescaling structure of Hodge Tate type.
       Then $H_1:=\mathcal{H}_{|\para=1}$ is special.
      \end{proposition}
      The rest of this section is devoted to prove this proposition. 
      \subsubsection{Regular singularity along $|(\rmod)_\infty|$}
        Let $(\mathcal{H},\nabla,\chi)$ be a rescaling structure.
        Put $S^*:=\C^*_\rmod \times \C^*_\para\subset S$.
        Let $\iota:S^*\into \C^*_\scale\times S$ be the embedding given by 
        $\iota(\rmod,\para):=(\rmod^{-1}\para^{-1},\rmod, \para)$.
        We observe that $\iota_\sigma:=\sigma\circ\iota$ gives $\iota_\sigma(\rmod,\para)=(\para^{-1},\rmod^{-1})$
        and $\iota_p:=p_2\circ \iota$ is the inclusion $S^* \into S$.
        Hence we have the isomorphism
        \begin{align}\label{tau}
         \iota^*\chi:\iota_p^*\mathcal{H}=\mathcal{H}_{|S^*}\simeqto \iota_\sigma^*\mathcal{H}.
        \end{align}
        We also remark that $\iota_\sigma$ extends to the map 
        $S\setminus |(\rmod)_0|\to S$ given by $(\rmod,\para)\mapsto (\para^{-1},\rmod^{-1})$, 
        which is denoted by $\overline{\iota}_\sigma$.
            
        \begin{lemma}\label{reg sing}
         The meromorphic connection $(\mathcal{H}(*(\rmod\para)_0),\nabla)$ is regular singular along $|(\rmod)_\infty|$. 
         \begin{proof}
          The isomorphism (\ref{tau}) gives a
          logarithmic extension $\widetilde{\mathcal{H}}$ of 
          $\mathcal{H}_{|\para\neq 0}$ along $|(\rmod)_\infty|$.
          The pull back $\overline{\iota}_\sigma^* \widetilde{\mathcal{H}}$ is 
          isomorphic to $\mathcal{H}_{|S\setminus |(\rmod\para)_0|}$.
         \end{proof}    
        \end{lemma}
        \subsubsection{Deligne lattice}
         Since $\mathcal{H}(*(\rmod\para)_0)$ is regular singular along 
         $|(\rmod)_\infty|\cup|(\para)_0|\subset S$, 
         we have the Deligne lattice 
         $U_{0}\mathcal{H}$ of $\mathcal{H}(*(\rmod\para)_0)$ along 
         $|(\rmod)_\infty|\cup|(\para)_0|$.
         Assume that $\mathcal{H}$ is nilpotent. 
         Then $U_0\mathcal{H}_{|\para=0}$ is equal to $\mathcal{H}(*(\rmod)_0)_{|\para=0}$
         by Lemma \ref{lemlat}.
         In particular,
         we have $U_0\mathcal{H}_{|(\rmod, \para)=(1,0)}=V$.         
         By (\ref{tau}),
         we have that the residue endomorphism $N:=\Res_{\rmod=\infty}\nabla$ 
         on $U_0\mathcal{H}|_{\rmod=\infty}$
         is nilpotent.
         We have the weight filtration ${}^k W$ on degree $k$ part of 
         $U_0\mathcal{H}|_{\rmod=\infty}$ 
         with respect to $N$ centered at $k$. 
         Let $W$ be the resulting filtration on $U_0\mathcal{H}|_{\rmod= \infty}$.
         Then we have logarithmic 
         $\mathcal{O}_S(*(\rmod)_0)$-submodules
         $W_\bullet (U_0\mathcal{H})$ of $\mathcal{H}(*(\rmod)_0)$  
         which
         coincide with $W_\bullet U_0\mathcal{H}_{|\rmod=\infty}$ on $\rmod=\infty$.        
         Define $\hat{\mathcal{H}}$ by 
         \begin{align}\label{hat H}
          \hat{\mathcal{H}}_{|\rmod\neq 0}
          :=\Image \left\{\bigoplus_\ell W_{2\ell}(U_0\mathcal{H})
                                 \otimes\mathcal{O}_S(-\ell (\rmod)_\infty) 
                                  \to \mathcal{H}(*(\rmod)_\infty)_{|\rmod\neq 0}\right\},
         \end{align}
         and $\hat{\mathcal{H}}|_{\rmod\neq\infty}=\mathcal{H}$.
         It is easy to see that $\hat{\mathcal{H}}_{|\para=1}$ is 
         $\hat{H}_1$.
       \begin{lemma}\label{tau del tau}
       The filtration on $V$ induced by $W_\bullet U_0\mathcal{H}$
       is equal to the weight filtration given in \S $\ref{weight filtration}$.       
       \end{lemma}
       \begin{proof}
       Let $T_1$ be the monodromy around $\{\rmod=\infty\}$ acting on 
       $V':=\mathcal{H}_{|(\rmod,\para)=(1,1)}$.
       Let $T_2$ be the monodromy around $\{\para=0\}$ acting on $V'$.       
       By the $\C^*$-equivariance of $\mathcal{H}$ (or, by (\ref{tau})), 
       $N^{(i)}:=\log T_i$ $(i=1,2)$ coincide with each other (both of them are nilpotent). 
       We have a trivialization 
       $(U_0\mathcal{H},\nabla)\simeq (V'\otimes \mathcal{O}_S(*(\rmod)_0),\nabla')$,
       where $\nabla'=d-N^{(1)}\rmod^{-1} d\rmod+N^{(2)}\para^{-1}d\para$.
       Identify $V$ and $V'$ via this isomorphism.
       Then the filtration induced by $W_\bullet U_0\mathcal{H}$
       corresponds to the filtration induced by $N^{(1)}$, 
       and the filtration given in \S $\ref{weight filtration}$
       corresponds to the filtration induced by $N^{(2)}$.
       Since $N^{(1)}=N^{(2)}$, 
       these filtrations are equal.
       \end{proof}

        \subsubsection{Proof of the Proposition $\ref{HT special}$ }
        Put $\hat{H}_0:=\hat{\mathcal{H}}_{|\para=0}$. 
        By Lemma \ref{tau del tau} and Lemma \ref{relation of construction}, 
        $\hat{H}_{0|\rmod\neq 0}$ is given by construction in \S\ref{action} taking $G_\ell=W_{2\ell}$ ($\ell\in\Z$).
        Then the Hodge-Tate condition implies the triviality of $\hat{H}_0$.    
        By the rigidity of triviality of vector bundles on $\P^1$, 
        there is a open neighborhood $U$ in $\C_\para$ such that 
        the restriction $\hat{\mathcal{H}}_{|\P^1_\rmod\times U}$ is trivial
        along $\P^1_\rmod$. 
        Using the $\C^*_\scale$-action, we can show that $\hat{\mathcal{H}}$ itself is 
        trivial along $\P^1_\rmod$. 
        In particular, $\hat{H}_1$ is trivial.\qed   
        \subsubsection{Relation to M. Saito's criterion}
        The referee of this paper indicated the relation between
        Proposition \ref{HT special}
        and the M. Saito's criterion for Birkhoff's problem.
        To see this, we recall the M. Saito's criterion in a special case. 
        Let $H$, $\nabla$ and $U_0H$ be as in \S \ref{htis}.
        For simplicity, we assume that $\mathrm{Res}_{\lambda=\infty}\nabla$ is nilpotent.
        Let $V_\infty$ denote the fiber of $U_0 H$ at $\lambda=\infty$. 
        Remark that the residue $N$ acts on $V_\infty$.
        We define a filtration $F$ on $V_\infty$ as follows:
        \begin{align*}
         F_kV_\infty:=\mathrm{Im}\left(\Gamma(\P^1_\lambda, U_0H\otimes \mathcal{O}_{\P^1_\lambda}(k\{0\}))
         \to V_\infty\right)
        \end{align*}
        where the map is the restriction.
        \begin{theorem}[{\cite[Lemma 2.8]{saiont}, \cite[IV 5.b]{sabiso}}]\label{saito criterion}
        Assume that we have an increasing filtration 
        $G_\bullet V_\infty$ such that 
        \begin{enumerate}
        \item $G_kV_\infty$ is invariant under the morphism $N$ for each $k\in \Z;$ $N(G_kV_\infty)\subset G_kV_\infty$ and
        \item $G_\bullet V_\infty$ is opposed to $F_\bullet V_\infty; F_{-p}\oplus G_{p+1}=V_\infty$ for each $p\in\Z$.         
        \end{enumerate}       
        Then the extension defined by replacing $W_{2\ell}$ by $G_\ell$ in Definition $\ref{special}$
        is logarithmic at infinity and isomorphic to trivial $\mathcal{O}_{\P^1_\lambda}$-module.\qed
        \end{theorem}
        Let us consider the case $H=H_1=\mathcal{H}_{|\tau=1}$
        where $\mathcal{H}$ is a nilpotent rescaling structure.
        We firstly observe that $\mathcal{H}$ is reconstructed from its restriction 
        ${H}_{1}$ as follows.
        Let $\varpi:\P^1_\lambda\times \C^*_\tau\to \P^1_\lambda$ be the map defined by $\varpi(\lambda,\tau):=(\lambda/\tau)$.
        Let $\iota_\tau:\P^1_\lambda\times \C^*_\tau\to \C^*_\theta\times S$ be the map defined by 
        $\iota_\tau(\lambda,\tau)=(\tau^{-1},\lambda,\tau)$.
        Then by taking the pullback of the morphism $\chi$ by $\iota_\tau$, similarly as $(\ref{tau})$, 
        we can identify $\mathcal{H}_{\P^1_\lambda\times \C^*_\tau}$ with 
        $\varpi^*(\mathcal{H}_{|\tau=1})$.
        Hence, by Lemma $\ref{lemlat}$, we obtain $\mathcal{H}$ by taking the Deligne lattice 
        of $\varpi^*(\mathcal{H}_{|\tau=1})$ along $\{\tau=0\}$.
         
        This identification also gives an isomorphism $\chi_V:V\simeqto V_\infty$
        and we have the following.
        \begin{lemma}
        $\chi_V:(V,F)\to (V_\infty, F)$ is a filtered isomorphism.
        \end{lemma}
        \begin{proof}
        For $\overline{v}\in F_kV_\infty$, take a lift $v\in\Gamma (\P^1_\lambda, U_0H_1(k\{0\}))$.
        Then we have a unique $\C^*_\theta$-equivariant section $\tilde{v}\in \Gamma(S, \mathcal{H}(k(\lambda)_0))$ 
        whose restriction to $\para=1$ is $v$.
        We have that 
        the restriction of $\tilde{v}$ to $\tau=0$ is the $\C^*$-invariant section with 
        $(\tilde{v}_{|\para=0})(1)=\tilde{v}(1,0)=\chi_V^{-1}(\overline{v})$, 
        and $\tilde{v}_{|\para=0}\in \mathcal{H}_{|\tau=0}(k\{0\})$.
        This proves the lemma.
        \end{proof}
        By this lemma, 
        Proposition \ref{HT special} 
        can be seen as a corollary of 
        M. Saito's criterion (Theorem \ref{saito criterion}).
        We also remark that the relation to Hodge-Tate condition is mentioned in \cite[Example 3.4.3]{sabfro}
        for classical Hodge structures.
        \begin{remark}
        From these observations, it seems that the parameter $\para$ plays a minor role.        
        However, this parameter naturally appears in some examples  \cite{ggi}, \cite{kkp14}, \cite{moctwi}, \cite{sabirr}.
        In particular, as we will see in Appendix \ref{app A},  
        the parameter $\tau$ appears as a quantum parameter for Tate twisted quantum $\mathcal{D}$-modules.
        In that case, the nilpotent-ness of the rescaling structure is deduced from the fact that 
        the quantum cup product converges to the classical cup product as the quantum parameter goes to zero.
        \end{remark}
%
%
%
%
%
%
%
%
%
%
%
%
%
%
%
%
%
%
%
%
%
%
%
%
%
%
%
%
  \section{Landau-Ginzburg models} \label{section LG}
  In this section, we consider the following pair $(\amb, f)$, referred as a Landau-Ginzburg model: 
  \begin{itemize}
   \item A smooth projective variety $X$ of dimension $n$ over $\C$.
   \item A flat projective morphism $f:X\to\P^1$ of varieties.
  \end{itemize} 
  We also consider $f$ as a meromorphic function on $\amb$.
  We assume that the pole divisor $(f)_\infty$ of $f$ is reduced. 
  The support $|(f)_\infty|$ is denoted by $D$.  
  We also assume that $D$ is simple normal crossing.
  Put $Y:=\amb\setminus \pole$.
  The restriction of $f$ to $Y$ is denoted by $\w$.
  \begin{remark}
  The terminology ^^ ^^ Landau-Ginzburg model" might be inappropriate for general $(X, f)$.
  We need to impose the condition that there is an isomorphism 
  $ \mathcal{O}\simeqto\Omega_X^n(D); 1\mapsto {\sf{vol}}_X$ 
  in order to regard the tuple $((X, f),D,{\sf{vol}}_X)$ as a tame compactified Landau-Ginzburg model in \cite{kkp14}
  $($see Appendix \ref{app kkp}$)$.
  In this paper, we do not use this condition. 
  However, since the main examples we have in mind are $($tame compactified$)$  Landau-Ginzburg models, 
  we call the pair $(X, f)$ a Landau-Ginzburg model for the sake of convenience. 
  \end{remark}
  \subsection{Rescaling structure for Landau Ginzburg models}
      \subsubsection{The Kontsevich complex}\label{ss kon }  
          Let  
          $ df: \Omega^k_X(\log D)\to {\Omega^{k+1}_X(\log D)(D)}$
          be a morphism induced by the multiplication of $df$.
          The inverse image of ${\Omega^{k+1}_X(\log D)}\subset {\Omega^{k+1}_X(\log D)(D)}$ 
          is denoted by $\Omega_f^k$.
          The multiplication $df$ induces a morphism $df:\Omega_f^k\to\Omega_f^{k+1}$.
          The exterior derivative $d$ induces a morphism $d:\Omega_f^k\to\Omega_f^{k+1}$.

          Let $\pi_S:S\times \amb\to \amb$ be the projection.
          Recall that $S=\P^1_\rmod\times\C_\para$.
          Put $\Omega_{f,\rmod,\para}^k
                  :=\pi_S^{-1}\Omega^k_f\otimes\rmod^{-k}\mathcal{O}_{S \times \amb}(*(\rmod)_\infty)$.
          We have morphisms of sheaves 
          $d+\rmod^{-1}\para df: \Omega_{f,\rmod,\para}^k\to \Omega_{f,\rmod,\para}^{k+1}$
          where $d$ is the relative exterior derivative, i.e., $d=d_{S \times \amb/S}$.
          Since $(d+\rmod^{-1}\para df)^2=0$, we have a complex 
          $(\Omega_{f,\rmod,\para}^\bullet, d+\rmod^{-1}\para df)$.
          \begin{definition}Let $p_S:S \times \amb\to S$ denote the projection.
          For each $k\in \Z$, we put
             \begin{align}\label{Kon}
               \mathcal{H}_f^k:=\R^kp_{S*}\left(\Omega_{f,\rmod,\para}^\bullet, d+\rmod^{-1}\para df\right).
             \end{align}
           We define a $\Z$-graded $\mathcal{O}_S(*(\rmod)_\infty)$-module by
           $\mathcal{H}_f:=\bigoplus_{k\in\Z}\mathcal{H}_f^k$.
          \end{definition}
       \subsubsection{The rescaling structure}
        Let $\m:\C_\scale^*\times S\to S$ denote the action of $\C^*_\scale$
        given in \S \ref{def and eg of scaling}.
        Let $\tilde{\m}:\C^*_\scale\times S \times \amb\to S \times \amb$ be the
        action 
        induced by $\m$ and trivial $\C^*_\scale$-action on $X$.
        Let $\tilde{p}_2:\C^*_\scale\times S \times \amb\to S \times \amb$
        denote the projection.
        We have the natural isomorphism 
        $\tilde{\chi}_f:\tilde{p}_2^*(\Omega_{f,\rmod,\para}^\bullet, d+\rmod^{-1}\para df )
        \simeqto \tilde{\m}^*(\Omega_{f,\rmod,\para}^\bullet , d+\rmod^{-1}\para df ).$ 
        It induces an isomorphism $\chi_f: p_2^*\mathcal{H}_f\simeqto\m^*\mathcal{H}_f$
        with the cocycle condition (Definition \ref{definition of rescaling structure} (3)).
          \begin{proposition}
           The pair $(\mathcal{H}_f,\chi_f)$ comes equipped with a  
           rescaling structure.
           \begin{proof}
            By the theorem of Esnault-Sabbah-Yu, M. Saito, and M. Kontsevich \cite{esy} (see also \cite{kkp14}, \cite{moctwi}), 
            $\mathcal{H}_f$ 
            is locally free
            over $\mathcal{O}_{S}(*(\rmod)_\infty)$. 
            Moreover,  \cite[Theorem 3.5]{moctwi} (see also its consequences in \cite[\S 3.1.8]{moctwi})
            implies  
            that we have a connection $\nabla$ on each $\mathcal{H}^k_f$
            with the properties in Definition \ref{definition of rescaling structure}.
           \end{proof} 
          \end{proposition}
        \subsubsection{Hodge filtration}
           Since $\mathcal{H}_f$ is a rescaling structure, 
           $V_f:=\mathcal{H}_{f|(\rmod,\para)=(1,0)}$  
           is equipped with a filtration
           $F_\bullet V_f$ (See \S \ref{section hodge}). 
           Note that $V_f\simeq \H^\bullet\big{(}X,(\Omega^\bullet_f, d)\big{)}.$
            
          \begin{lemma}[{\cite{esy},\cite{moctwi}}]\label{lemhod}
           Let $F_\bullet(\Omega_f^\bullet, d)$ be the stupid filtration on $(\Omega_f^\bullet, d)$,
           that is, we put $F_{-p}\Omega^k_f=0$ for $p>k$ and 
           $F_{-p}\Omega^k_f=\Omega_f^k$ for $p\leq k$.
           Then we have the following$:$
           \begin{align}\label{hodge filter}
            F_{-p}V^k_f\simeq\Image \left(\H^k\big{(}X,F_{-p}(\Omega^\bullet_f, d)\big{)}
                                 \to \H^k\big{(}X,(\Omega^\bullet_f, d)\big{)}\right).         
           \end{align}
           \end{lemma}
               \begin{proof}               
                Let $\pi_\rmod: \C_\rmod \times \amb\to \amb$ be the projection.
                Define $\Omega_{f,\rmod}^k:=\pi^*_\rmod\Omega_f^k$.
                Let $p_\rmod: \C_\rmod \times \amb\to \C_\rmod$ denote the projection. 
                By the local freeness, we have an $\C^*_\scale$-equivariant isomorphism
                $$\mathcal{H}^k_{f|\para=0}\simeq \R^kp_{\rmod *}(\Omega_{f,\rmod}^\bullet,\rmod d).$$
                The isomorphism $\chi$ on $\R^kp_{\rmod *}(\Omega_{f,\rmod}^\bullet,\rmod d)$
                is induced by 
                $\scale^p\tilde{\chi}_{f|\para=0}: 
                (\tilde{p}_2^*\Omega_{f,\rmod,\para}^p)_{|\para=0}
                \simeqto (\tilde{\m}^*\Omega_{f,\rmod,\para}^p)_{|\para=0}
                $.                
                
                Let $\mathscr{A}^{p, q}_\amb$ denote the sheaf of $(p, q)$-forms on $\amb$.
                Let $\del:\mathscr{A}_\amb^{p, q}\to \mathscr{A}_\amb^{p+1, q}$ 
                and $\overline{\del}:\mathscr{A}_\amb^{p, q}\to \mathscr{A}_\amb^{p, q+1}$
                be the Dolbeault operators.
                Set $\mathscr{A}_f^{p, q}:=\Omega^p_f\otimes_{\mathcal{O}_\amb}\mathscr{A}_\amb^{0, q}$.
                Put  
                $\mathscr{A}^{p, q}_{f,\rmod}
                :=\mathcal{O}_{\C_\rmod\times X}
                \otimes_{\pi_\rmod^{-1}\mathcal{O}_\amb} 
                \pi^{-1}_\rmod
                \mathscr{A}_f^{p, q}$.
                The operators on $\mathscr{A}^{p, q}_{f,\rmod}$
                induced by $\del$ and $\overline{\del}$ are denoted by the same notation.
                Then we obtain the double complex 
                $(\mathscr{A}_{f,\rmod}^{\bullet,\bullet}, \rmod \del, \overline{\del})$.
                Let $(\mathscr{A}_{f,\rmod}^\bullet, \rmod\del+\overline{\del})$ be the total complex.
                Remark that $\mathscr{A}^k_{f,\rmod}=\bigoplus_{p+q=k}\mathscr{A}^{p, q}_{f,\rmod}$.
                We obtain a $\C_\scale^*$-equivariant quasi-isomorphism
                $$(\Omega_{f,\rmod}^\bullet,\rmod d)\simeqto (\mathscr{A}_{f,\rmod}^\bullet, \rmod\del+\overline{\del}),$$
                where the isomorphism on $\mathscr{A}^{p, q}_{f,\rmod}$ is induced by $\scale^p\tilde{\chi}_{f|\para=0}$.
                Hence, we have $\C^*_\scale$-equivariant isomorphism:
                \begin{align}\label{dolres}
                \R^kp_{\rmod *}(\Omega_{f,\rmod}^\bullet,\rmod d)\simeq
               \mathscr{H}^k p_{\rmod *}(\mathscr{A}_{f,\rmod}^\bullet, \rmod\del+\overline{\del}).
               \end{align}           
          
               Applying Lemma \ref{doucpx}
               for $C^{p, q}:=\Gamma(\amb,\mathscr{A}^{p, q}_f), \delta_1:=\del,$ and $\delta_2:=\overline{\del}$,
               the fiber of the cohomology sheaf
               $ \mathscr{H}^k p_{\rmod *}(\mathscr{A}_{f,\rmod}^\bullet, \rmod\del+\overline{\del})$
               at $\rmod=1$
               has the filtration $G_\bullet$
               as in Lemma \ref{doucpx} 
               (The fact that we can apply the lemma is due to \cite[Theorem 1.3.2]{esy}).
               Since the restriction of $(\ref{dolres})$ to $\rmod=1$ 
               gives a filtered isomorphism 
               $(V_f^k,F)\simeq (H^k(C^\bullet,\delta), G)$, 
               we obtain the conclusion.
               \end{proof}
           By this lemma, we have $\gr^F_{-p}V_f^k=H^{k-p}(X,\Omega_f^p)$.
           Define $ f^{p, q}(Y, \w):=\dim H^q(X,\Omega^p_f)$.
           Then we have $ f^{p, q}(Y, \w)=f^{p, q}(\mathcal{H}_f)$.
           In the rest of \S \ref{section LG}, 
           we investigate $h^{p, q}(\mathcal{H}_f)$, 
           or the weight filtration of the rescaling structure.
           \subsection{Meromorphic connections for Landau-Ginzburg models}           
           We set  ${\amb}^{(1)}:=\C_\para\times \amb$. We also set ${\pole}^{(1)}:=\C_\para\times\pole$.
           Let $p_\para:{\amb}^{(1)}\to \C_\para$ and 
           $\pi_\para:\amb^{(1)}\to X$ denote the projections.
           We shall review some results on a meromorphic flat bundle $\L:=\mathcal{O}(*{\pole}^{(1)})v$
           with $\nabla v=d(\para f)v$ in \cite{moctwi}, where $v$ denotes a global frame.
           We have
           $$\L \simeq \left(\mathcal{O}_{\amb^{(1)}}(*{\pole^{(1)}}),
           d+d(\para f)\right); v\mapsto 1.$$      
           Remark that, in our case, some of the results in \cite{moctwi} are simplified
           since we assume that $(f)_\infty$ is reduced and the horizontal divisor
           (denoted by $H$ in \cite{moctwi}) is empty.
          \subsubsection{V-filtration along $\para$}\label{V-filter L}
              Regard $\pi_\para^*\dmod_{\amb}$ as a sheaf of subalgebra in $\mathcal{D}_{\amb^{(1)}}$.
              Let ${}^\para\! V_0\dmod_{{\amb^{(1)}}}$ denote the sheaf of subalgebra generated by 
              $\pi_\para^*\dmod_{\amb}$
              and $\para\del_\para$.
              For $\alpha=0,1$, we set 
              $$ U_\alpha \L:=\pi^*_\para\dmod_X\cdot\mathcal{O}_{{\amb^{(1)}}}
              \left((\alpha+1){\pole}^{(1)}\right)v \subset \L.$$ 
              For $\alpha\in \Z_{<0}$, we set $U_\alpha \L:=\para^{-\alpha}U_{0}\L$. 
              For $\alpha\in\Z_{>0}$, we set $U_\alpha \L:=\sum_{p+q\leq \alpha}\del_\para^pU_q\L$.
              Then we have the following:
                  \begin{proposition}[{\cite[Proposition 2.3]{moctwi}}]
                     $U_\bullet \L$ is a V-filtration on $\L$ along $\para$ indexed by integers 
                     with the standard order $($up to shift of degree by $1)$. 
                     More precisely,  we have the following$:$ 
                     \begin{itemize}
                      \item $U_\alpha\L$ are coherent ${}^\para\! V_0\dmod_{\amb^{(1)}}$-modules  
                               such that $\bigcup_\alpha U_\alpha \L=\L$.
                      \item We have $\para U_\alpha \L\subset U_{\alpha-1}\L$ and 
                               $\del_\para U_\alpha\L\subset U_{\alpha +1}\L$.
                      \item Define 
                               $\gr^U_\alpha\L:=U_\alpha\L/U_{\alpha-1}\L$. 
                               Then $\para\del_\para+\alpha$ is nilpotent on $\gr^U_\alpha\L$.\qed
                     \end{itemize}                  
                     \end{proposition}
          \subsubsection{Relative de Rham complexes}\label{qisL}
          We set $\Omega^k_{f,\para}:=\pi_\para^*\Omega_f^k$.
          We obtain a complex $(\Omega^\bullet_{f,\para},d+\para df)$
          where $d=d_{{\amb}^{(1)}/\C_\para}$
          is the relative exterior derivative.
          We have the following:
              \begin{proposition}[{\cite{moctwi}}]
               We have a quasi-isomorphism of complexes
               \begin{align}\label{gi dokei L}
                  (\Omega^\bullet_{f,\para},d+\para df)\simeqto U_0\L\otimes \Omega^\bullet_{{\amb^{(1)}}/\C_\para}.
              \end{align}  
                  \begin{proof}
                   Combine \cite[Proposition 2.21]{moctwi} and \cite[Proposition 2.22]{moctwi} in the case $\alpha=0$. 
                  \end{proof}              
              \end{proposition}
          As a consequence, we have the following (see also the proof of  \cite[Corollary 2.23]{moctwi}):
              \begin{corollary}\label{712}
              We have the following isomorphism of logarithmic connections: 
               $$ \mathcal{H}^k_{f|\rmod=1}\simeqto
               \R^kp_{\para *}\left(U_0 \L\otimes \Omega^\bullet_{{\amb}^{(1)}/\C_\para}\right).$$
               We also have a quasi-isomorphism of complexes: 
               \begin{align*}
                (\Omega^\bullet_f, d)\left(=(\Omega_{f,\para}^\bullet, d+\para df)_{|\para=0}\right)
                \simeqto \gr^U_0\L\otimes \Omega^\bullet_{\amb},
               \end{align*}
              which induces 
              $V_f^k\simeqto \H^k \left(X,\left(\gr^U_0\L\otimes \Omega^\bullet_{\amb}\right)\right)$. 
              The residue endomorphism on $V_f^k$ is 
              identified with the nilpotent endomorphism on 
              $\H^k\left(\amb,\left(\gr^U_0\L\otimes \Omega^\bullet_{\amb}\right)\right)$
              associated with $\varphi_0$ on $\gr^U_0\L\otimes \Omega_X^\bullet$, 
              where $\varphi_0$ denotes the endomorphism induced by $\para\del_\para$.
              \qed
              \end{corollary}  
     \subsubsection{Residue endomorphisms}\label{resM}
     We shall give another description of $\varphi_0$ in Corollary \ref{712}.
      Consider 
      $\Omega^k_{\amb^{(1)}}(\log \para)_0
      :=\mathcal{O}_{\{0\}\times\amb}\otimes \Omega^k_{\amb^{(1)}}(\log \para)$
      as an $\mathcal{O}_\amb$-module. 
      It naturally decomposes to   
      the following module: 
      $$\Omega^k_{\amb}\oplus [\para^{-1}d\para]\cdot \Omega^{k-1}_{\amb},$$  
      where $[\para^{-1} d\para]$ denotes the section induced by $\para^{-1}d\para$.
      
      Since $U_\alpha\L$ is a $V_0\dmod_{\amb^{(1)}}$-module, 
      we have 
      $\nabla: U_\alpha\L\otimes\Omega^k_{\amb^{(1)}}(\log \para)
      \to U_\alpha\L\otimes \Omega^{k+1}_{\amb^{(1)}}(\log \para).$
      This induces 
      \begin{align*}
      \nabla':\gr^U_0\L\otimes \Omega^k_{\amb^{(1)}}(\log \para)_0
      \to \gr^U_0\L\otimes \Omega^{k+1}_{\amb^{(1)}}(\log \para)_0.
      \end{align*}
      The morphisms 
      \begin{align*}
      &\nabla'_0:\gr^U_0\L\otimes \Omega^k_{\amb}
      \to \gr^U_0\L\otimes \Omega^{k+1}_{\amb}, \text{   and   }\\
      &\nabla'_0:\gr^U_0\L\otimes[\para^{-1}d\para]\cdot \Omega^{k-1}_{\amb}
      \to \gr^U_0\L\otimes [\para^{-1}d\para]\cdot \Omega^{k}_{\amb}
      \end{align*}
      induced by $\nabla'$ 
      are 
      the same as the flat connection $\nabla_0$
      given by the $\dmod_\amb$-module structure of $\gr^U_0\L$.
      The morphism $\gr^U_0\L\otimes \Omega^k_\amb\to \gr^U_0\L\otimes[\para^{-1}d\para]\cdot\Omega^k_\amb$
      induced by $\nabla'$
      is given by 
      $m\mapsto [\para^{-1}d\para ]\varphi_0(m)$.
      
      We have the following exact sequence of complexes: 
      \begin{align}\label{GRU}
       0\longrightarrow 
         \gr^U_0\L\otimes\left([\para^{-1}d\para]\cdot\Omega^\bullet_\amb[-1]\right)
         \longrightarrow
         \gr^U_0\L\otimes\Omega^\bullet_{\amb^{(1)}}(\log \para)_0        
        \xrightarrow{h}
        \gr^U_0\L\otimes\Omega^\bullet_\amb
        \longrightarrow 
        0
      \end{align}
      From this exact sequence, 
      we obtain a morphism
      $$\varphi_1:\gr^U_0\L\otimes\Omega^\bullet_\amb
      \to 
      \gr^U_0\L \otimes
      \left(
      [\para^{-1}d\para]\cdot \Omega^\bullet_\amb
      \right)
      \simeq 
      \gr^U_0\L\otimes\Omega^\bullet_\amb
      $$
      in the derived category $D^b(\C_X)$ of $\C_X$-modules.
      \begin{lemma}\label{varphi}
       $\varphi_0=\varphi_1.$
      \end{lemma}
      \begin{proof}
      Let $\mathrm{C}^\bullet(h)$ be the mapping cone of $h$ in  (\ref{GRU}),
      that is, 
      \begin{align*}
      \mathrm{C}^k(h)
      :=\gr^U_0\L \otimes \Omega^{k+1}_{\amb^{(1)}}(\log \para)_0
      \oplus\gr^U_0\L\otimes \Omega_\amb^k, \ \ \ 
      &d_{\mathrm{C}^\bullet(h)}(a, b):= (-\nabla' a, ha+\nabla_0 b), 
      \end{align*}
      where $a\in \gr^U_0\L \otimes \Omega^{k+1}_{\amb^{(1)}}(\log \para)_0$
      and $b\in\gr^U_0\L\otimes \Omega_\amb^k$.
      Then the morphism 
      $$\gr^U_0\L\otimes\Omega_\amb^k\ni \omega\mapsto
      [\para^{-1}d\para]\cdot\omega
      \in \gr^U_0\L \otimes \Omega^{k+1}_{\amb^{(1)}}(\log \para)_0$$
      induces a quasi-isomorphism 
      $\iota_0: \gr^U_0\L\otimes\Omega_\amb^\bullet\to \mathrm{C}^\bullet(h).$
     The morphism $\varphi_1$ is induced by 
      a natural morphism $\iota_1: \gr^U_0\L\otimes\Omega_\amb^\bullet\to \mathrm{C}^\bullet(h).$

      Using the identification 
      $\Omega^k_{\amb^{(1)}}(\log \para)_0=\Omega^k_\amb\oplus [\para^{-1}d\para]\cdot\Omega^{k-1}_\amb$, 
      we obtain a morphism $\Omega_\amb^k\to\Omega^k_{\amb^{(1)}}(\log \para)_0$ of $\mathcal{O}_X$-modules.
      This morphism induces $\Psi: \gr^U_0\L\otimes\Omega_\amb^k\to \mathrm{C}^{k-1}(h).$
      For a section $\omega\in\gr^U_0\L\otimes\Omega_\amb^k$, 
      we have 
      $\Psi\circ d_{\gr^U_0\L\otimes \Omega_\amb^\bullet}(\omega)=\Psi(\nabla_0(\omega))=\nabla_0'(\Psi(\omega))$.
      We also have
      \begin{align*}
      d_{\mathrm{C}^\bullet(h)}\circ\Psi(\omega)
        &=(-\nabla'(\Psi(\omega)), h\circ\Psi(\omega)) \\
        &=-\nabla'_0(\Psi(\omega))-[\para^{-1}d\para]\cdot\varphi_0(\omega)+\iota_1(\omega)\\
        &=-\nabla'_0(\Psi(\omega))-\iota_0\varphi_0(\omega)+ \iota_1(\omega).
      \end{align*}
      Hence we obtain $(d\circ\Psi+\Psi\circ d)(\omega)=\iota_1(\omega)-\iota_0\varphi(\omega)$,
      which implies $\varphi_1=\varphi_0$.
      \end{proof}
      
   \subsection{Relative cohomology groups for Landau-Ginzburg models}\label{sub section Beilinson functor}
       Let $t$ denote a coordinate 
       on the target space of $\w:Y\to \C$.
       Put $s:=1/t$ and let $\C_s\subset\P^1$ be the complex plane with coordinate $s$.
       Take a sufficiently small holomorphic disk $\Delta_s\subset \C_s$ centered at infinity 
       so that 
       no critical values of $f$ are contained in $\Delta_s^\times:=\Delta_s\setminus\{ \infty\}$.
     
       Set $\mathfrak{\amb}:=\amb\times \Delta_s,$ $\mathfrak{\pole}:=\pole\times \Delta_s$.  
       Let $\pi_s:\mathfrak{\amb}\to \amb$ and $p_s:\mathfrak{\amb}\to \Delta_s$ be the projections.     
       Put $g:=1/f$.
       Set 
       $\Gamma:=\{(x, s)\in \mathfrak{\amb}\mid g(x)=s\}$.  
       The inclusion $\Gamma \hookrightarrow\mathfrak{\amb}$ is denoted by $i_\Gamma$. 
              The divisor $\mathfrak{\pole}\cup \Gamma$ is  normal crossing. 
       The intersection $\mathfrak{\pole}\cap \Gamma$ is denoted by $\mathfrak{\pole}_\Gamma$.
       \subsubsection{De Rham complexes}   
       For $k\in\Z_{\geq 0}$, we have a natural morphism
       \begin{align*}
         \bm{\phi}^k: \Omega_\mathfrak{\amb}^k\left(\log \mathfrak{D}\cup\{s=0\}\right)
         \longrightarrow
         i_{\Gamma *}\Omega^k_\Gamma(\log \mathfrak{D}_\Gamma).
       \end{align*}
        Let $\bm{E}^k$ be the kernel of $\bm{\phi}^k$.
       This gives a subcomplex $\bm{E}^\bullet$ of
       $\Omega_\mathfrak{\amb}^\bullet\left(\log \mathfrak{D}\cup\{s=0\}\right)$. 
       \begin{lemma}
       For each $k$, we have
        \begin{align}\label{bmE}
        \bm{E}^k=\left(
                        \frac{ds}{s}-\frac{dg}{g}
                         \right)
                         \cdot
                         \pi_s^*\Omega^{k-1}_\amb(\log \pole)
                         \oplus
                         (s-g)\cdot \pi^*_s\Omega^k_\amb(\log \pole).
        \end{align}
        In particular, 
        $\bm{E}^k$ is a locally free $\mathcal{O}_\mathfrak{\amb}$-module.
       \end{lemma}
       \begin{proof}
        It is trivial that the right hand side of (\ref{bmE}) is included in $\bm{E}^k$.
        Let 
        $s^{-1}ds\cdot \omega_1+\omega_2$ be
        a section of $\bm{E}^k$,
        where 
        $\omega_1\in\pi^*_s \Omega^{k-1}_\amb(\log\pole)$
        and $\omega_2\in \pi_s^*\Omega^k_\amb(\log \pole)$.
        Since $(s^{-1} ds-g^{-1}dg)\cdot\omega_1$
        is a (local) section of $\bm{E}^k$, 
        $g^{-1} dg\cdot\omega_1+\omega_2$ is 
        a section of $\bm{E}^k$. 
        We observe that $g^{-1} dg\cdot\omega_1+\omega_2$ is also
        a section of $\pi^*_s\Omega^k_\amb(\log\pole)$.
        Since 
        $\bm{E}^k\cap \pi^*_s\Omega^k_\amb(\log\pole)=(g-s)\pi^*_s\Omega^k_\amb(\log\pole)$, 
        we obtain that $g^{-1} dg\cdot\omega_1+\omega_2$
        is a section of $(s-g)\pi^*_s\Omega^k_\amb(\log\pole)$.
        This implies that 
        $s^{-1}ds\cdot \omega_1+\omega_2$
        is a section of the right hand side of (\ref{bmE}).
       \end{proof}
       \subsubsection{Relative de Rham complex}
       For $k\in \Z_{\geq 0}$, we have a canonical morphism
       \begin{align*}
        \phi^k:\pi_s^*\Omega^k_\amb(\log \pole)\longrightarrow 
        i_{\Gamma *}\Omega_{\Gamma/\disk}^k(\log \mathfrak{\pole}_\Gamma). 
       \end{align*}
       Remark that $\pi_s^*\Omega^k_\amb(\log \pole)$ is given by
       \begin{align*} 
       \Omega_{\Gamma/\disk}^k(\log \mathfrak{\pole}_\Gamma)
       :=\frac{\Omega_{\Gamma}^k(\log \mathfrak{\pole}_\Gamma)}
                {\Omega_{\Gamma}^{k-1}(\log \mathfrak{\pole}_\Gamma)\wedge p_\Gamma^*\Omega_\disk^1(\log s)},
       \end{align*}    
       where $p_\Gamma$ denotes the composition of $i_\Gamma$ and $p_s$.      
       \begin{definition}[\cite{kkp14}]
       The kernel of the morphism $\phi^k$ is denoted by $E^k$.
       The induced subcomplex of $\pi_s^*\Omega^\bullet_\amb(\log \pole)$
       is denoted by $E^\bullet$.
       \end{definition}
            By definition, 
            $E^k_{|\mathfrak{\amb}\setminus \mathfrak{\pole}}
            \simeq \big{(}(s-g)\pi_s^*\Omega^k_\amb(\log \pole)\big{)}_{|\mathfrak{\amb}\setminus \mathfrak{\pole}}$.
       \begin{lemma}[\cite{kkp14}]
       Let $Q$ be a point in $\pole$.
       If we take a sufficiently small neighborhood $U$ of $Q$, 
       we have the following: 
       \begin{align}\label{Ek}
        E^k_{|\mathfrak{U}}=
        \frac{dg}{g}\cdot(\pi_s^*\Omega^{k-1}_\amb(\log \pole))_{|\mathfrak{U}}
        +\big{(}(s-g)\pi_s^*\Omega^{k-1}_\amb(\log \pole)\big{)}_{|\mathfrak{U}},
       \end{align}
       where $\mathfrak{U}:=U\times \disk$.
       Moreover, $E^k$ is a locally free $\mathcal{O}_\mathfrak{\amb}$-module.
       \end{lemma}
       \begin{proof}
       Since the complex $(\Omega^\bullet_\amb(\log \pole), g^{-1}dg)$
       is acyclic near $Q\in D$, we have a decomposition 
       $$\pi_s^*\Omega^\ell_\amb(\log \pole)_{|\mathfrak{U}(=U\times \disk)}=\mathcal{F}^\ell\oplus\mathcal{G}^\ell$$
       such that $g^{-1}dg:\mathcal{F}^{\ell-1}\simeqto \mathcal{G}^\ell$ $(\ell\in\Z_{\geq 0})$
       for a sufficiently small neighborhood $U$ of $Q$ (see the proof of \cite[Lemma 2.29]{moctwi}).
       We have $E^k_{|\mathfrak{U}}\supset \mathcal{G}^k$,
       and $E^k_{|\mathfrak{U}}\cap\mathcal{F}^k=(s-g)\mathcal{F}^k$.      
       The local freeness of $E^k$ and the equation $(\ref{Ek})$ are obvious by this description.
       \end{proof}
       By this lemma, the restriction of $E^\bullet$ to $s=0$
       is identified with $(\Omega_f, d)$.
         Here, we remark that we have 
       $$\Omega^k_{f|U}=g\cdot \Omega_\amb^k(\log\pole)_{|U}+
       \frac{df}{f}\wedge\Omega_\amb^{k-1}(\log \pole)_{|U}$$
       for sufficiently small $U$ 
       (see \cite[(2.3.1)]{kkp14}, \cite[Lemma 2.29]{moctwi} for example).       

       \subsubsection{Gauss-Manin connection}
       We have a canonical epi-morphism
       $\varphi:\bm{E}^k\to E^k$.
       \begin{lemma}
        $\mathrm{Ker }\ \!\varphi =s^{-1}ds\cdot E^{k-1}$.
       \end{lemma}
       \begin{proof}
        $\mathrm{Ker}\ \!\varphi\supset s^{-1}ds\cdot E^{k-1}$ is trivial.
        Let $(s^{-1}ds-g^{-1}dg)\omega_1+(s-g)\omega_2$ be a (local) section of $\mathrm{Ker}\ \!\varphi$,
        where $\omega_1\in\pi_s^* \Omega^{k-1}_\amb(\log \pole)$
        and $\omega_2\in\pi_s^* \Omega^{k}_\amb(\log \pole)$.
        We have 
        $$-g^{-1}dg\omega_1 +(s-g)\omega_2=0.$$
        Hence, we have 
        $$ \omega_1=(s-g)\tau_1+g^{-1}dg\tau_2,\ \ \ \omega_2=g^{-1}dg\tau_1$$
        for some $\tau_1\in \Omega^{k-1}_\amb(\log \pole)$ and 
        $\tau_2\in \Omega^{k-2}_\amb(\log \pole)$.
        We obtain
        \begin{align*}
         &(s^{-1}ds-g^{-1}dg)\omega_1+(s-g)\omega_2\\
         &= (s^{-1}ds-g^{-1}dg)\big{(}(s-g)\tau_1+g^{-1}dg\tau_2\big{)}
              +(s-g)g^{-1}dg\tau_1\\
         &=s^{-1}ds(s-g)\tau_1+{(}s^{-1}ds-g^{-1}dg{)}g^{-1}dg\tau_2\\
         &=s^{-1}ds\big{(}(s-g)\tau_1+g^{-1}dg\tau_2\big{)}.
        \end{align*}
        This implies $\mathrm{Ker}\ \!\varphi\subset s^{-1}ds\cdot E^{k-1}$.
       \end{proof}
       By this lemma, 
       we have the following diagram, whose rows and columns are exact: 
       \begin{align}\label{GME}
       \xymatrix{&0\ar[d] &0\ar[d]&0\ar[d]\\
        0\ar[r]  & s^{-1}ds\cdot E^\bullet[-1]\ar[r]\ar[d]               
        &   \bm{E}^\bullet  \ar[r]^{\ \varphi\ }\ar[d]&   E^\bullet\ar[d]   \ar[r]&    0\\
        0\ar[r] & s^{-1}ds\cdot\pi_s^*\Omega^\bullet_\amb(\log \pole)[-1]\ar[r]\ar[d]^{\phi^\bullet[-1]}
         &   \Omega^\bullet_\mathfrak{\amb}(\log (\mathfrak{\pole}\cup\{s=0\}))\ar[r] \ar[d]^{\bm\phi^\bullet}
         & \pi_s^*\Omega^\bullet_\amb(\log \pole)\ar[r]\ar[d]^{\phi^\bullet}    
         &0          \\
         0\ar[r]&  s^{-1}ds\cdot i_{\Gamma*}\Omega^\bullet_{\Gamma/\disk}(\log \mathfrak{D}_\Gamma)[-1]\ar[r]\ar[d]&
         i_{\Gamma*}\Omega^\bullet_{\Gamma}(\log \mathfrak{D}_\Gamma)\ar[r]\ar[d]&
         i_{\Gamma*}\Omega^\bullet_{\Gamma/\disk}(\log \mathfrak{D}_\Gamma)\ar[r]\ar[d]&0\\
         &0&0&0
           }
       \end{align}
       From this exact sequence, we obtain a morphism 
       $$E^\bullet\longrightarrow s^{-1}ds\cdot E^\bullet $$
       in the derived category $D^b(\C_X)$.
       This gives a logarithmic connection 
       \begin{align}\label{gauman}
       \nabla^\mathrm{GM}:\R^kp_{s*}E^\bullet \longrightarrow \R^kp_{s*}E^\bullet\otimes \Omega^1_\disk(\log s).
       \end{align}
       On $\Delta_s^\times$, the kernel of $\nabla^\mathrm{GM}$
       is the local system of the relative cohomology $H^k(Y,Y_b)$ $(b\in \Delta_s^\times)$
       (\cite{kkp14}).        
       Hence (\ref{gauman}) gives a logarithmic extension 
       of the flat connection 
       associated with 
       the local system of the relative cohomology $H^k(Y,Y_b)$ $(b\in \Delta_s^\times)$.
      \subsubsection{Residue endomorphisms}\label{sub res}
      Put $\bm{E}^\bullet_0:=\bm{E}^\bullet\otimes\mathcal{O}_{\amb\times\{0\}}$.
      The complex $E^\bullet_0$ can naturally be 
      considered as a complex on $\amb$.
      The complex $E^\bullet_0$ is a subcomplex
      of the complex $\Omega_\amb^\bullet(\log \pole)\oplus s^{-1}d s\otimes\Omega^{\bullet-1}_\amb(\log \pole)$.
      On $\pi_s(\Gamma)$, we have 
      $$ \bm{E}^k_0=g\cdot \Omega^k_\amb(\log \pole)\oplus
           \left(\frac{d s}{s}-\frac{dg}{g}\right)\otimes\Omega^{k-1}_\amb(\log\pole).$$
      On $\amb\setminus \pi_s(\Gamma)$, 
      we have 
      $$ \bm{E}^k_0=\Omega^k_\amb(\log \pole)\oplus \frac{ds}{s}\otimes \Omega_\amb^{k-1}(\log \pole).$$
       From the exact sequence (\ref{GME}),
       we have the following exact sequence: 
       \begin{align}\label{GME0}
        0\longrightarrow 
          \frac{ds}{s}\otimes(\Omega^\bullet_f, d)[-1]
          \longrightarrow 
          \bm{E}^\bullet_0
          \longrightarrow
          (\Omega^\bullet_f, d)
          \longrightarrow
          0
       \end{align}
       From this exact sequence, we obtain a morphism
       $$\varphi_2:(\Omega^\bullet_f, d) \longrightarrow \frac{ds}{s}\otimes(\Omega^\bullet_f, d)\simeq (\Omega^\bullet_f, d).$$
       This induces 
       a residue endomorphism 
       $$\Res_{\{s=0\}}(\nabla^\mathrm{GM}):\H^k(X,(\Omega_f, d))\longrightarrow \H^k(X,(\Omega_f, d))$$
       of $\nabla^{\mathrm{GM}}$ along $\{s=0\}$.
     \subsection{Hodge-Tate conditions for Landau-Ginzburg models}\label{HTLG}
       \subsubsection{Comparison of the residue endomorphisms}\label{W}
        We shall compare the residue endomorphisms given in \S \ref{qisL}
        (see also \S \ref{resM})
        and \S \ref{sub res}.
        Put $\Omega_\mathfrak{\amb}^\bullet(\log s)_0
        :=\Omega_\mathfrak{\amb}^\bullet(\log s)
            \otimes \mathcal{O}_{\amb\times\{0\}}$.
        Let $[s^{-1}d s]$ denote the section of 
        $\Omega_\mathfrak{\amb}^\bullet(\log s)_0$
        induced by $s^{-1}d s$.
        The correspondence $[s^{-1}ds]\leftrightarrow [\para^{-1}d\para]$
        gives an isomorphism 
        $\Omega_\mathfrak{\amb}^\bullet(\log s)_0
        \simeq \Omega^\bullet_{\amb^{(1)}}(\log \para)_0$.
        Via this isomorphism, 
        we identify $\Omega_\mathfrak{\amb}^\bullet(\log s)_0$ with $ \Omega^\bullet_{\amb^{(1)}}(\log \para)_0$.
        Similarly, 
        we identify $$\Omega_\mathfrak{\amb}^\bullet\big{(}\log (\mathfrak{\pole}\cup\{s=0\})\big{)}_0
        :=\Omega_\mathfrak{\amb}^\bullet\big{(}\log (\mathfrak{\pole}\cup\{s=0\})\big{)}
            \otimes \mathcal{O}_{\amb\times\{0\}}$$
        with 
        $$\Omega_{\amb^{(1)}}^\bullet\big{(}\log ({\pole^{(1)}}\cup\{\para=0\})\big{)}_0
        :=\Omega_{\amb^{(1)}}^\bullet\big{(}\log ({\pole^{(1)}}\cup\{\para=0\})\big{)}
            \otimes \mathcal{O}_{\{0\}\times\amb}.$$
            
        By the construction of $U_0\L$, 
        we have an inclusion
        $$\Omega_{\amb^{(1)}}^k(\log \para)\otimes \mathcal{O}_{\amb^{(1)}}(\pole^{(1)})\cdot v
        \into \Omega_{\amb^{(1)}}^k(\log \para)  \otimes\L.$$
        We also have another inclusion
        $$\Omega_{\amb^{(1)}}^k\big{(}\log (D^{(1)}\cup\{\para=0\})\big{)}\cdot v \into
        \Omega_{\amb^{(1)}}^k(\log \para)\otimes \mathcal{O}_{\amb^{(1)}}(\pole^{(1)})\cdot v.$$
        Hence we obtain a morphism
        \begin{align*}
         \Omega^k_\mathfrak{\amb}\big{(}\log (\mathfrak{\pole}\cup\{s=0\})\big{)}
         \longrightarrow 
         U_0\L\otimes \Omega^k_{X^{(1)}}(\log \para).
        \end{align*}
        Since the filtration $U_\bullet\L$ is indexed by $\Z$,
        we have a morphism 
        $$\Omega_\mathfrak{\amb}^k\big{(}\log (\mathfrak{\pole}\cup\{s=0\})\big{)}_0
        \to \gr^U_0\L\otimes \Omega^k_{\amb^{(1)}}(\log \para)_0$$
        given by $\eta\mapsto v\otimes\eta$,
        where $v$ denotes the section of $\gr^U_0\L$ induced by the global section $v$ of $\L$.
        By restricting this morphism to $\bm{E}^k_0$, 
        we obtain a morphism
        $$\Phi:\bm{E}^k_0\longrightarrow \gr^U_0\L\otimes \Omega^k_{\amb^{(1)}}(\log \para)_0.$$
        \begin{lemma}\label{Phi}
         $\Phi$ defines a morphism of complexes.
        \end{lemma}
        \begin{proof}
        Firstly, we verify the lemma on $\pi_s(\Gamma)$. 
        Since $f=1/g$, we have 
        $$\nabla'(v)=v\para dg^{-1}+v\para g^{-1}[\para^{-1}d\para]=v\para g^{-1}\big{(}[\para^{-1}d\para]-g^{-1}dg\big{)} $$
        in $\gr^U_0\otimes \Omega^1_{\amb^{(1)}}(\log\para)_0$.
        Hence, we have
        $$\nabla'(v)\cdot\big{(}-g^{-1}dg +[\para^{-1}d \para]\big{)}=0 $$
        in $\gr^U_0\otimes \Omega^2_{\amb^{(1)}}(\log\para)_0$.
        Since $v g^{-1}dg$ and $v\para^{-1}d\para$ are sections of $U_0\L\otimes\Omega^1_{\amb^{(1)}}(\log \para)$, 
        we have
        $$g\nabla'(v)=v\para(-g^{-1}dg+[\para^{-1}d\para])=0 $$
        in $\gr^U_0\otimes \Omega^1_{\amb^{(1)}}(\log\para)_0$.
        
        Let $(g^{-1}dg-[\para^{-1}d\para])\omega_1+g\omega_2$ be a section of $\bm{E}^k_0$, 
        where $\omega_1\in\Omega^{k-1}_\amb(\log \pole) $ and $\omega_2\in \Omega^k_\amb(\log \pole)$
        (see \S \ref{sub res}). We then obtain 
        \begin{align*}
         &\nabla'(v\cdot(g^{-1}dg-[\para^{-1}d\para])\omega_1)\\
         &=\nabla'(v)\cdot(g^{-1}dg-[\para^{-1}d\para])\omega_1
              +v\cdot d((g^{-1}dg-[\para^{-1}d\para]\omega_1))\\
         &=v\cdot d((g^{-1}dg-[\para^{-1}d\para]\omega_1))
        \end{align*}
        and
        \begin{align*}
         \nabla'(v\cdot g\omega_2)
         =\nabla'(v)\cdot g\cdot\omega_2+vd(g\cdot \omega_2)
         =vd(g\cdot \omega_2).        
        \end{align*} 
         Hence we have $\nabla'\circ \Phi=\Phi\circ d$ on $\pi_s(\Gamma)$.
         
         On $\amb\setminus \pi_s(\Gamma)$, 
         $f=1/g$ is a holomorphic function.
         Hence, $\nabla'(v)=v\para df+v\para f\cdot \para^{-1}d\para$
         is a section of $\para U_0\L\otimes\Omega^1_{\amb^{(1)}}(\log \para)$.
         This implies $\nabla'(v)=0$ on $\gr^U_0\L\otimes\Omega^1_{\amb^{(1)}}$.
         Then we can prove $\nabla'\circ \Phi=\Phi\circ d$
         on $\amb\setminus \pi_s(\Gamma)$ similarly.
        \end{proof}
        We then obtain the following.
        \begin{theorem}\label{icchi}
       The nilpotent endomorphism $(\Res_{\{\para=0\}}\nabla)_{|\rmod=1}$ on $V_f^k$ coincide 
       with 
       the residue endomorphism of the Gauss-Manin connection $\nabla^{\mathrm{GM}}$
       for the relative cohomology group.
       \end{theorem}
       \begin{proof}
        By Lemma \ref{Phi},  
        We obtain the following commutative diagram in the abelian category of complexes on $X$: 
        $$ \xymatrix{
        0 \ar[r] & \frac{ds}{s}\cdot (\Omega_f^\bullet, d)[-1] \ar[r]\ar[d]^{\text{qis}} 
                   & \bm{E}^\bullet_0 \ar[r]\ar[d]^{\Phi} 
                   & (\Omega_f^\bullet, d)\ar[r]\ar[d]^{\text{qis}}& 0\\
        0 \ar[r] & \gr^U_0\L\otimes ([\para^{-1}d\para]\cdot\Omega_\amb^\bullet[-1]) \ar[r]
                   & \gr^U_0\L\otimes \Omega_{\amb^{(1)}}^\bullet(\log\para)_0\ar[r]
                   & \gr^U_0\L\otimes \Omega_\amb^\bullet\ar[r]
                   & 0
        }$$
        The rows of this diagram are the exact sequences (\ref{GME0}) and (\ref{GRU}).
        Left and right columns are the quasi-isomorphisms given in Corollary \ref{712}.
        This diagram shows $\varphi_1=\varphi_2$ in the derived category, 
        which implies the Theorem (See Lemma \ref{varphi}).
       \end{proof}

   \subsubsection{Koszul complex}\label{KC}
   Let $W_\bullet\Omega^\ell_\amb(\log \pole)$ 
   be the weight filtration given by 
   \begin{align*}
    W_m\Omega^\ell_\amb(\log \pole)
    :=\begin{cases}
       \Omega^\ell_\amb(\log\pole)  & (m\geq \ell) \\
       \Omega^{\ell-m}_\amb\wedge \Omega^m_\amb(\log \pole)  & (0\leq m< \ell) \\
       0 & (m<0).
       \end{cases}
   \end{align*} 
   Take the irreducible decomposition $\pole=\bigcup_{i\in \Lambda}\pole_i$.
   Fix an order of $\Lambda$. 
   Note that each $\pole_i$ is a smooth hypersurface in $\amb$ by the assumption.
   Put $\pole(0):=\amb$, and $\pole(m):=\bigsqcup_{I\subset \Lambda, |I|=m}(\bigcap_{i\in I}D_i)$
   for $m\in \Z_{>0}$.
   We have the isomorphism of complexes: 
   $\mathrm{R\acute{e}s}_m:\gr^W_m\Omega^\bullet_\amb(\log \pole)\simeqto a_{m*}\Omega^\bullet_{\pole(m)}[-m],$
   where $a_m:\pole(m)\to X$ denotes the morphism induced by inclusions
   (\cite{delIII}, \cite{guisur}, \cite{petmix}).
   
   We recall that the morphism $\mathrm{R\acute{e}s}_m$ is locally described as follows. 
   Let $(U;(z_1,\dots, z_n))$ be a local coordinate system 
   such that $U\cap \pole=\bigcup_{1\leq j\leq k}\{z_j=0\}$.   
   Assume that we have $\{i_1<i_2<\cdots<i_k\}\subset \Lambda$
   such that $D_{i_j}\cap U=\{z_j=0\}$. 
   For $J=(j_1, \dots,j_m)$ with $1\leq j_1<j_2<\cdots <j_m\leq k$, 
   put $D_J:=\{z_{j_1}=\cdots=z_{j_m}=0\}$
   and $(z^{-1}d z)_J:=z_{j_1}^{-1}dz_{j_1}\wedge\cdots z_{j_m}^{-1}dz_{j_m}$.   
   For $\omega\in W_m\Omega^\ell_\amb(\log \pole)$, 
   we have a unique expression 
   $$\omega=(z^{-1}dz)_J\wedge\alpha+\beta,$$
   where $\alpha\in \Omega_X^{\ell-m}$, $\beta\in\Omega^k_\amb(\log \pole)$
   such that $\beta$ does not have the component $(z^{-1}d z)_J$.
   The residue $\mathrm{R\acute{e}s}_J\omega$ is defined by 
   $\mathrm{R\acute{e}s}_J\omega:=\alpha_{|D_J}$, and $\mathrm{R\acute{e}s}_m$ is defined by 
   $$\mathrm{R\acute{e}s}_m(\omega):=\sum_{J\subset \{1,\dots, k\},|J|=m}  \mathrm{R\acute{e}s}_J(\omega).$$

   Let $\mathcal{M}^{\rm gp}_{\amb,\pole}$ be the sheaf of invertible sections of $\mathcal{O}_\amb(*\pole)$.  
   We have the morphism $\mathcal{O}_\amb\to\mathcal{M}^{\rm gp}_{\amb,\pole}$
   given by $h\mapsto \exp(2\pi \i h)$, where $\i:=\sqrt{-1}$.
   We have the following exact sequence of $\Z_\amb$-modules: 
   $$0\longrightarrow \Z_\amb\to \mathcal{O}_\amb\xrightarrow{\exp(2\pi \i-)}\mathcal{M}_{\amb,\pole}^{\rm gp}
   \xrightarrow{v_{D(1)}} a_{1*}\Z_{D(1)}\longrightarrow 0, $$
   where $\Z_\amb$-module structure of $\mathcal{M}_{\amb,\pole}^{\rm gp}$ is given by the multiplication 
   and $v_{D(1)}$ denotes taking the valuation along the divisors.
   The induced morphism $\mathcal{O}_\amb\to \mathcal{M}^{\rm gp}_{\amb, \pole}\otimes_\Z\Q$
   is denoted by ${\bm{\mathrm{e}}}$. We also have the following exact sequence: 
   \begin{align}\label{rme}
    0\longrightarrow \Q_\amb\to \mathcal{O}_\amb\xrightarrow{\ {\bm{\mathrm{e}}}\ }
    \mathcal{M}_{\amb,\pole}^{\rm gp}\otimes_\Z\Q
   \longrightarrow a_{1*}\Q_{D(1)}\longrightarrow 0.
   \end{align}
   We shall consider the following ^^ ^^ Koszul complex'' of ${\bm{\mathrm{e}}}$ (\cite{illcom}, \cite{petmix}):
   \begin{align*}
   K^\ell_m:=
   \mathrm{Sym}_\Q^{m-\ell}(\mathcal{O}_\amb)
   \otimes_\Q\bigwedge^\ell_\Q(\mathcal{M}^{\rm gp}_{\amb,\pole}\otimes_\Z\Q).
   \end{align*}
   We have the natural inclusion 
   $K^\ell_m\into K^\ell_{m+1}$ by
   $h_1\cdots h_{m-\ell}\otimes y\mapsto 1\cdot h_1\cdots h_{m-\ell}\otimes y$
   and the differential $d:K^\ell_m\to K^{\ell+1}_{m}$ by 
   \begin{align*}
   d(h_1\cdots h_{m-\ell}\otimes y):=\sum_{i=1}^{m-\ell}h_1\cdots h_{i-1}\cdot h_{i+1}\cdots h_{m-\ell}\otimes 
  {\bm{\mathrm{e}}}({h_i})\wedge y.
   \end{align*} 
   \begin{lemma}[{\cite[Proposition 4.3.1.6]{illcom}, \cite[Theorem 4.15]{petmix}}]
   \label{ratiso}
   \begin{align*} 
    \mathscr{H}^q(K^\bullet_p)\simeq 
    \begin{cases}
    a_{q*}\Q_{\pole(q)} &\text{ for }q\leq p\\
    0&\text{ for }q>p. \qed
    \end{cases}    
    \end{align*}
    \end{lemma}
    By this lemma, the natural inclusion $K^\bullet_p\into K^\bullet_{p+1}$ is a quasi-isomorphism 
    for $p\geq n=\dim X$.
    We put $K^\bullet_{\infty}:=K^\bullet_{n}$
    and let $W_mK_\infty^k$ be the image of $K_m^k$ to $K_{n}^k$ for $m< n$ and 
    $W_mK_\infty^k:=K_\infty^k$ for $m\geq n$.
    We obtain a filtered complex $(K^\bullet_\infty, W)$.
    \begin{theorem}[{\cite[Theorem 4.15, Corollary 4.16]{petmix}}]\label{QisK}
      The morphism $K^\ell_m\to W_m\Omega^\ell_\amb(\log \pole)$ given by  
    \begin{align}\label{qislog}
    h_1\cdots h_{m-\ell}\otimes y_1\wedge \cdots \wedge y_\ell
    \mapsto \frac{1}{(2\pi \i)^\ell}\left(\prod_{i=1}^{m-\ell}h_i\right)\cdot 
    \frac{dy_1}{y_1}\wedge\cdots \wedge \frac{dy_\ell}{y_\ell}
    \end{align}
    induces a filtered quasi-isomorphism 
       $\alpha:(K^\bullet_\infty, W)\otimes\C \to (\Omega_\amb^\bullet(\log \pole), W)$, 
    or an isomorphism in the derived category of filtered complexes $D^b(F\C_X)$ \cite[\S 7.1]{delIII}.
    \qed 
     \end{theorem}
    \begin{corollary}[{\cite[Proposition-Definition 4.11, Corollary 4.17]{petmix}}]\label{COR}
    Let $F$ be the stupid filtration on $\Omega_\amb^\bullet(\log \pole)$.
    Then the tuple 
     $$\mathcal{H}dg(\amb\log \pole):=\big{(}(K_\infty^\bullet, W),(\Omega^\bullet_\amb(\log \pole),F,W),\alpha\big{)}$$
    is isomorphic to the cohomological mixed $\Q$-Hodge complex 
    $\big{(}(\R\jmath_* \Q_Y,{\tau_{\leq}}),(\Omega^\bullet_\amb(\log \pole),F,W),\alpha'\big{)}$ on $\amb$
    in \cite[(8.1.8)]{delIII}, \cite{delII}.    
    Here, $\jmath: Y\into \amb$ is the inclusion,
    $\tau_\leq$ denotes the filtration by truncation functor 
    and 
    $\alpha': (\R\jmath_*\C_Y,\tau_\leq)\to (\Omega^\bullet_\amb(\log \pole),W)$
    is an isomorphism in $D^+(F\C_X)$.
   \end{corollary}
   \begin{proof}
    By Theorem \ref{QisK},  we have the following commutative diagram:
    \begin{align}\label{diagl}
    \xymatrix{
     (\R\jmath_*\C_Y,\tau_\leq)\ar@{=}[d]\ar[r]^{\sim}& (\R\jmath_*\Omega^\bullet_Y,\tau_\leq)
     &\ar[l]_{\sim}(\Omega_X^\bullet(\log\pole),\tau_{\leq})\ar[r]& 
     (\Omega^\bullet_X(\log \pole),W)\\
     (\R\jmath_*\C_Y,\tau_\leq)\ar[r]^{\sim\ \ \ }
     &\R\jmath_*\jmath^{-1}(K^\bullet_\infty,\tau_{\leq})\otimes\C\ar[u]_{\simeq}^{\R\jmath_*\jmath^{-1}\alpha}
     &\ar[l]_{\ \ \sim}(K^\bullet_\infty,\tau_{\leq})\otimes\C\ar[u]_\simeq^\alpha\ar[r]
     &(K^\bullet_\infty, W)\otimes\C\ar[u]^\alpha_\simeq
     }
    \end{align}
    Here, the arrows $\simeqto$ and $\uparrow\simeq$ denote filtered quasi-isomorphisms. 
    Since the natural morphism 
    $$(\Omega^\bullet_\amb(\log \pole),\tau_{\leq})\longrightarrow (\Omega^\bullet_\amb(\log \pole), W) $$
    is a filtered quasi-isomorphism, 
    the morphism 
    $$(K^\bullet_\infty,\tau_{\leq})\otimes\C \longrightarrow(K^\bullet_\infty, W)\otimes\C$$
    is also a filtered quasi-isomorphism.
   Remark that $\alpha'$ is defined by the first row of (\ref{diagl}),
   and the second row comes from the following sequence:
   \begin{align}\label{isoCMHC}
   \xymatrix{
   (\R\jmath_*\Q_Y,\tau_\leq)\ar[r]^{\sim\ \ \ }
     &\R\jmath_*\jmath^{-1}(K^\bullet_\infty,\tau_{\leq})
     &\ar[l]_{\ \ \sim}(K^\bullet_\infty,\tau_{\leq})\ar[r]^\sim
     &(K^\bullet_\infty, W)
   }
   \end{align}
   It follows that (\ref{isoCMHC}) defines the isomorphism of cohomological mixed Hodge complexes.
   \end{proof}
   The cohomological mixed Hodge complex  $\mathcal{H}dg(\amb\log \pole)$
   gives a mixed $\Q$-Hodge structure on the cohomology groups $H^k(Y,\Q)$, $k\in \Z_{\geq0}$,
   which is denoted by $H^k(Y):=(H^k(Y,\Q),F,W)$.
   \subsubsection{Cohomological mixed Hodge complex}\label{MHC}
   Put $\widetilde{A}^{p, q}:=\Omega^{p+q}_\amb(\log \pole)/W_{q-1}\Omega^{p+q}_\amb(\log \pole)$ 
   and $\widetilde{C}^{p, q}:=(K^{p+q}_\infty/W_{q-1}K^{p+q}_\infty)(q)$, 
   where $p, q\in\Z_{\geq0}$ and $(q)$ denotes the Tate twist. 
   We have the following differentials 
   \begin{align*}
    \delta':\widetilde{A}^{p, q}\to \widetilde{A}^{p+1,q};& [\eta \mod W_{q-1}]\mapsto [d\eta \mod W_{q-1}],\\
    \delta'':\widetilde{A}^{p, q}\to \widetilde{A}^{p, q+1};& [\eta \mod W_{q-1}]\mapsto [g^{-1}dg\wedge \eta\mod W_q],\\
    \delta':\widetilde{C}^{p, q}\to \widetilde{C}^{p+1,q};
              &[x\otimes y \mod W_{q-1}]\otimes(2\pi\i)^{q}
              \mapsto[d(x\otimes y)\mod W_{q-1}]\otimes (2\pi\i)^q,\\
    \delta'':\widetilde{C}^{p, q}\to \widetilde{C}^{p, q+1};
             &[x\otimes y \mod W_{q-1}]\otimes(2\pi\i)^{q}
             \mapsto[ x\otimes g\wedge y\mod W_{q}]\otimes(2\pi\i)^{q+1},
   \end{align*}
   where $\eta\in \Omega^{p+q}_\amb(\log \pole)$, $x\in \mathrm{Sym}^{k}_\Q(\mathcal{O}_\amb)$
   for $k\geq 0$, 
   and $y\in \bigwedge^{p+q}_\Q(\mathcal{M}_{\amb,\pole}^{\rm gp}\otimes_\Z\Q)$.
   The total complex of this double complexes are denoted by $\tot$
   and $\rtot$,
   that is, $\tot^k:=\bigoplus_{p+q=k}\widetilde{A}^{p, q}$ and $\delta:=\delta'+\delta'':\tot^k\to\tot^{k+1}$
   is the differential. $\rtot$ is defined similarly. 
   We also have the filtrations 
   \begin{align*}
   W_r\widetilde{A}^{p, q}
       &:=W_{r+2q}\Omega^{p+q}_\amb(\log \pole)/W_{q-1}\Omega^{p+q}_\amb(\log \pole)
       \subset \widetilde{A}^{p, q},\\
   W_r\widetilde{C}^{p, q}
       &:=(W_{r+2q}K^{p+q}_\infty/W_{q-1}K^{p+q}_\infty)(q)
       \subset \widetilde{C}^{p, q}.
   \end{align*}
   It induces the following filtrations on $\tot$ and $\rtot$:
   $W_r\tot^k:=\bigoplus_{p+q=k}W_r\widetilde{A}^{p, q},$ and
   $W_r\rtot^k:=\bigoplus_{p+q=k}W_r\widetilde{C}^{p, q}.$
   We define the filtration $F$ by 
   $F_\ell\tot^k:=\bigoplus_{p+q=k}\bigoplus_{p\geq -\ell} \widetilde{A}^{p, q}.$
  
   Since $\delta''W_r\subset W_{r-1}$ for the filtrations $W$ on $\tot$ and $\rtot$, 
    we obtain the isomorphisms
    $\gr^W_j s(\widetilde{A}^{\bullet,\bullet})\simeq 
    \bigoplus_{k\geq 0, -j}\gr^W_{j+2k}\Omega^\bullet_\amb(\log \pole)$
    and $\gr^W_j s(\widetilde{C}^{\bullet,\bullet})\simeq 
    \bigoplus_{k\geq 0, -j}\gr^W_{j+2k}K_\infty^\bullet(k)
    $ of complexes.
    Then the following lemma is trivial by Theorem $\ref{QisK}$. 
   \begin{lemma}
   The morphisms 
   $ K^{p+q}_\infty(q)\to \Omega^{p+q}_\amb(\log\pole)$
   given by 
   $$(h_1\cdots h_{n-p-q}\otimes y_1\wedge\cdots\wedge y_{p+q})\otimes (2\pi\i)^q
   \mapsto \frac{1}{(2\pi\i)^p}\left(\prod^{n-p-q}_{i=1}h_i\right)
   \frac{dy_1}{y_1}\wedge\cdots\frac{dy_{p+q}}{y_{p+q}}$$
   induces a filtered quasi-isomorphism 
   $\alpha_1:(\rtot, W)\otimes \C\to (\tot,W)$.\qed
   \end{lemma}

   Put $A^{p, q}:=\widetilde{A}^{p, q+1}$, $C^{p, q}:=\widetilde{C}^{p, q+1}$ for $p, q\in \Z_{\geq0}$. 
   The total complexes are denoted by $s(A^{\bullet,\bullet})$ and $s(C^{\bullet,\bullet})$.
   We note that $s(A^{\bullet, \bullet})$ and $s(C^{\bullet, \bullet})$ are supported on $D$.
   Let $W_r A^{p, q}:=W_{r-1}\widetilde{A}^{p, q+1}$ and $W_rC^{p, q}:=W_{r-1}\widetilde{C}^{p, q+1}$.
   These filtrations induce filtrations on $s(A^{\bullet,\bullet})$ and $s(C^{\bullet,\bullet})$.
   We have a quasi-isomorphism $\alpha_0:(s(C^{\bullet,\bullet}),W)\otimes\C\to (s(A^{\bullet,\bullet}),W)$
   by restricting $\alpha_1$. 
   We also have the filtration $F$ on $s(A^{\bullet,\bullet})$ by  
   $F_\ell s(A^{\bullet,\bullet})^k:=\bigoplus_{p+q=k}\bigoplus_{p\geq -\ell} {A}^{p, q}$.
   \begin{theorem}[{\cite[Theorem 11.22]{petmix}}]
   The tuple 
   \begin{align*}
   \psi^{\rm Hdg}_g:=\big{(}(s(C^{\bullet,\bullet}),W), (s(A^{\bullet,\bullet}),F,W), \alpha_0\big{)}
   \end{align*}
   is a cohomological mixed $\Q$-Hodge complex on $X$, 
   which defines a mixed Hodge structure 
   on the hypercohomology $\H^\bullet(X, \psi_g(\Q_X))$ of the nearby cycle $\psi_g(\Q_X)$.\qed
   \end{theorem} 
   The mixed $\Q$-Hodge structure on the hypercohomology group
   $\H^k(X,\psi_g\Q_X)$ is denoted by $H^k(Y_\infty)$.  
   Define $\vartheta_{\C}: \Omega^p_\amb(\log \pole)\to A^{p, 0}$ 
   by $\vartheta_{\C}(\eta):=(-1)^p[g^{-1}dg\wedge \eta\mod W_0]$. 
   It induces a morphism of complexes $\vartheta_{\C}:\Omega^\bullet_\amb(\log\pole)\to s(A^{\bullet,\bullet})$.
   Define $\vartheta_{\Q}: K^p_\infty\to C^{p,0}$
   by $\vartheta_{\Q}(x\otimes y):=(-1)^p[x\otimes g\wedge y]$.
   It induces a morphism of complexes
   $\vartheta_{\Q}:K^\bullet_\infty\to s(C^{\bullet,\bullet})$.
   By the construction, we have $\alpha_0\circ\vartheta_{\Q}=\vartheta_{\C}\circ \alpha$.
   Hence, we obtain a morphism of cohomological mixed $\Q$-Hodge complexes 
    $\vartheta:\mathcal{H}dg(\amb\log \pole) \to \psi_g^{\rm Hdg}$ 
    (See \cite[\S 3.3.4.2]{elzmix} for the definition of morphism of cohomological mixed Hodge complex).     
    We have the mixed cone complex 
    $\mathrm{C}(\vartheta)=\big{(}(\mathrm{C}(\vartheta_\Q),W), (\mathrm{C}(\vartheta_\C),W,F),\alpha_\vartheta)\big{)}$ 
    (\cite[\S 3.3.4.2]{elzmix}).
    We also have the notion of shift (\cite[\S 3.3.3.1]{elzmix}) for cohomological mixed Hodge complexes.
   \begin{proposition}\label{XI}
   The tuple 
   $\Xi^{\mathrm{Hdg}}_g:=
   \big{(} (s(\widetilde{C}^{\bullet,\bullet}),W),(s(\widetilde{A}^{\bullet,\bullet}),W,F),\alpha_1\big{)}$  
   constitutes 
   a cohomological mixed $\Q$-Hodge complex  on $\amb$, which is 
   isomorphic to $\mathrm{C}(\vartheta)[-1]$. 
   \end{proposition}
   \begin{proof}
   The shifted cone $\mathrm{C}(\vartheta_{\Q})[-1]$ of $\vartheta_{\Q}$
   is given by 
   \begin{align*}
   (\mathrm{C}(\vartheta_{\Q})[-1])^k
   &=K^k_\infty\oplus s(C^{\bullet,\bullet})^{k-1}\\
   &=\widetilde{C}^{k,0}\oplus \bigoplus_{p+q=k-1}C^{p, q}=\bigoplus_{p+q=k}\widetilde{C}^{p, q}.
   \end{align*}
   The differential
   $d:\bigoplus_{p+q=k}\widetilde{C}^{p, q}\to \bigoplus_{p+q=k+1}\widetilde{C}^{p, q}$
   of $\mathrm{C}(\vartheta_{\Q})[-1]$
   is given by $d_{|\widetilde{A}^{p, q}}=-\delta$ for $q>0$, 
   and $d_{|\widetilde{A}^{p,0}}=\delta'+(-1)^{p+1}\delta''$.
   The isomorphism $h_\Q:\rtot\to \mathrm{C}^\bullet(\vartheta_{\Q})[-1]$
   is given by $h_{\Q|\widetilde{C}^{p, 0}}:=\id_{\widetilde{C}^{p, 0}}$ and 
   $h_{\Q|\widetilde{C}^{p, q+1}}=(-1)^{p+q}\id_{\widetilde{C}^{p, q+1}}$.
   The weight filtration on $\mathrm{C}(\vartheta_{\Q})[-1]$
   is given by 
   \begin{align*}
   W_\ell(\mathrm{C}^k(\vartheta_\Q)[-1])^k
   &=W_\ell K^{k}_\infty\oplus W_{\ell+1}s(C^{\bullet,\bullet})^{k-1}\\
   &=W_\ell \rtot^k.
   \end{align*}
   This shows the compatibility of the weight filtrations.
   Similar argument can be applied to $\mathrm{C}(\vartheta_{\C})$.
   The compatibility of Hodge filtration $F$ can easily be checked.
   Let $h_\C:\tot\to \mathrm{C}(\vartheta_\C)$ be the isomorphism defined by the same way as $h_\Q$.   
   It can also be checked that $$(\alpha_\vartheta[-1])\circ(h_\Q\otimes \id_\C)=h_\C\circ \alpha_1.$$
   Remark that $\alpha_\vartheta:\mathrm{C}(\vartheta_\Q)\otimes \C\to \mathrm{C}(\vartheta_\C)$
   is defined by $\alpha_\vartheta(x, y):=(\alpha x, \alpha_0 y)$ for $x\in K^{k+1}_\infty\otimes \C$,
   $y\in s(C^{\bullet,\bullet})^k\otimes \C$.
   This proves the proposition.
   \end{proof}
   The mixed Hodge complex 
    $\Xi^{\rm Hdg}_g$
    defines a mixed Hodge structure on $\H^k(X, s(\widetilde{C}^{\bullet, \bullet}))$,
   which we denote by $H^k(Y, Y_\infty)=(H^k(Y,Y_\infty;\Q),F, W)$.
   \begin{corollary}[{\cite{petmix}}]
   We have the following long exact sequence of mixed Hodge structures: 
      \begin{align}\label{long infty}
    \cdots \longrightarrow H^{k-1}(Y_\infty) \longrightarrow 
    H^k(Y,Y_\infty) \longrightarrow
    H^k(Y) \longrightarrow 
    H^k(Y_\infty) \longrightarrow 
    \cdots.
    \end{align}\end{corollary}
    \begin{proof}
    Apply \cite[Theorem 3.22 (2)]{petmix} to the cone $\mathrm{C}(\vartheta)$.
    \end{proof}
    \begin{remark}
    Although we postpone to clarify the precise relation, 
    the notation $\Xi^{\rm Hdg}_g$ comes from the notation for the Beilinson's maximal extension
    $($see \cite[Theorem E.3]{esy}$)$.
    \end{remark}
   \subsubsection{Monodromy weight filtration}
   Let $\nu:\widetilde{A}^{p, q}\to \widetilde{A}^{p-1,q+1}$ be the morphism 
   given by $\nu([\eta \mod W_{q-1}]):=[\eta \mod W_q]$.
   It induces a nilpotent endomorphism on $\tot$,
   which is also denoted by $\nu$.      
   It can easily be observed that $\nu(W_r)\subset W_{r-2}$, 
   and $\nu (F_i)\subset F_{i+1}$.
   We also define $\nu: \widetilde{C}^{p, q}\to\widetilde{C}^{p-1,q+1}(-1)$ 
   similarly: $[x\otimes y\mod W_{q-1}]\otimes (2\pi\i)^{q-1}
   \mapsto [x\otimes y\mod W_{q}]\otimes (2\pi\i)^{q-1}$.
   Hence we have a morphism $\nu: H^k(Y,Y_\infty)\to H^k(Y,Y_\infty)(-1)$ 
   of mixed Hodge structures
   for each $k$.
   The following theorem is proved in \S \ref{phlm}.
   \begin{theorem}\label{thmmon}
   The map $\nu$ induces isomorphisms
   $$\nu^r: \gr^W_{k+r}H^k(Y, Y_\infty)\simeqto \gr^W_{k-r}H^k(Y, Y_\infty)(-r),$$
   i.e., the weight filtration $W$ on $H^k(Y,Y_\infty)$ is
    the monodromy weight filtration of $\nu$ centered at $k$.
   \end{theorem}
   The way to prove this theorem is essentially the same as in \cite[Theorem 5.2]{guisur}.
   We remark that $n$ in \cite{guisur} corresponds to $n-1$ in this paper.%
   \subsubsection{Monodromy weight spectral sequence}
     By Proposition \ref{XI}, we have the following.
    \begin{corollary}\label{speseq}
   The spectral sequence for $(\R\Gamma(\amb, s(\widetilde{C}^{\bullet,\bullet})), W)$ whose $E_1$-term is given by    
   \begin{align*}
   E_1^{-r, q+r}=\H^q(\amb,\gr^W_r s(\widetilde{C}^{\bullet,\bullet})) 
   \end{align*}
   degenerates at $E_2$-term.
   In other words, 
   $\gr^W_{q+r}\H^q(X,\rtot)$ is the cohomology of the complex:
   $$ E^{-r-1,q+r}_1\xrightarrow{d_1}E_1^{-r, q+r}\xrightarrow{d_1}E_1^{-r+1,q+r}.$$
   \end{corollary} 
   \begin{proof}
   Apply (\cite[(8.1.9)]{delIII}) 
   to the cohomological mixed $\Q$-Hodge complex  $\Xi_g^{\rm Hdg}$ on $\amb$.
   \end{proof}
    By Theorem \ref{QisK}, we have a quasi-isomorphism $\gr^W_mK_\infty^\bullet\simeq a_{m*}\Q_{\pole(m)}[-m](-m)$. 
    Recall that  $\gr^W_j s(\widetilde{C}^{\bullet,\bullet})\simeq 
    \bigoplus_{k\geq 0, -j}\gr^W_{j+2k}K_\infty^\bullet(k).$
    Hence,
    \begin{align*}
      E_1^{-r, q+r}
      &=\H^q(\amb,\gr^W_r s(\widetilde{C}^{\bullet,\bullet})) 
      \simeq\bigoplus_{k\geq 0, -r} \H^q(\amb,\gr^W_{r+2k}K_\infty^\bullet(k))\\
      &\simeq \bigoplus_{k\geq 0,-r}H^{q-r-2k}(D(2k+r);\Q)(-r-k).
    \end{align*}
    Following \cite{guisur}, we put $K^{i, j, k}_\Q:=H^{i+j-2k+n}(D(2k-i);\Q)(i-k)$ for $k\geq 0,i$,
    and $K_\Q^{i, j, k}=0$ otherwise. 
    Then we have $E_1^{-r, q+r}\simeq\bigoplus_{k\in \Z} K_\Q^{-r, q-n, k}$.
    We also put $E^{-r, q+r}_{1,\R}:=E_1^{-r, q+r}\otimes\R$, $K^{i, j, k}:=K_\Q^{i, j, k} \otimes \R$,
    and $K^{i, j}:=\bigoplus_{k}K^{i, j, k}$.
    The induced morphism $d_1\otimes \id_\R$ is also denoted by $d_1$.
    \begin{proposition}[cf.{\cite[Lemma (2.7), Proposition (2.9)]{guisur}}]\label{gysin}
    The restriction of $d_1$ to $K^{i, j, k}$
    decomposes to 
    $d_1':K^{i, j, k}\to K^{i+1, j+1, k}$ and 
    $d_1'':K^{i, j, k}\to K^{i+1, j+1,k+1}$. 
    Moreover, $d_1'$ is the alternating sum of the Gysin map $\gamma^{(2k-i)}$ in \cite[(1.3)]{guisur} times $(-1)$, 
    and $d_1''$ is the alternating sum of restriction map $\rho^{(2k-i)}$ in \cite[(1.3)]{guisur}.
    \end{proposition}
    \begin{proof}
    By the definition, $d_1:E_1^{-r, q+r}\to E^{-r+1,q+r}_1$ is induced by the following short exact sequence.
    \begin{align*}
    0\longrightarrow \gr^W_{r-1}\rtot\longrightarrow W_r\rtot/W_{r-2}\rtot\longrightarrow \gr^W_r\rtot\longrightarrow 0.
    \end{align*}
    We shall compute the complex version of $d_1$ using Dolbeault resolution, and then 
    observe the compatibility with the rational structure.
    
    Let $\mathscr{A}^{p, q}_\amb$ be the sheaf of $(p, q)$-forms on $X$.
    Put $\mathscr{A}^{p, q}_{\amb,\pole}:=\Omega_X^p(\log \pole)\otimes_{\mathcal{O}_\amb}\mathscr{A}_X^{0,q}$,
    $\mathscr{A}^k_\amb:=\bigoplus_{p+q=k}\mathscr{A}_\amb^{p, q}$,
    $\mathscr{A}^k_{\amb,\pole}:=\bigoplus_{p+q=k}\mathscr{A}^{p, q}_{\amb,\pole}$, 
    and $W_m\mathscr{A}^k_{\amb,\pole}:=\mathscr{A}_\amb^{k-m}\wedge\mathscr{A}_{\amb,\pole}^{m}$.
    Let $d:=\del+\delbar:\mathscr{A}_\star^k\to \mathscr{A}_\star^{k+1}$
    be the differential ($\star=\amb, \text{ or } \amb,\pole$).
    We have a resolution
    $(\Omega^\bullet_{\amb}(\log \pole), d)\simeq (\mathscr{A}^\bullet_{\amb,\pole},d)$
    compatible with the filtrations.
    Put $\widetilde{\mathscr{A}}^{p, q}:=\mathscr{A}^{p+q}_{\amb,\pole}/W_{q-1}\mathscr{A}_{\amb,\pole}^{p+q}$
    and define $\delta':\widetilde{\mathscr{A}}^{p, q}\to \widetilde{\mathscr{A}}^{p+1, q}$,
    $\delta'':\widetilde{\mathscr{A}}^{p, q}\to \widetilde{\mathscr{A}}^{p, q+1}$
    by $\delta'([\eta\mod W_{q-1}])=[d\eta \mod W_{q-1}]$
    and $\delta''([\eta \mod W_{q-1}])=[g^{-1}dg\wedge \eta \mod W_q]$.
    Denote $s(\widetilde{\mathscr{A}}^{\bullet,\bullet})$ the associated single complex.
    We also define the filtration on $s(\widetilde{\mathscr{A}}^{\bullet,\bullet})$ by 
    $W_r\widetilde{\mathscr{A}}^{p, q}
    =W_{r+2q}\mathscr{A}_{\amb,\pole}^{p+q}/W_{q-1}\mathscr{A}_{\amb,\pole}^{p+q}$.
    We have the quasi-isomorphism $\tot\simeq s(\widetilde{\mathscr{A}}^{\bullet,\bullet})$ compatible with 
    the filtrations.
    
    For $k\geq 0,-r$, take a class 
    $$[x]\in \H^q(X,\gr^W_{r+2k}\Omega_X^\bullet(\log \pole))\subset \H^q(X,\gr^W_rs(\widetilde{A}^{\bullet,\bullet})).$$
    Since we have the isomorphism 
    $\H^q(X,,\gr^W_{r+2k}\Omega_X^\bullet(\log \pole))\simeq H^q(\Gamma(X,\gr^W_{r+2k}\mathscr{A}_{X,D}^\bullet))$, 
    we can take a representative $x\in\Gamma(X,\gr^W_{r+2k}\mathscr{A}^{q}_{X,D})$
    with $0=dx\in \Gamma(X,\gr^W_{r+2k}\mathscr{A}^{q+1}_{X,D})$.
    Take a lift 
    $\widetilde{x}\in 
    \Gamma(X,W_{r+2k}\mathscr{A}^q_{X,D}/W_{k-1}\mathscr{A}^q_{X,D})
    =\Gamma(X,W_r\widetilde{\mathscr{A}}^{q-k, k})$.
    We have $\delta''\widetilde{x}\in\Gamma(X,W_{r-1}\widetilde{\mathscr{A}}^{q-k, k+1}).$
    Since $dx=0,$ we have $\delta'\widetilde{x}\in \Gamma(X,W_{r-1}\widetilde{\mathscr{A}}^{q-k+1,k})$.
    We obtain that 
    $$d_1[x]=[\delta'\widetilde{x}]+[\delta''x] 
    \in\H^{q+1}(X,\gr^W_{r+2k-1}\Omega_X^\bullet(\log D)) \oplus 
    \H^{q+1}(X,\gr^W_{r+2k+1}\Omega_X^\bullet(\log D)).$$
    Defining $d'_1[x]:=[\delta' \widetilde{x}]$, 
    and $d_1''[x]:=[\delta''\widetilde{x}]$, we have the decomposition $d_1=d_1'+d_1''$.
    
    By the construction, 
    $d_1':\H^q(X, \gr^W_{r+2k}\Omega^\bullet_\amb(\log \pole))
    \to \H^{q+1}(X,\gr_{r+2k-1}^W\Omega^\bullet_X(\log \pole))$ is induced 
    by the short exact sequence 
    $$0\longrightarrow  \gr^W_{r+2k-1}\Omega^\bullet_\amb(\log \pole)
    \longrightarrow \frac{W_{r+2k}\Omega^\bullet_\amb(\log \pole)}{W_{r+2k-2}\Omega^\bullet_\amb(\log \pole)}
    \longrightarrow \gr^W_{r+2k}\Omega^\bullet_\amb(\log \pole)
    \longrightarrow 0.$$ 
    The differential 
    $d''_1:\H^q(X,\gr^W_{r+2k}\Omega^\bullet_{\amb}(\log \pole))
    \to \H^{q+1}(X,\gr^W_{r+2k+1}\Omega^\bullet_{\amb}(\log \pole))$
    is induced by 
    $$g^{-1}dg:\gr^W_m\Omega_\amb^p(\log \pole)\to \gr^W_{m+1}\Omega^{p+1}_X(\log D)\ \ \ \ (m\geq0).$$
     In {\cite{guisur}}, it is shown that 
    $\mathrm{R\acute{e}s}_{r+2k-1}\circ d_1'=(-\gamma^{(r+2k)})\circ \mathrm{R\acute{e}s}_{r+2k}$
    and $\mathrm{R\acute{e}s}_{r+2k+1}\circ d_1''=\rho^{(r+2k)}\circ \mathrm{R\acute{e}s}_{r+2k}$ 
    holds,
    where $\gamma^{(m)}:\H^{k-m}(D(m);\C)\to \H^{k-m+2}(D(m-1);\C)$ denotes the (alternating sum of) Gysin map
    and $\rho^{(m)}:\H^k(D(m);\C)\to \H^k(D(m+1);\C)$ denotes (the alternating sum of) restriction \cite[(1.3)]{guisur}. 
    It is also shown that similar commutativity holds 
    for rational cohomology ({\cite[(1.8),(2.9)]{guisur}}).
    Hence, we obtain the conclusion.
    \end{proof}
    The morphism $\nu:\rtot\to\rtot(-1)$ induces
    morphisms
    $\nu:K^{i, j, k}\to K^{i+2, j, k+1}(-1)$, 
    which is identity whenever $k\geq 0, i$.
    Hence, we obtain:     
    \begin{lemma}[{\cite[Lemma (2.7), Proposition (2.9)]{guisur}},{\cite[Proposition 11.34]{petmix}}]\label{nyu}\ 
    
     \begin{enumerate}
     \item  For all $i\geq 0$, $\nu$ induces an isomorphism $\nu^i:K^{-i, j}\simeqto K^{i, j}(-i)$.
     \item  $\mathrm{Ker}(\nu^{i+1})\cap K^{-i, j}=K^{-i, j, 0}$. \qed
     \end{enumerate}
    \end{lemma}
   \subsubsection{Polarized Hodge-Lefschets modules}\label{phlm}
   We shall use the Guill\'en-Navarro Aznar's formulation \cite[\S 4]{guisur} of  
   the result of
   Saito \cite{saimod} and Deligne on 
   the Hodge-Lefschetz modules. 
   Let $L^{\bullet,\bullet}=\bigoplus_{i, j\in \Z}L^{i, j}$ be a 
   bi-graded finite dimensional $\R$-vector space.
   Let $\ell_1, \ell_2$ be endomorphisms on $L$ such that 
   $\ell_1(L^{i, j})\subset L^{i+2, j}, \ell_2(L^{i, j})\subset L^{i, j+2}$, 
   and $[\ell_1,\ell_2]=0$.
   The tuple $(L^{\bullet,\bullet},\ell_1,\ell_2)$ is called Lefschetz module if 
   $\ell^i_1:L^{-i, j}\to L^{i, j}$ are isomorphisms for all $i>0$ 
   and $\ell_2^j:L^{i, -j}\to L^{i, j}$ are isomorphisms for all $j>0$.
   A Lefschetz module $(L^{\bullet,\bullet},\ell_1,\ell_2)$ is called Hodge-Lefschetz module
   if every $L^{i, j}$ has real Hodge structure and $\ell_1,\ell_2$ are 
   morphisms of real Hodge structures of some types 
   (\cite[(1.2)]{grirec}, or \cite[Definition 7.22]{voihod}). 
   
   A polarization $\psi$ of a Hodge-Lefschetz module $(L^{\bullet,\bullet},\ell_1,\ell_2)$
   is a morphism of real Hodge structures
   $\psi:L^{\bullet,\bullet}\otimes L^{\bullet,\bullet}\to \R $
   of certain type with the following properties:
   \begin{enumerate}
   \item[(P1)] $\psi(\ell_ix,  y)+\psi(x,  \ell_i y)=0$
   for $i=1,2$ and
   \item[(P2)]
    $\psi(-,\ell_1^i\ell_2^jC-)$ is 
   symmetric positive definite on 
   $L_0^{-i, -j}:=L^{-i, -j}\cap \mathrm{Ker}(\ell_1^{i+1})\cap \mathrm{Ker}(\ell_2^{j+1})$.
   \end{enumerate}
   Here, $C$ denotes the Weil operator.
   The tuple $(L^{\bullet,\bullet},\ell_1,\ell_2,\psi)$ of Hodge-Lefschetz module and its polarization is 
   called polarized Hodge-Lefschetz module.
   
   A differential $d$ on a polarized Hodge-Lefschetz module $(L^{\bullet,\bullet},\ell_1,\ell_2,\psi)$
   is a morphism of real Hodge structures $d:L^{\bullet,\bullet}\to L^{\bullet,\bullet}$ of certain type 
   such that   
   \begin{enumerate}
   \item[(D1)]\label{d1} 
    $d(L^{i, j})\subset L^{i+1,j+1} $ for $i, j\in \Z$,
   \item[(D2)]\label{d2}
    $d^2=0$, 
   \item[(D3)]\label{d3}
    $[d,\ell_i]=0$ for $i=1,2$, and 
   \item[(D4)]\label{d4}
   $\psi(d x, y)=\psi (x, dy).$
    \end{enumerate}
    The tuple $(L^{\bullet,\bullet},\ell_1,\ell_2,\psi, d)$ is called differential polarized Hodge-Lefschetz module.
    By definition, $\ell_i$ defines an endomorphism on the cohomology group $H^*(L^{\bullet,\bullet}, d)$
    for $i=1,2$, which is denoted by the same notation. 
    We also have a bilinear map on $H^*(L^{\bullet,\bullet}, d)$, which is also denoted by $\psi$.
    \begin{theorem}\cite[Theorem (4.5)]{guisur}\label{surles}
    Let $(L^{\bullet,\bullet},\ell_1,\ell_2,\psi, d)$ be a differential polarized Hodge-Lefschetz module.
    Then $(H^*(L^{\bullet,\bullet}, d),\ell_1, \ell_2,\psi)$ is a polarized Hodge-Lefschetz module.\qed
    \end{theorem}   
    Fix a  K\"ahler form $\kah$ on $\amb$. 
    Let
    $[\kah]\in H^2(\amb;\R)$ be its cohomology class.
    Cup product with the restriction of the class $[\kah]$ to
    $H^2(D(2k+i);\R)$
    defines mappings
    $L:K^{i, j, k}\to K^{i, j+2, k}$ for all $k\geq 0, i$. 
    Define the linear mapping $\psi: K^{\bullet,\bullet}\otimes K^{\bullet,\bullet}\to \R$
    by 
    \begin{align*}
     \psi(x, y):=\begin{cases}
                      \varepsilon(i+j-n)\left({2\pi\i}\right)^{2k+i}\int_{\pole(2k+i)}x\wedge y 
                        & \text{ if } x\in K^{-i, -j, k}, y\in K^{i, j, k+i}\\
                      0& \text{ else},
                      \end{cases}
    \end{align*}
    where $\varepsilon(a):=(-1)^{a(a-1)/2}$.
    
    \begin{theorem}[cf.{\cite[Theorem (5.1)]{guisur}}]\label{dphl}
    The tuple $(K^{\bullet,\bullet}, (2\pi \i)^{}\nu, L, \psi, d_1)$
    is a differential polarized Hodge-Lefschetz module.
    \end{theorem}
    \begin{proof}
    By Lemma \ref{nyu}, $(2\pi \i\nu)^i:K^{-i, j}\simeqto K^{i, j}$ for $i>0$. 
    By the hard Lefschetz theorem, we also have $L^j:K^{i, -j}\simeqto K^{i, j}$ for $j>0$.
    Hence, $(K^{\bullet, \bullet}, (2\pi\i)\nu, L)$ is a Hodge-Lefschetz module.
    Since the trace map and the cup product are the morphisms of Hodge structures, 
    $\psi$ is a morphism of real Hodge structures.
    By some direct computations as in \cite[Proposition 3.5]{guisur}, 
    we have $\psi(x, y)=(-1)^n\psi(y, x)$, $\psi((2\pi\i)\nu x, y)+\psi(x, (2\pi\i)\nu y)=0$, 
    $\psi(L x, y)+\psi(x, L y)=0$. This proves (P2). 
    By Lemma \ref{nyu}, and the last formula in \cite[(1.3)]{guisur}, we also have $\psi(d_1'x, y)=\psi(x, d''_1y)$.  
    It follows that $\psi(d_1x, y)=\psi(x, d_1y)$. This proves (D4).
    (D1), (D2) are trivial by definition. (D3) follows from Proposition \ref{gysin}.
    
    It remains to prove (P1). Put $K^{-i, -j}_0:=K^{-i,-j}\cap \mathrm{Ker}(\nu^{i+1})\cap\mathrm{Ker}(L^{j+1})$.
    By the hard Lefschetz theorem and Lemma \ref{nyu}, 
    $K^{-i, -j}_0$ is the primitive part of $H^{n-i-j}(D(i);\R)(-i)$.
    If we put $Q(x, y):=\psi(x,((2\pi\i)\nu)^iL^jCy)$ for $x, y\in K^{-i, -j}_0$, 
    we have
    \begin{align*}
    Q(x, y)=\varepsilon(i+j-n)\int_{D(i)}\big{(}(2\pi \i)^ix\big{)}\wedge L^j C(2\pi\i)^i y
    \end{align*}
    Note that $\xi:=(2\pi \i)^ix$ and $\eta:=(2\pi \i)^i y$ are the element of the primitive part of $H^{n-i-j}(D(i);\R)$.
    Since $L$ is the Lefschetz operator on $D(i)$, 
    the map $(\xi,\eta)\mapsto \varepsilon(i+j-n)\int_{D(i)}\xi\wedge L^jC\eta$
    is positive definite by the classical Hodge-Riemann bilinear relations.
    This implies (P1).
    \end{proof}
    \begin{proof}[Proof of Theorem $\ref{thmmon}$]
    By Theorem \ref{surles} and Theorem \ref{dphl}, 
    the tuple
    $$(H^*(K^{\bullet,\bullet}, d_1),(2\pi \i)\nu, L, \psi)$$ 
    is a polarized Hodge-Lefschetz module.
    In particular, 
    $(2\pi\i\nu)^i:H^*(K^{\bullet,\bullet}, d_1)^{-i, j}\to H^*(K^{\bullet,\bullet}, d_1)^{i, j}$ are isomorphisms for $i>0$. 
    By Corollary \ref{speseq}, this implies the theorem.
    \end{proof}
 \subsubsection{Main theorem}
    We firstly compare the nilpotent endomorphisms in \S \ref{W}
    with $\nu$ in \S \ref{MHC}.  
    Recall that the stupid filtration on $(\Omega_f^\bullet, d)$ was
    denoted by $F$ in Lemma \ref{lemhod}.
    \begin{proposition}\label{lastprop}
    We have a filtered quasi-isomorphism   
    $ \rho:((\Omega_f^\bullet, d),F)\simeqto (s(\widetilde{A}^{\bullet,\bullet}),F)$, 
    which is compatible with the nilpotent endomorphisms
    $\varphi_2$ and $\nu$. In other words, $\nu\circ\rho=\rho\circ\varphi_2$ in the derived category.
    \end{proposition}
    \begin{proof}
    The morphism $\rho$ is given by the natural inclusion 
    $\Omega^p_f\into \Omega^p_\amb(\log \pole)=\widetilde{A}^{p,0}$.
    It is trivial that $\rho$ is strictly compatible with $F$.
    By (\ref{GME}), 
    we have a short exact sequence 
    $$0\longrightarrow \Omega_f^p\longrightarrow \Omega^p_\amb(\log \pole)
    \longrightarrow \Omega^p_\amb(\log \pole)\otimes\mathcal{O}_D\longrightarrow 0$$      
    By \cite{stelim}, we have an exact sequence 
    $$0\longrightarrow  \Omega^p_\amb(\log \pole)\otimes\mathcal{O}_D
     \xrightarrow{\theta_p} A^{p,0}\xrightarrow{\delta''}\cdots$$
    where $\theta_p(\eta):=(-1)^p[g^{-1}dg\wedge\eta\mod W_0]$.
    Hence, we obtain an exact sequence
    $$0\longrightarrow \Omega_f^p\xrightarrow{\rho} \widetilde{A}^{p,0}
    \xrightarrow{\delta''} \widetilde{A}^{p,1}
    \xrightarrow{\delta''}\cdots.$$
    This implies that $\rho$ is a filtered quasi-isomorphism.
    
    Take the shifted cone $B^\bullet:=\mathrm{C}^\bullet(\nu)[-1]$ of $\nu$.  
    Define $\varrho:\bm{E}_0^k\to B^k=\tot^k\oplus\tot^{k-1}$  as
    the restriction of the following morphism:
    \begin{align*}
    \Omega_{\mathfrak{\amb}}^k(\log (\mathfrak{D}\cup\{s=0\}))_0
    &=\Omega^k_\amb(\log \pole)\oplus s^{-1}ds\Omega^{k-1}_\amb(\log\pole)\ni\\
     \omega_1+s^{-1}ds\omega_2
    &\mapsto \omega_1\oplus \omega_2 \\
    &\in A^{k,0}\oplus A^{k-1,0}
    \subset \tot^k\oplus\tot^{k-1}.
    \end{align*}
    Then, $\varrho$ gives a morphism of complex.
    Indeed, it is trivial on $\amb\setminus \pi_s(\Gamma)$.
    On $\pi_s(\Gamma)$, 
    take a section $g\omega_1+(s^{-1}d s-g^{-1}dg)\omega_2$ of $\bm E_0^k$.
    Remark that $[g\omega_1\mod W_0]=0$, and 
    $[dg\wedge \omega_1\mod W_0]=[g(g^{-1}dg\wedge \omega_1)\mod W_0]=0$.
    Then we have
    \begin{align*}
    d\varrho(g\omega_1)
     =&(dg\wedge\omega_1+gd\omega_1)\oplus 0\\
     =&\varrho(d(g\omega_1)),\\
    d \varrho((s^{-1}ds-g^{-1}dg)\cdot \omega_2)
    =&d((-g^{-1}dg\omega_2)\oplus \omega_2)\\
    =&(g^{-1}dg d\omega_2)
      \oplus(-d\omega_2, [-g^{-1}dg\wedge\omega_2+g^{-1}dg\wedge\omega_2\mod W_0])\\
    =&(g^{-1}dg  d\omega_2)\oplus- d\omega_2\\
    =&\varrho\circ d((s^{-1}ds-g^{-1}dg)\omega_2).
    \end{align*}
    
    We obtain the following diagram:
    \begin{align*}
    \xymatrix{0\ar[r]&s^{-1}ds\cdot(\Omega^\bullet_f, d)[-1] \ar[d]^{\rho}\ar[r] & {\bm E}_0\ar[d]^{\varrho}\ar[r]& (\Omega^\bullet_f, d)\ar[d]^{\rho}\ar[r]&0\\
    0\ar[r]& s(\widetilde{A}^{\bullet,\bullet})[-1]\ar[r]&B^\bullet\ar[r]& s(\widetilde{A}^{\bullet,\bullet})\ar[r]&0
    }
    \end{align*}
    The compatibility with $\varphi_2$ and $\nu$ follows from this diagram.
    \end{proof}

   Combining the Theorem \ref{icchi}, Theorem \ref{thmmon},  and Proposition \ref{lastprop}
   we attain the following main theorem of this paper.
   \begin{theorem}\label{main theorem}
   The filtrations $F$ and $W$ on $V_f^k$ are
   identified with the Hodge filtration and the weight filtration on $H^k(Y,Y_\infty;\C)$.
    In particular, the rescaling structure $\mathcal{H}_f$ is of Hodge-Tate type
    if and only if
    the mixed Hodge structures
    $(H^k(Y, Y_\infty;\Q),F,W)$ 
    are Hodge-Tate for all $k$.\qed
   \end{theorem}
   We also have the equation 
   \begin{align}\label{hodge number h}
   h^{p, q}(\mathcal{H}_f)=\dim \gr^W_{2p}H^{p+q}(Y,Y_\infty).
   \end{align}
   The right hand side of (\ref{hodge number h}) is denoted by $h^{p, q}(Y, \w)$
   in \S 1. By Lemma \ref{2.8} and Proposition \ref{HT special}, we obtain Theorem \ref{thm intro}, (1). 
   Theorem \ref{thm intro}, (2) follows from Theorem \ref{main theorem} immediately.
   \begin{remark} 
   A similar relation between $V_f$ and $H^{\bullet}(Y,Y_{\infty})$ 
   is obtained in \cite[Theorem (4.3), Theorem (5.3)]{sabmon} in terms of 
   Hodge modules. However, it is not clear whether the weight filtrations
   are the same as ours.
   \end{remark} 
      
    By the strictness of the morphisms of mixed Hodge structures \cite[Theorem (2.3.5)]{delII}, we have the following
    well known fact
    (see  \cite[Corollary 3.8]{petmix}, for example): 
    \begin{lemma}\label{lemstr} 
    Let   $V^i=(V^i_{\Q}, F, W)$ $(i=1,2,3)$ be mixed $\Q$-Hodge structures, 
    where $V_\Q^i$ is the $\Q$-vector space, $F$ is the Hodge filtration on $V_\C^i:=V^i_\Q\otimes\C$,
    and $W$ is the weight filtration for each $i$.
    Assume that we have the following 
    \begin{align*}
     V^1\longrightarrow V^2\longrightarrow V^3
    \end{align*}  
    be a exact sequence of mixed $\Q$-Hodge structures. 
     
    Then, for all $k, p\in \Z$, the sequences 
    $$\gr^{F}_{-p} \gr^{W}_kV^1_\C\longrightarrow  \gr^{F}_{-p} \gr^{W}_kV^2_\C\longrightarrow
    \gr^{F}_{-p} \gr^{W}_kV^3_\C$$
    of complex vector spaces are exact.\qed
    \end{lemma}
    Remark that a mixed $\Q$-Hodge structure $V=(V_\Q, F, W)$ is Hodge-Tate 
    if and only if $$\gr^F_{-p}\gr^W_{p+q}V_\C=0$$ for $p\neq q$. 
    Then, we immediately have the following: 
    \begin{corollary}\label{lemHt}
    Let $V^i$ be as in Lemma $\ref{lemstr}$.
    If $V^1$ and $V^3$ are Hodge-Tate, then so is $V^2$.\qed
  \end{corollary}             
   By the long exact sequence (\ref{long infty}) of mixed Hodge structures, we have the following: 
   \begin{corollary}\label{the cor}
    If the mixed Hodge structures $H^k(Y)$ and $H^k(Y_\infty)$ are of Hodge-Tate type
    for all $k$, 
    then $\mathcal{H}_f$ is of Hodge-Tate type.\qed
   \end{corollary}

  \section{Examples}\label{section Ex}
  In this section,  we shall give some examples of 
  Landau-Ginzburg models $(X, f)$ in \S \ref{section LG}
  such that the induced rescaling structures $\mathcal{H}_f$
  are of Hodge-Tate type. 
  In \S \ref{ellsur}, we consider the case $\dim X=2.$ 
  In \S \ref{torthr}, we consider the case $\dim X=3.$  
  \subsection{Two dimensional examples}\label{ellsur}
  We shall prove the following: 
   \begin{proposition}\label{2-dim prop}
   Let $f:X\to\P^1$ be a rational elliptic surface such that 
   $(f)_\infty$ is reduced normal crossing, and 
  $\pole=|(f)_\infty|$ is a wheel of $d$ smooth rational curves 
  for $2\leq d\leq 9$.  
  Then  the rescaling structure $\mathcal{H}_f$ of $(X, f)$ is of Hodge-Tate type. 
  \end{proposition} \begin{proof}
  Since $X$ is a rational surface, we have $h^{p, q}(X)=0$ for $p\neq q$.
  Since $\pole$ is a wheel of $d$ rational curves, 
  the (co)homology of $\pole$ is of Hodge-Tate type (see \cite[Example 5.34]{petmix} for example).
  We have the exact sequence of mixed Hodge structures \cite[(9.2.1.2)]{delIII}: 
  \begin{align*}
  \cdots\longrightarrow H^{k}(X)\longrightarrow H^k(Y)\longrightarrow H^{k-1}(D)(-1)\longrightarrow \cdots.
  \end{align*}
  By Corollary \ref{lemHt}, 
  it follows that $H^k(Y)$ are Hodge-Tate for all $k$.
  By the Clemens-Schmid exact sequence \cite[(10.14), Theorem (10.16)]{grirec},  we  have 
  the following exact sequence of mixed Hodge structures:
  \begin{align*}
   H^k(D)
     \longrightarrow H^k(Y_\infty)
     \xrightarrow{N} H^k(Y_\infty)(-1)
     \longrightarrow H_{2-k}(D)(-2),
  \end{align*}
  where $0\leq k\leq 2$ and $N$ is the nilpotent endomorphism.
  Since $H^k(D)$ and $H_{2-k}(D)$ are Hodge-Tate, by Corollary \ref{lemHt}, we have the exact sequence 
  \begin{align}\label{CSE2}
  0\longrightarrow A_1\longrightarrow H^k(Y_\infty)\longrightarrow 
  H^k(Y_\infty)(-1)\longrightarrow A_2\longrightarrow 0, 
  \end{align}
  where $A_1$ and $A_2$ are Hodge-Tate.
  Let $p_k(x, y)$ be the Hodge number polynomial of $H^k(Y_\infty)$ 
  (see \cite[(II-1), Lemma 2.8, and (III-2)]{petmix} for example).   
  The exact sequence (\ref{CSE2}) implies
  that $(1-x y)p_k(x, y)=\sum_pa_px^p y^p$ for some $a_p$. 
  Hence, we have $p_k(x, y)=\sum_p b_p x^p y^p$ for some $b_p$.
  Namely, we have that $H^k(Y_\infty)$ is of mixed Hodge Tate for each $k$. 
  By Corollary \ref{the cor}, we have the conclusion.
  \end{proof} 
  By Theorem \ref{main theorem}
  Lemma \ref{2.8}, Proposition \ref{HT special}, (\ref{hodge filter}), and (\ref{hodge number h}), we obtain the following:
  \begin{corollary}\label{2-dim cor}
   Let $(X, f)$ be as in Proposition $\ref{2-dim prop}$.
   Then, we have
   $f^{p, q}(Y, \w)=h^{p, q}(Y, \w)$,
   and $\mathcal{H}_{f|\para=1}$ is special.\qed
  \end{corollary}
  \begin{remark}
  This example was studied by Auroux-Katzarkov-Orlov \cite{aurmir} as homological mirrors of del Pezzo surfaces.
  The equality of Hodge numbers $f^{p, q}(Y, \w)$ and $h^{p, q}(Y, \w)$ was proved by Lunts-Przjalkowski \cite{lunlan}
  who directly computed both of the numbers
  $($The number $f^{p, q}(Y, \w)$ was also computed in Harder's thesis \cite{harthe}$)$.
  Here, we gave a more conceptual proof of the equality.
  To the best of the author's knowledge, the speciality of $\mathcal{H}_{f|\para=1}$ was not known.
  \end{remark}

  \subsection{Three dimensional examples}\label{torthr}
  We consider toric Landau-Ginzburg models considered in Harder's thesis \cite{harthe}.
  \subsubsection{Fano polytope}\label{fan pol}
  Let $M$ be a free Abelian group of rank $3$.
  Put $M_\R:=M\otimes \R$, and $N:=\Hom_\Z(M, \Z)$. 
  We have the natural pairing $\langle \cdot,\cdot\rangle: M\times N\to\Z$.
  Define $N_\R$ similarly.
  We consider an integral polytope $\poly$
  with the following properties: 
  \begin{enumerate}
  \item[(a)] There is a finite set $\{u_\facet\mid \facet \text{ is facet of }\poly\}$ of primitive vectors in $N$ 
  indexed by all facets of $\poly$ 
  such that 
  \begin{align*}\begin{cases}
   \poly&=\left\{m\in M_\R \middle| \langle m, u_\facet \rangle \geq -1, \text{ for all } \facet \right\}, \\
   \facet&=\{m\in \poly \mid \langle m, u_\facet \rangle =-1\}.
   \end{cases}
  \end{align*}
  In particular,  the origin $0\in M$ is contained in the interior of $\poly$.
   \item [(b)] For each facet $\facet$, the set of vertex of $\facet$ form a basis of $M$. 
   In particular, $\facet$ is a triangle whose interior does not contain 
   the point of $M$. 
  \end{enumerate}
  \begin{remark}
   The condition $(\mathrm{a})$ is called reflexivity. 
   The condition $(\mathrm{b})$ implies that the cone generated by $\facet$ is smooth. 
   These cones generates a smooth fan, which defines a smooth Fano variety. 
  \end{remark}
  \subsubsection{Toric varieties}
      For a face $\face$ of $\poly$, 
      let $\cone_\face$ be the cone generated by $\{u_\facet\mid\face\subset \facet\}$.
      We remark that $\cone_\poly=\{0\}$ since $\{0\}$ is the cone generated by empty set.
      Then we have a fan $\Sigma_\poly:=\{\cone_\face\mid \face \text{ is a face of } \poly\}$
      (see \cite[Theorem 2.3.2]{coxtor}, for example).
      Although this fan is not smooth in general, 
      we have a smooth refinement $\Sigma$ of $\Sigma_\poly$. 
      Since the dimension of $\Sigma_\poly$ is $3$, the refinement is given by 
      a triangulation of the convex hull of the set  
      $\{u_\facet \mid \facet \text{ is a facet of } \poly\}$. 
      In particular, together with the condition $(\mathrm{a})$, 
      we may assume that for every primitive vector $u_\ray$ of a ray $\ray$ in $\Sigma$,
      we have $\min_{m\in\poly}\langle m, u_\ray \rangle= -1$. 
      The toric variety corresponding to $\Sigma$ is denoted by $\amb_\Sigma$.
      It contains the algebraic torus $T_N=\mathrm{Spec}( \C[M])$ as 
      an open dense subset.
      Put $\pole_\Sigma:=\amb_\Sigma\setminus T_N$.
  \subsubsection{A non-degenerate Laurent polynomial}
      We consider a Laurent polynomial 
      $$f_\poly(\chi)=\sum_{m\in M}c_m \chi^m \in \C[M], $$ 
      where $c_m$ are complex numbers and 
      $\chi^m$ is the monomial corresponding to $m\in M$.
      The polynomial $f_\poly$ is considered as an algebraic function on $T_N$.
      Since $T_N$ is an open dense subvariety of $\amb_\Sigma$, 
      $f_\poly$ is considered as a meromorphic function on $\amb_\Sigma$,
      whose pole divisor is contained in $\pole_\Sigma$.
      We impose the following non-degenerate condition on $f_\poly$:
          \begin{enumerate}
                 \item[(c)] The convex hull of $\{m\mid c_m\neq 0\}$ in $M_\R$ is $\poly$. 
                 \item[(d)] For every face $\face \subset \poly$, 
               put 
              $ f_\face(\chi):=\sum_{m\in \face}c_m\chi^m$.
              Then, the intersection of $(d f_\face)^{-1}(0)$ and $ f_\face^{-1}(0)$
              in $T_N$ is empty for every $\face$. 
          \end{enumerate}
      The meaning of the non-degenerate condition considering 
      $f_\poly$ as a meromorphic connection on $\amb_\Sigma$
      is explained later.
  \subsubsection{Coordinate system with respect to a cone}\label{torcor}
      Fix an isomorphism $M \simeqto \Z^3$; $m\mapsto (m_1,m_2,m_3)$.
      Let $(e_i)_{i=1}^3$ be a canonical base of $M$ via $M \simeqto \Z^3$.
      We have an isomorphism
      $\C[M]\simeqto \C[x_1^{\pm},x_2^{\pm },x_3^{\pm }]$ by 
      $\chi^m\mapsto x_1^{m_1}x_2^{m_2}x_3^{m_3}$.
      For a maximal cone $\cone\in \Sigma(3)$, take primitive vectors $u_\ray$
      for rays $\ray$ of $\cone$.
        Then the open subvariety $U_\cone=\mathrm{Spec}(\C[\cone^\vee \cap M])$ 
        of $\amb_\Sigma$ have 
      coordinate $(y_\ray)_{\ray\in\sigma(1)}$.
      The relation between the two coordinates is given by $x_i=\prod_\ray y_\ray^{\langle e_i, u_{\ray}\rangle}$.
      The function $f_\poly$ considered as a meromorphic function on $U_\cone$
      is given by
          \begin{align}\label{f poly}
           f_\poly (y)=\sum_{m\in\poly}c_m \prod_{\ray\in \sigma(1)}y^{\langle m, u_\ray\rangle}_\ray.             
          \end{align}
%
%
%
 \subsubsection{Pole orders along invariant divisors}\label{pole order}
     For each ray $\ray\in\Sigma(1)$, we have the divisor 
     $\pole_\ray$ invariant under the action of $T_N$.
      If $\ray\in \cone(1)$, the intersection $U_\cone\cap \pole_\ray$ is given by $\{y_\ray=0\}$.
      Let $\face_\ray$ be a face defined by
      $$\face_{\ray}:=
      \left\{m\in\poly\middle| \langle m, u_{\ray}\rangle=
      \min_{m'\in\poly}\langle m',u_{\ray}\rangle
      =-1
      \right\}.$$ 
      Remark that $\face_\ray\neq \emptyset$.
     The equation (\ref{f poly}) is written as follows: 
          \begin{align}\label{f ray}
              f_\poly(y)=y_\ray^{-1}\left(y_\ray f_{\face_\ray}(y)
                               +y_\ray\sum
                               _{\substack{m\in\poly , \langle m, u_\ray\rangle\geq 0 }}
                               c_m\prod_{\ray'\in \sigma(1)}
                               y_{\ray'}^{\langle m, u_{\ray'}\rangle}
                               \right).
          \end{align}
      Remark that $y_\ray f_{\face_\ray}(y)$ does not depend on $y_\ray$.
      The pole order along $\pole_\ray$ is one.
  \subsubsection{Non-degenerate condition}\label{sNDC}
      For a $\tau\in \cone(2)$, take $\ray,\ray'\in\cone(1)$ so that $\tau=\ray+\ray'$.
      Put $\face_\tau:=\face_\ray\cap\face_{\ray'}$. 
      We have 
       \begin{align}
              f_\poly(y)=y_\ray^{-1}y_{\ray'}^{-1}\left(y_\ray y_{\ray'}f_{\face_\tau}(y)
                               +y_\ray y_{\ray'}\sum
                               _{\substack{m\in\poly \\ \langle m, u_\ray\rangle\geq 0,
                                   \text{or} \langle m, u_{\ray'}\rangle\geq 0  }}
                               c_m\prod_{\ray''\in \sigma(1)}
                               y_{\ray''}^{\langle m, u_{\ray'}\rangle}
                               \right).
          \end{align}
          Note that $y_\ray y_{\ray'}f_{\face_\tau}(y)$ does not depend on $y_\ray$ nor $y_{\ray'}$.
          There is also a similar description of $f_\poly$ for the vertex 
          $\face_\cone=\bigcap_{\ray\in \cone(1)}\face_\ray$.
          From these descriptions, we have the following properties of 
          the zero divisor $(f_\poly)_0$ in $\amb_\Sigma$: 
          \begin{itemize}
          \item The divisor $(f_\poly)_0$ is  a (reduced) smooth hypersurface of $\amb_\Sigma$.
          \item The fixed points of the action of $T_N$ is not contained in $(f_\poly)_0$.
          \item The divisor $\pole_\Sigma\cup (f_\poly)_0$ is simply normal crossing.
          \end{itemize}
  \subsubsection{Base locus} 
      Put $B_\ray:=|(f_{\sf P})_0|\cap D_{\ray}$ for all ray $\ray$ in $\Sigma$.       
          \begin{lemma}
            For every $\ray$, $B_\ray$ is isomorphic to a projective line.
          \end{lemma}
      \begin{proof}
                  By the non-degenerateness of $f_{\sf P}$, all $B_\ray$ are smooth curves in $\amb_\Sigma$.
                  Since $D_\Sigma\cup(f_{\sf P})_0$ is normal crossing, 
                  the intersections of 
                  $B_\ray$ and the lower  dimensional $T_N$-orbits in $D_\ray$
                  are $0$-dimensional.
                  Therefore, it is enough to show that
                  the intersection of $|(f_\poly)_0|$ and the two dimensional orbit in $\pole_\ray$
                  is rational.
                  
                  Take a facet $\facet\subset \poly$ which contains $\face_\ray$.
                  By the assumption $(a)$, $(b)$ in \S \ref{fan pol}, 
                  $\facet$ is a triangle, whose vertexes $e_1, e_2, e_3$ form a $\Z$-basis of $M$.
                  Using this basis, we take an isomorphism $M\simeq \Z^3$. 
                  Let $(x_1,x_2,x_3)$ be the corresponding coordinate as in \S \ref{torcor}.
                  Put $$I:=\{i\in\{1,2,3\}\mid e_i \text{ is a vertex of } \face_\ray\}.$$
                  Remark that $I\neq \emptyset$,
                  and $f_{\face_\ray}=\sum_{i\in I}c_ix_i\neq 0$. 
                  
                  Take $\sigma\in \Sigma (2)$ so that $\ray\in \sigma (1)$.
                  Let $\ray_1:=\ray$, $\ray_2$, $\ray_3$ be the three ray of $\sigma$.
                  Put $y_i:=y_{\ray_i}$ for $i=1,2,3$.
                  Then $g:=y_1f_{\face_\ray}$ is a Laurent polynomial 
                  depending only on $y_2, y_3$. 
                  We need to show that 
                  $\{(y_2,y_3)\in (\C^*)^2\mid g(y_2,y_3)=0\}$ is rational.
                  This space is isomorphic 
                  to the quotient space of 
                  $\{(y_1,y_2,y_3)\in (\C^*)^3\mid f_{\face_\ray}(y_1,y_2,y_3)=0\}$
                  by the $\C^*$-action 
                  defined by 
                  $t\cdot(y_1,y_2,y_3):=(ty_1,y_2,y_3)$.
                  
                  Using the coordinate $(x_1, x_2, x_3)$, 
                  the $\C^*$-action is given by 
                  $t\cdot (x_1,x_2,x_3)=(t^{-1}x_1,t^{-1}x_2, t^{-1}x_3)$ 
                  since $\langle e_i, u_{\ray_1}\rangle=-1$. 
                  We are considering quotient space of 
                  $\{(x_1,x_2,x_3)\in (\C^*)^3\mid \sum_{i\in I}c_ix_i=0\}$. 
                  Since the quotient of 
                   $\{(x_1,x_2,x_3)\in \C^3\mid \sum_{i\in I}c_ix_i=0\}$
                  by the action defined above is a line  in $\P^2$, 
                  we obtain the rationality.
      \end{proof}
  \subsubsection{Blowing ups}
      Take an ordering $\Sigma(1)=\{\ray_1,\dots,\ray_\ell\}$ for the set of all rays in $\Sigma$.
      We consider the following sequence of blowing ups: 
       $$X=X^{(\ell)}\xrightarrow{p^{(\ell-1)}}\cdots\to X^{(j+1)}\xrightarrow{p^{(j)}} X^{(j)}\to 
       \cdots\xrightarrow{p^{(0)}} X^{(0)}=X_\Sigma,$$
      where $p^{(j)}:X^{(j+1)}\to X^{(j)}$ is the blowing up along the strict transform of $B_{\ray_{j+1}}$ in $X^{(j)}$.
      The composition $X\to X_\Sigma$ is denoted by $\pi_\Sigma$.
      The strict transform of $D_{\ray_j}$ 
      is denoted by $D_j$ $(1\leq j\leq \ell)$. 
      \begin{lemma}\label{Bl}
       We have the following: 
       \begin{enumerate}
       \item The divisor $D_j$ is given by the composition of blowing ups of $D_{\ray_j}$ along reduced $0$-schemes. 
       \item The union $D:=\bigcup_jD_j$ is simple normal crossing.      
       \item The pole divisor of  $\pi_\Sigma^* f_\poly$ is reduced and the support $|(\pi_\Sigma^* f_\poly)_\infty|$ is $D$.
       \item The pull back of $f_\poly$ by $\pi_\Sigma$ gives a well defined morphism 
                $\pi_\Sigma^* f_\poly: X\to\P^1.$
       \end{enumerate}
      \end{lemma}
      \begin{proof}
      Let $\pi^{(i)}:X^{(i)}\to X^{(0)}$ be the composition $p^{(i-1)}\circ\cdots \circ p^{(0)}$ for $i=1,2,\dots \ell$.
      We put $\pi^{(0)}:=\id_{X^{(0)}}$. Let $f^{(i)}$ be the pull back of $f_\poly$ by $\pi^{(i)}$ for $i=0,1,\dots,\ell$.
      Let $D^{(i)}_j$ (resp. $B_j^{(i)}$) denote the strict transform of $D_{\ray_j}$ (resp.  $B_{\ray_j}$) in $X^{(i)}$
      for $i, j=1,2,\dots, \ell$.
      Put $D^{(0)}_j:=D_{\ray_j}$ and $B_j^{(0)}:=B_{\ray_j}$, respectively. 
      We define $D^{(i)}:=\bigcup_j D^{(i)}_j$.
      We shall prove the following by the induction on $i$: 
      \begin{enumerate}
      \item[$(1)_i$] The divisor $D^{(i)}_j$ is given by the composition 
                            of blowing ups of $D_j^{(0)}$ along reduced $0$-schemes.
      \item[$(2)_i$] The zero divisor $(f^{(i)})_0$ is a reduced smooth hypersurface of $X^{(i)}$, 
                           and the union $(f^{(i)})_0\cup D^{(i)}$ is simple normal crossing. 
      \item[$(3)_i$] The pole divisor $(f^{(i)})_\infty$ is reduced and the support $|(f^{(i)})_\infty|$ is $D^{(i)}$.
      \item[$(4)_i$] The intersection $(f^{(i)})_0\cap (f^{(i)})_\infty \cap\Big{(}\bigcup_{j=1}^{i}D^{(i)}_j\Big{)} $ is empty.
      \end{enumerate}
      Remark that $(1)_0$, and $(4)_0$ are trivial.  
      We also remark that $(2)_0$ and $(3)_0$ are shown in \S \ref{pole order} and \S \ref{sNDC}.          
      
       Take $i\in \{1,2,\dots, \ell\}$. 
       Assume that $(1)_{i-1}, (2)_{i-1}, (3)_{i-1}, (4)_{i-1}$ holds. 
       Let $Q$ be an arbitrary point in $B_{i}^{(i-1)}$. 
       By the assumption $(2)_{i-1}$, $(3)_{i-1}$,  
       we have a local coordinate system $(U_Q; z_0,z_1, z_2)$ centered at $Q$ 
       with the following properties:
       \begin{enumerate}
       \item[] $D^{(i-1)}\cap U_Q=\bigcup_{i=1}^k\{z_i=0\}$ , $D_{i}^{(i-1)}\cap U_Q=\{z_1=0\}$, and 
                $f_{ |U_Q}^{(i-1)}(z)=z_0\cdot \prod_{i=1}^k z_i^{-1}$,  
       \end{enumerate}where $k=1$, or $2$.
       We have $B_{i}^{(i-1)}\cap U_Q=\{z_0=z_1=0\}$.
       Let $V_Q$ be the inverse image of $U_Q$ by $p^{(i-1)}$.       
       Then we have 
       \begin{align*}
       V_Q=\left\{\big{(}(z_0,z_1,z_2),[w_0:w_1]\big{)}\in U_Q\times \P^1\mid z_0w_1-z_1w_0=0 \right\}.
       \end{align*}       
       If $k=2$ and $\{z_2=0\}=D_{j}^{(i-1)}$ then $j>i$ by the assumption $(4)_{i-1}$. 
       $D^{(i)}_j\cap V_Q$ is given by the blowing up
       of $D_{j}^{(i-1)}\cap U_Q$ at the reduced point $Q$.
       On $V_Q^+:=V_Q\cap\{w_0\neq 0\}$, we have a local coordinate $(u_0,u_1,u_2)$
       with $z_0=u_0$, $z_1=u_0u_1$, $z_2=u_2$, and $w_1/w_0=u_1$.
       We have $f^{(i)}_{|V_Q^+}(u)=\prod_{i=1}^ku_i^{-1}$. 
       The strict transform $D^{(i)}_i\cap V_Q^+$ 
       is given by $\{u_1=0\}$.       
       On $V_Q^-:=V_Q\cap\{w_1\neq 0\}$, we have a local coordinate $(v_0,v_1,v_2)$ 
       with $z_0=v_0v_1$, $z_1=v_1$, $z_2=v_2$, and $w_0/w_1=v_0$.
       We have $f^{(i)}_{V^-_Q}(v)=v_0$ if $k=1$, and $f^{(i)}_{V^-_Q}(v)=v_0v_2^{-1}$ if $k=2$.
       The strict transform $D^{(i)}_i\cap V_Q^-$ 
       is given by $\{v_1=0\}$.          
       By this description and the assumptions, 
       we have $(1)_i,(2)_i,(3)_i,(4)_i$.             
       Then, by the induction, we obtain $(1)_\ell, (2)_\ell, (3)_\ell, (4)_\ell$. 
       It is easy to prove that  $(1)_\ell, (2)_\ell, (3)_\ell, (4)_\ell$ implies the lemma.
       \end{proof}
  \subsubsection{Hodge-Tate condition}
  We obtain the following:
  \begin{proposition}\label{3-dim prop}
  Let $f:\amb\to \P^1$ be the pull back of $f_\poly$ by $\pi_\Sigma$.
  Then the rescaling structure $\mathcal{H}_f$ is of Hodge-Tate type.
 \end{proposition} 
 \begin{proof}  
  By Lemma \ref{Bl}, the pair $(X, f)$ satisfies the condition in \S \ref{section LG}.  
  Since $X$ is given by blowing ups of a toric manifold along projective lines,
  $h^{p, q}(X)=0$ for $p\neq q$ (\cite[Theorem 7.31]{voihod}).
  Since $D_j$ is given by the
  composition of blowing ups of $D_{\ray_j}$ along reduced $0$-schemes (Lemma \ref{Bl} (1)), 
  and each $D_i\cap D_j$ is isomorphic to $\P^1$,
  the (co)homology of $D$ is Hodge-Tate (see \cite[Example 5.34]{petmix} for example).
  Hence, by Lemma \ref{lemHt} and the exact sequence 
  \begin{align*}
  \cdots\longrightarrow H^{k}(X)\longrightarrow H^k(Y)\longrightarrow H^{k-1}(D)(-1)\longrightarrow \cdots,
  \end{align*}  
  we have that the mixed Hodge structure on $H^k(Y)$ is Hodge-Tate for each $k$. 
  By Corollary \ref{the cor}, it remains to show that the limit mixed Hodge structure $H^k(Y_\infty)$ is of 
  Hodge-Tate type.
  From the Clemens-Schmid exact sequence \cite[(10.14), Theorem (10.16)]{grirec}, 
  we obtain the following exact sequence of mixed Hodge structures:
  $$H^k(D)\longrightarrow H^k(Y_\infty)\xrightarrow{N}H^k(Y_\infty)(-1)\longrightarrow H_{4-k}(D)(-3),$$
  where $0\leq k\leq 4$.
  Since by Corollary \ref{lemHt}, 
  we have the exact sequence 
  $$0\longrightarrow A_1\longrightarrow H^k(Y_\infty)\longrightarrow 
  H^k(Y_\infty)(-1)\longrightarrow A_2\longrightarrow 0,$$
  where $A_1$ and $A_2$ are Hodge-Tate.
  Then, by the similar argument as in the proof of Proposition \ref{2-dim prop}, $H^k(Y_\infty)$ is also Hodge-Tate
  for each $k$.
  \end{proof}
  Similarly as Corollary \ref{2-dim cor}, we have the following:
  \begin{corollary}
  Let $(X, f)$ be as in Proposition $\ref{3-dim prop}$.
  Then we have $f^{p, q}(Y, \w)=h^{p, q}(Y, \w)$.
  We also have that $\mathcal{H}_{f|\para=1}$ is special. 
  \qed
  \end{corollary}
  \begin{remark}
     In \cite{harthe}, A. Harder computed the number $f^{p, q} (Y, \w)$ and compare it with 
     the Hodge number of the smooth toric Fano manifold $X_\poly$ associated to $\poly$
     \cite[Theorem 2.3.7]{harthe}. 
     In \cite{reilog}, Reichelt-Sevenheck studied hypergeometric $\dmod$-module 
     associated to $($a family of$)$ $f_\poly$, and solved a kind of Birkhoff problem.
     The result here is a priori different from theirs 
     since the cohomology considered here is different from the one considered in \cite{reilog}.
     We also remark that 
     T. Mochizuki informed that we can obtain similar but a priori different results
     from the viewpoint of twistor $\dmod$-modules. 
   \end{remark}

   \appendix
  \section{Rescaling structures for quantum $\dmod$-modules of Fano manifolds}\label{app A}
  
  \subsection{Square roots of Tate twists}
  We use the notation in \S \ref{def and eg of scaling}.
  Set $\T^{1/2}:=\mathcal{O}_S(*(\rmod)_\infty)w$ where $w$ is a global section 
  with $\deg w=1$.
  We define a connection $\nabla$ on $\T^{1/2}$ by $\nabla w:=-(1/2) w\rmod^{-1}d \rmod$.  
  Since $p_2^*(\T^{1/2},\nabla)$ is not isomorphic to $\sigma^*(\T^{1/2},\nabla)$, 
  $(\T^{1/2},\nabla)$ is not equipped with a rescaling structure.
  However, we have a flat isomorphism
   $(\T^{1/2})^{\otimes 2}\simeqto \T; w^{\otimes 2}\mapsto v$. 
  Hence we use the notation $\T(-1/2):=(\T^{1/2},\nabla).$
  For each $k\in \Z$, we define 
  \begin{align}
   \T(-k/2):=\begin{cases}
                     \T(-1/2)^{\otimes k} &(k\geq 0)\\
                     (\T(-1/2)^\vee)^{\otimes -k} &(k<0).
                  \end{cases}
  \end{align}
  In the case where $k\in 2\Z$, $\T(k/2)$ is identified with the rescaling structure defined in Example $\ref{Tate twist}$.
  For a meromorphic connection $(\mathcal{H},\nabla)$ as in Definition $\ref{definition of rescaling structure}$, 
  we also define $\mathcal{H}(k/2):=\mathcal{H}\otimes \T(k/2)$.
  \subsection{Tate twisted quantum $\dmod$-modules}
  Let $\fano$ be a smooth projective Fano variety over $\C$ of dimension $n$. 
  Put $\HH_a(\fano):=\bigoplus_{a=q-p}H^q(\fano,\Omega_\fano^p)$.
  Set $\HH_\bullet(\fano):=\bigoplus_a\HH_a(\fano)$ and identify it with $H^\bullet(\fano;\C)$ by the Hodge decomposition.  
  Let $\star_{\para}$ be the quantum cup product of $\fano$ with respect to 
  the parameter $c_1(\fano)\log \para\in H^2(\fano;\C)$, where $c_1(\fano)$ is the first Chern class of
  the tangent bundle of $\fano$. 
  This is well defined for all $\para\in\C$. Indeed, the right hand side of
  \begin{align}\label{A.2}
  (\alpha\star_{\para} \beta, \gamma)_\fano=
  \sum_{d\in H_2(\fano;\Z)}\langle \alpha,\beta, \gamma\rangle^\fano_{0,3,d}\para^{c_1(\fano)\cdot d}
  \end{align}
  is a finite sum since $\fano$ is Fano,
  where $\alpha, \beta, \gamma\in H^\bullet(\fano;\C)\simeq \HH_\bullet(\fano)$,
  $(\cdot,\cdot)_\fano$ denotes the Poincar\'e pairing, 
  and $\langle\cdot,\cdot,\cdot \rangle^\fano_{0,3,d}$
  denotes genus-zero 3-points Gromov-Witten invariant of degree $d\in H_2(\fano;\Z)$ 
  (see \cite{behgro}, \cite{behsta},  \cite{coxmir}, and references therein). 
  
  For any non-negative integer $k$, we take a finite rank free $\mathcal{O}_S(*(\rmod)_\infty)$-module
  ${}^\mathfrak{a}{H}^k:=\HH_{k-n}(\fano)\otimes \mathcal{O}_S(*(\rmod)_\infty)$.
  The $\Z$-grading of ${}^\mathfrak{a}{H}^k$ is defined to be $0$.
  Define $\mu_\fano\in \End(\HH_{k-n}(\fano))$ 
  by $\mu_{\fano|H^q(\fano,\Omega_\fano^p)}:=(p+q-n)/2 \cdot \id_{H^q(\fano,\Omega_\fano^p)}$.
  We also have an endomorphism 
  $c_1(\fano)\star_{\para}$ on $\HH_{k-n}(\fano)$.
  We have the Dubrovin connection ${}^\mathfrak{a}\nabla$  
   on ${}^\mathfrak{a}{H}^k$ as follows (\cite{dubgeo2}, \cite{dubgeo}, \cite{dubpai}):
  \begin{align*}
  {}^\mathfrak{a} \nabla:=d+\frac{c_1(\fano)\star_{\para}}{\rmod}\frac{d\para}{\para}+\mu_\fano\frac{d\rmod}{\rmod}
                    -c_1(\fano)\star_{\para}\frac{d\rmod}{\rmod^2}.
  \end{align*}
  \begin{proposition}\label{pA2}
  $\mathcal{H}_\fano^k:={}^\mathfrak{a}{H}^k(-k/2)$ comes equipped with a  
  rescaling structure.
  \end{proposition}
  \begin{proof}
  $\mathcal{H}_\fano^k$ is identified with the free $\mathcal{O}_S(*(\rmod)_\infty)$-module
  $\HH_{k-n}(\fano)\otimes \mathcal{O}_S(*(\rmod)_\infty)$ with the connection: 
  \begin{align*}
   \nabla=d+\frac{c_1(\fano)\star_{\para}}{\rmod}\frac{d\para}{\para}
              +\left(\mu_\fano-\frac{k}{2}\cdot \id\right)\frac{d\rmod}{\rmod}
                    -c_1(\fano)\star_{\para}\frac{d\rmod}{\rmod^2}.
  \end{align*}  
  Taking the pull back by $\sigma:\C^*_\scale\times S\to S; (\scale, \rmod, \para)\mapsto (\scale \rmod, \scale \para)$,
  we have 
   \begin{align*}
   \sigma^*\nabla=d+\frac{c_1(\fano)\star_{\scale\para}}{\scale\rmod}
                            \frac{d\para}{\para}+\left(\mu_\fano-\frac{k}{2}\cdot \id\right)
                            \left(\frac{d\rmod}{\rmod}+\frac{d\scale}{\scale}\right)
                            -\frac{c_1(\fano)\star_{\scale\para}}{\scale}\frac{d\rmod}{\rmod^2}.
  \end{align*}
  Put $\mu_k:=\mu_\fano-(k/2)\cdot\id$.
  On $H^q(\fano,\Omega_\fano^p)$ with $q-p=k-n$, 
  we have $\mu_k=(q-k)\cdot \id=(p-n)\cdot\id$. 
  Hence we have a morphism of $\mathcal{O}_{\C_\scale^*\times S}(*(\rmod)_\infty)$-modules:
  $$\theta^{-\mu_k}:p_2^*\mathcal{H}_\fano^k\simeqto \m^*\mathcal{H}_\fano^k.$$
  By (\ref{A.2}),  we obtain
  $$c_1\star_\para =\scale^{\mu_k}\left(\frac{c_1(\fano)\star_{\scale \para}}{\scale}\right)\scale^{-\mu_k},$$
  which implies that $\scale^{-\mu_k}$ is flat with respect to the connections 
  (see \cite[\S 2.2]{ggi} for example).
  \end{proof}
  \begin{definition}  
  We define a rescaling structure $\mathcal{H}_\fano$ by
  \begin{align*}
   \mathcal{H}_\fano:=\bigoplus_{k\in \Z}\mathcal{H}^k_\fano.
  \end{align*}
  We call $ \mathcal{H}_\fano$ a Tate twisted quantum $\dmod$-module of $\fano$.
  \end{definition}
  \begin{remark}
   The $\Z/2\Z$-graded flat meromorphic connection 
   ${}^\mathfrak{a} H$ in the introduction  $($or \cite{kkp14} $)$ 
   is given by ${}^\mathfrak{a} H=\bigoplus_k{}^\mathfrak{a}{H}^k$,
   where the $\Z/2\Z$-grading on 
   ${}^\mathfrak{a}{H}^k$ is given by $(k\mod 2)$.
  \end{remark}
    \subsection{Hodge-Tate condition and Hodge numbers}
    The fiber of $\mathcal{H}_\fano^{k}$
    at $(\rmod,\para)=(1,0)$
     is naturally identified with 
    $\HH_{k-n} (\fano)$.
    We shall describe the Hodge and weight filtrations 
    on $\HH_{k-n} (\fano)$
    in the sense of \S \ref{section hodge}.
    
    As we have seen in the proof of Proposition \ref{pA2}, 
    the $\C^*$-action on $\mathcal{H}^k_{\fano|\para=0}$ 
    is given by $\scale^{-(q-k)}=\scale^{-(p-n)}$ on 
    $ H^q(\fano,\Omega_\fano^p)\otimes \mathcal{O}_{\C_\rmod}$
    with $q-p=k-n$. 
    Hence the Hodge filtration on $\HH_{k-n}(\fano)$ is given as follows: 
    \begin{align}\label{h qdm}
     F_{i}\HH_{k-n}(\fano)=\bigoplus_{\substack{ p-n\leq i,\\ q-p=k-n}}H^q(\fano,\Omega^p_\fano).
    \end{align}
    We obtain $f^{p, q}(\mathcal{H}_\fano)=\dim H^{q}(\fano,\Omega^{n-p}_\fano)=h^{n-p, q}(\fano)$.
        
    The residue endomorphism 
    $N_k:=\Res_\para\nabla$ on $\HH_{k-n}(\fano)$ is identified with 
    $c_1(\fano)\cup$.
    It follows that the monodromy weight filtration centered at $k$ is 
    given as follows: 
    \begin{align}\label{w qdm}
    {}^k W_i\HH_{k-n}(\fano)=\bigoplus_{\substack{p\geq n-i/2 ,\\q-p=k-n}}H^q(\fano, \Omega^p_\fano).
    \end{align}
    Hence, we have 
    $h^{p, q}(\mathcal{H}_\fano)=h^{n-p, q}(\fano).$
    By (\ref{h qdm}) and (\ref{w qdm}), we obtain the following:
    \begin{proposition}
    The Tate twisted quantum $\dmod$-module $\mathcal{H}_\fano$
    satisfies the Hodge-Tate condition 
    for any smooth projective Fano variety $\fano$.\qed
    \end{proposition}
    
    \section{Relation to the work of Katzarkov-Konstevich-Pantev}\label{app kkp}
    \subsection{Tame compactified Landau-Ginzburg model}
    In \cite{kkp14}, Katzarkov-Kontsevich-Pantev considered the following: 
    \begin{definition}[{\cite[Definition 2.4, (T)]{kkp14}, See also \cite[Definition 3]{lunlan}}]
    A tame compactified Landau-Ginzburg model is a tuple $((X, f), D, {\sf{vol}}_X)$, where
    \begin{enumerate}
        \item $X$ is a smooth projective variety and $f:X\to \P^1$ is a flat projective morphism. 
        \item $D=(\bigcup_iD^{\sf h}_i)\cup(\bigcup_jD_j^{\sf v})\subset X$ is a reduced normal crossing divisor such that 
           \begin{enumerate}
             \item $D^{\sf v}=\bigcup_jD_j^{\sf v}$ is a scheme theoretic pole divisor of $f$, i.e. $(f)_\infty=D^{\sf v}$.
                      In particular, the pole order of $f$ along $D^{\sf v}_j$ is one;
             \item each component $D_i^{\sf h}$ of $D^{\sf h}:=\bigcup_iD_i^{\sf h}$ is smooth and horizontal for $f$, 
                      i.e. $f_{|D_i^{\sf h}}$ is a flat morphism;
             \item the critical locus of $f$ does not intersect $D^{\sf h}$.
           \end{enumerate}
        \item ${\sf vol}_X$ is a nowhere vanishing meromorphic section of the canonical bundle 
                  $K_X$ with poles of order exactly one
                 along each component of $D$.  
                 In other words, we have an isomorphism $\mathcal{O}_X\simeqto K_X(D); 1\mapsto {\sf{vol}}_X$.              
    \end{enumerate}
    \end{definition}
    In this paper (\S \ref{section LG}), the horizontal divisor $D^{\sf h}$
    is assumed to be empty, and each component $D_j^{\sf v}$ is assumed to be smooth.
    Although we do not impose the existence of 
    ${\sf{vol}}_X$ in \S \ref{section LG},
    all examples in \S \ref{section Ex} have ${{\sf vol}}_X$.
    
 \subsection{Landau-Ginzburg Hodge numbers}
     The Hodge number $f^{p, q}(Y, \w)$ in this paper 
     corresponds to $f^{q, p}(Y,{\sf w})$ in \cite[Definition 3.1]{kkp14}.
     The definition in this paper suits to the convention in the classical Hodge theory.
     The number $h^{p ,q}(Y, \w)$ 
     in \cite{kkp14} is $\dim \gr^{W}_pH^{p+q}(Y,Y_\infty)$
     in our notation. 
     Our definition of $h^{p, q}(Y, \w)$ is $\dim \gr^{W}_{2p}H^{p+q}(Y,Y_\infty)$, 
     which is different from their definition.
     As mentioned in \cite{lunlan}, 
     their definition seems not to be what they had in mind. 
     The definition of $h^{p, q}(Y, \w)$ in this paper 
     corresponds to $h^{q, p}(Y, \w)$ in \cite[Definition 3]{lunlan}.
     In \cite{lunlan}, they also gave a counter-example 
     for the part of equality with the numbers $i^{p, q} (Y, \w)$ in \cite[Conjecture 3.6]{kkp14}. 
    \subsection{One parameter families}
      Recall that $S=\P_\rmod^1\times\C_\para$.
      We also recall that  
      $\pi_S:S\times\amb\to \amb$ and 
      $p_S:S\times\amb\to S$ denote the projections.     
      Put 
      $$\Omega_{X,S}^k(*\pole):=
      \mathcal{O}_{\amb\times S}(*(\rmod)_\infty)\otimes \pi_S^{-1}\Omega_\amb^k(*\pole).$$
      Let ${}^{\mathfrak{b}}H^k$ be the $\mathcal{O}_S(*(\rmod)_\infty)$-module defined by
     $${}^\mathfrak{b}H^k:=
     \R^kp_{S*}(\Omega_{X,S}^\bullet(*\pole),
     \rmod d+\para df\wedge). $$ 
     Let $\nabla:\Omega^\bullet_{X, S}(*D)\to 
     \Omega^{\bullet}_{\amb, S}(*D)\otimes p_S^*\Omega^1_S(*|(\rmod\para)_0|)$
     be the connection on $ \Omega^{\bullet}_{\amb, S}(*\pole):=\bigoplus_k \Omega^k_{\amb, S}(*\pole)$ defined by 
     $$\nabla= d_S+\frac f{\rmod}d \para +{\sf G}\frac{d\rmod}{\rmod}-\para f\frac{d\rmod}{\rmod^2},$$
     where ${\sf G}=-(k/2)\id$ on $\Omega_{\amb, S}^k(*\pole)$.
     Then, we have 
     $[\nabla_{\del_\para},\rmod d+\para df\wedge]=0$, 
     and $[\nabla_{\del_\rmod},\rmod d+\para df\wedge]=(2\rmod)^{-1}(\rmod d+\para df\wedge)$.
     Let $\mathscr{A}_\amb^{p, q}$ be the sheaf of $(p, q)$-forms on $\amb$
     and $\del$ and $\overline{\del}$ be the Dolbeault operators.
     Put 
     $\mathscr{A}^{p, q}_{\amb, S,\pole}
     :=\Omega^p_{\amb, S}(*\pole)
        \otimes_{\pi_S^{-1}\mathcal{O}_\amb}\pi_S^{-1}\mathscr{A}_\amb^{0,q}$.    
     Let
     $\del:\mathscr{A}^{p, q}_{\amb, S, \pole}\to\mathscr{A}^{p+1,q}_{\amb, S, \pole}$, and
     $\overline{\del}:\mathscr{A}^{p, q}_{\amb, S, \pole}\to\mathscr{A}^{p, q+1}_{\amb, S, \pole}$
     be the induced operators.
     Put $\mathscr{A}_{\amb, S, \pole}^\ell:=\bigoplus_{p+q=\ell}\mathscr{A}^{p, q}_{\amb, S, \pole}$
     and 
     $$d_{\rm tot}:=\rmod\del+\overline{\del}+\para\del f:
     \mathscr{A}_{\amb, S,\pole}^\ell \to
     \mathscr{A}_{\amb, S, \pole}^{\ell+1}.$$
     We have a natural quasi-isomorphism
     \begin{align*}
     \iota_{\rm Dol}:(\Omega_{X, S}^\bullet(*\pole), \rmod d+\para d{f})
     \simeqto 
     (\mathscr{A}_{\amb, S, \pole}^\bullet, d_{\rm tot}).
     \end{align*}
     We also have the connection 
     ${\bm\nabla}:\mathscr{A}^\bullet_{\amb, S, \pole}\to 
     \mathscr{A}^\bullet_{\amb, S, \pole}\otimes \Omega^1_S(*|(\rmod\para)_0|)$
     by 
     $$ {\bm\nabla}:=d_S+\frac{f}{\rmod}d\para+\mu_f\frac{d\rmod}{\rmod}-\para f\frac{d\rmod}{\rmod^2}, $$
     where $\mu_{f|\mathscr{A}^{p, q}_{X, S, D}}=2^{-1}(q-p)\cdot \id$.
     Then $\iota_{\rm Dol}\circ \nabla={\bm\nabla}\circ \iota_{\rm Dol}$ by definition.
     We have $[{\bm\nabla}_{\del_\para},d_{\rm tot}]=0$, and 
    \begin{align*} 
     [\bm{\nabla}_{\del_\rmod}, d_{\rm tot}]
     &=[\del_\rmod+\rmod^{-1}\mu_f-\rmod^{-2}\para f, \rmod\del+\overline{\del}+\para \del f]\\
     &=\del-(1/2)\del+(1/2)\rmod^{-1}\overline{\del}-(1/2)\rmod^{-1}\para\del f
         +\rmod^{-1}\para\del f\\
     &=(2\rmod)^{-1}(\rmod\del+\overline{\del}+\para\del f)=(2\rmod)^{-1}d_{\rm tot}.
     \end{align*}
     Hence $\bm\nabla$ gives a connection ${}^\mathfrak{b}\nabla^k$
     on $^{\mathfrak{b}}H^k
             \simeq\mathscr{H}^kp_{S *}(\mathscr{A}^\bullet_{X,S,D},d_{\rm tot})$.
     We remark that similar discussions are given in \cite{eficyc} and \cite{kkp08}.
     
    \begin{lemma}
     For each $k\in \Z_{\geq0}$, 
      we have $({}^{\mathfrak{b}}H^k,{}^\mathfrak{b}\nabla^k) (-k/2)\simeq \mathcal{H}_f^k$.  
     \end{lemma}
     \begin{proof}
    We have a natural isomorphism 
    $({}^{\mathfrak{b}}H^k,{}^\mathfrak{b}\nabla^k) (-k/2)
    \simeq({}^{\mathfrak{b}}H^k,{}^\mathfrak{b}\nabla^k-(k/2)\rmod^{-1}d\rmod)$. 
    Then the connection ${}^\mathfrak{b}\nabla^k-(k/2)\rmod^{-1}d\rmod$
    is induced from the following
    connection on $\mathscr{A}_{X, S, D}^\bullet$:
    $${\bm\nabla}':=d_S+\frac{f}{\rmod}d\para+P\frac{d\rmod}{\rmod}-\para f\frac{d\rmod}{\rmod^2}, $$
    where $P_{|\mathscr{A}_{X, S, D}^{p, q}}=2^{-1}((q-p)-(p+q))\cdot \id=(-p)\cdot\id$.
    Remark that $[{\bm\nabla}',d_{\rm tot}]=0$.
    Moreover, 
    it is induced from the following connection on $\Omega^\bullet_{X,S}(*D)$:
    $$\nabla'= d_S+\frac f{\rmod}d \para +{P}\frac{d\rmod}{\rmod}-\para f\frac{d\rmod}{\rmod^2},$$
    where $P_{|\Omega_{X, S}^p(*D)}=(-p)\cdot \id$.
    We also remark that $[\nabla', \rmod d+\para df]=0$.
    Then, the quasi-isomorphism
    $${\rm{iso}}:(\Omega^\bullet_{f,\rmod,\para},d+\rmod^{-1}\para df)
    \simeqto (\Omega_{X,S}^\bullet(*D),\rmod d+\para df)$$
    on $S^*\times X=(\C_\rmod^*\times\C_\para^*)\times X$
    defined by ${\rm{iso}}_{|\Omega^p_{f,\rmod,\para}}=\rmod^p$
    induces the conclusion naturally.
    \end{proof}
    \begin{remark}
     It seems that the connection on $^{\mathfrak{b}}H$ which
     Katzarkov-Kontsevich-Pantev had in mind in \cite[(3.2.2)]{kkp14} 
     was the one where $\sf f$ is replaced by $q{\sf f}$.  
     The dual of it $($or, the connection 
     $({}^{\mathfrak{b}}H,{}^\mathfrak{b}\nabla)$
     defined firstly in \cite[\S 3.2.2]{kkp14}$)$ is isomorphic to 
     $\bigoplus_{k\in\Z}(^\mathfrak{b}H^k, ^{\mathfrak{b}}\nabla^k)$.
    \end{remark}

    \subsection*{Acknowledgement}
         The author hopes to express 
         the deepest appreciation to his supervisor Takuro Mochizuki.
         The discussions with him are always enlightening to the author.
         He would like to thank Fumihiko Sanda for very useful comments 
         especially on the mirror symmetry and quantum cohomology rings.
         He is grateful to Claus Hertling and Claude Sabbah for
         their advise, encouragement, and kindness
         in many occasions.
         In particular, C. Hertling explained the paper \cite{hernil} to the author, and
         C. Sabbah indicated the author to the notion of rescaling in the irregular Hodge theory.
         He also thank Thomas Reichelt for explaining his results in \cite{reilog}. 
         He is grateful to the referee for his/her useful remarks and comments.

          The author was supported by Grant-in-Aid for JSPS Research Fellow number
          16J02453 and the Kyoto Top Global University Project (KTGU).       
  \bibliographystyle{plain}


\begin{thebibliography}{10}

\bibitem{aurmir}
D.~Auroux, L.~Katzarkov, and D.~Orlov.
\newblock Mirror symmetry for del {P}ezzo surfaces: vanishing cycles and
  coherent sheaves.
\newblock {\em Invent. Math.}, 166(3):537--582, 2006.

\bibitem{aurwei}
D.~Auroux, L.~Katzarkov, and D.~Orlov.
\newblock Mirror symmetry for weighted projective planes and their
  noncommutative deformations.
\newblock {\em Annals of Mathematics}, 167(3):867--943, 2008.

\bibitem{behgro}
K.~Behrend.
\newblock Gromov-{W}itten invariants in algebraic geometry.
\newblock {\em Invent. Math.}, 127(3):601--617, 1997.

\bibitem{behsta}
K.~Behrend and Y.~I. Manin.
\newblock Stacks of stable maps and {G}romov-{W}itten invariants.
\newblock {\em Duke Math. J.}, 85(1):1--60, 1996.

\bibitem{coxmir}
D.~A. Cox and S.~Katz.
\newblock {\em Mirror symmetry and algebraic geometry}, volume~68 of {\em
  Mathematical Surveys and Monographs}.
\newblock American Mathematical Society, Providence, RI, 1999.

\bibitem{coxtor}
D.~A. Cox, J.~B. Little, and H.~K. Schenck.
\newblock {\em Toric varieties}, volume 124 of {\em Graduate Studies in
  Mathematics}.
\newblock American Mathematical Society, Providence, RI, 2011.

\bibitem{delII}
P.~Deligne.
\newblock Th\'eorie de {H}odge. {II}.
\newblock {\em Inst. Hautes \'Etudes Sci. Publ. Math.}, (40):5--57, 1971.

\bibitem{delIII}
P.~Deligne.
\newblock Th{\'e}orie de hodge : {I}{I}{I}.
\newblock {\em Publications Math{\'e}matiques de l'IH{\'E}S}, 44:5--77, 1974.

\bibitem{delloc}
P.~Deligne.
\newblock Local behavior of {H}odge structures at infinity.
\newblock In {\em Mirror symmetry, {II}}, volume~1 of {\em AMS/IP Stud. Adv.
  Math.}, pages 683--699. Amer. Math. Soc., Providence, RI, 1997.

\bibitem{delirr}
P.~Deligne, B.~Malgrange, and J-P. Ramis.
\newblock {\em Singularit\'es irr\'eguli\`eres}, volume~5 of {\em Documents
  Math\'ematiques (Paris) [Mathematical Documents (Paris)]}.
\newblock Soci\'et\'e Math\'ematique de France, Paris, 2007.
\newblock Correspondance et documents. [Correspondence and documents].

\bibitem{dimont}
A.~Dimca and M.~Saito.
\newblock On the cohomology of a general fiber of a polynomial map.
\newblock {\em Compositio Math.}, 85(3):299--309, 1993.

\bibitem{sabgau}
A.~Douai and C.~Sabbah.
\newblock Gauss-{M}anin systems, {B}rieskorn lattices and {F}robenius
  structures. {I}.
\newblock In {\em Proceedings of the {I}nternational {C}onference in {H}onor of
  {F}r\'ed\'eric {P}ham ({N}ice, 2002)}, volume~53, pages 1055--1116, 2003.

\bibitem{dubgeo2}
B.~Dubrovin.
\newblock Geometry of {$2$}{D} topological field theories.
\newblock In {\em Integrable systems and quantum groups ({M}ontecatini {T}erme,
  1993)}, volume 1620 of {\em Lecture Notes in Math.}, pages 120--348.
  Springer, Berlin, 1996.

\bibitem{dubgeo}
B.~Dubrovin.
\newblock Geometry and analytic theory of {F}robenius manifolds.
\newblock In {\em Proceedings of the {I}nternational {C}ongress of
  {M}athematicians, {V}ol. {II} ({B}erlin, 1998)}, number Extra Vol. II, pages
  315--326, 1998.

\bibitem{dubpai}
B.~Dubrovin.
\newblock Painlev\'e transcendents in two-dimensional topological field theory.
\newblock In {\em The {P}ainlev\'e property}, CRM Ser. Math. Phys., pages
  287--412. Springer, New York, 1999.

\bibitem{eficyc}
A.~I. Efimov.
\newblock Cyclic homology of categories of matrix factorizations.
\newblock {\em International Mathematics Research Notices}, page rnw332, 2017.

\bibitem{elzmix}
Fouad El~Zein and L\^e~D\ ung Tr\'ang.
\newblock Mixed {H}odge structures.
\newblock In {\em Hodge theory}, volume~49 of {\em Math. Notes}, pages
  123--216. Princeton Univ. Press, Princeton, NJ, 2014.

\bibitem{esy}
H.~Esnault, C.~Sabbah, and J-D. Yu.
\newblock ${E}_1$-degeneration of the irregular {H}odge filtration (with an
  appendix by {M}. saito).
\newblock {\em Journal f{\"u}r die reine und angewandte Mathematik}, 2015.

\bibitem{ggi}
S.~Galkin, V.~Golyshev, and H.~Iritani.
\newblock Gamma classes and quantum cohomology of {F}ano manifolds: gamma
  conjectures.
\newblock {\em Duke Math. J.}, 165(11):2005--2077, 2016.

\bibitem{grirec}
P.~Griffiths and W.~Schmid.
\newblock Recent developments in {H}odge theory: a discussion of techniques and
  results.
\newblock pages 31--127, 1975.

\bibitem{guisur}
F.~Guill\'en and V.~Navarro~Aznar.
\newblock Sur le th\'eor\`eme local des cycles invariants.
\newblock {\em Duke Math. J.}, 61(1):133--155, 1990.

\bibitem{harthe}
A.~Harder.
\newblock {\em The {G}eometry of {L}andau-{G}inzburg models}.
\newblock PhD thesis, University of Alberta, 2016.

\bibitem{hernil}
C.~Hertling and C.~Sevenheck.
\newblock Nilpotent orbits of a generalization of {H}odge structures.
\newblock {\em J. Reine Angew. Math.}, 609:23--80, 2007.

\bibitem{illcom}
Luc Illusie.
\newblock {\em Complexe cotangent et d\'eformations. {I}}.
\newblock Lecture Notes in Mathematics, Vol. 239. Springer-Verlag, Berlin-New
  York, 1971.

\bibitem{kkp08}
L.~Katzarkov, M.~Kontsevich, and T.~Pantev.
\newblock Hodge theoretic aspects of mirror symmetry.
\newblock In {\em From {H}odge theory to integrability and {TQFT}
  tt*-geometry}, volume~78 of {\em Proc. Sympos. Pure Math.}, pages 87--174.
  Amer. Math. Soc., Providence, RI, 2008.

\bibitem{kkp14}
L.~Katzarkov, M.~Kontsevich, and T.~Pantev.
\newblock Bogomolov-{T}ian-{T}odorov theorems for {L}andau-{G}inzburg models.
\newblock {\em J. Differential Geom.}, 105(1):55--117, 2017.

\bibitem{lanont}
A~Landman.
\newblock On the picard-lefschetz transformation for algebraic manifolds
  acquiring general singularities.
\newblock {\em Transactions of the American Mathematical Society}, 181:89--126,
  1973.

\bibitem{lunlan}
V.~Lunts and V.~Przyjalkowski.
\newblock Landau-{G}inzburg {H}odge numbers for mirrors of del {P}ezzo
  surfaces.
\newblock arXiv:1607.08880, 2016.

\bibitem{mocgkz}
{T}. Mochizuki.
\newblock Twistor property of {GKZ}-hypergeometric systems.
\newblock {arXiv:1501.04146}, 2015.

\bibitem{moctwi}
{T}. Mochizuki.
\newblock A twistor approach to the {K}ontsevich complexes.
\newblock To appear in manuscripta math., 2017.

\bibitem{petmix}
A.~M. Peters, C and J.~H.~M. Steenbrink.
\newblock {\em Mixed {H}odge structures}, volume~52 of {\em Ergebnisse der
  Mathematik und ihrer Grenzgebiete. 3. Folge. A Series of Modern Surveys in
  Mathematics [Results in Mathematics and Related Areas. 3rd Series. A Series
  of Modern Surveys in Mathematics]}.
\newblock Springer-Verlag, Berlin, 2008.

\bibitem{reilog}
T.~Reichelt and C.~Sevenheck.
\newblock Logarithmic frobenius manifolds, hypergeometric systems and quantum
  {$\mathscr{D}$}-modules.
\newblock {\em Journal of Algebraic Geometry}, 24(2):201--281, 2015.

\bibitem{sabmon}
C.~Sabbah.
\newblock Monodromy at infinity and {F}ourier transform.
\newblock {\em Publ. Res. Inst. Math. Sci.}, 33(4):643--685, 1997.

\bibitem{sabhyp}
C.~Sabbah.
\newblock Hypergeometric periods for a tame polynomial.
\newblock {\em Port. Math. (N.S.)}, 63(2):173--226, 2006.

\bibitem{sabiso}
C.~Sabbah.
\newblock {\em Isomonodromic deformations and {F}robenius manifolds}.
\newblock Universitext. Springer-Verlag London, Ltd., London; EDP Sciences, Les
  Ulis, french edition, 2007.
\newblock An introduction.

\bibitem{sabirr}
C.~Sabbah.
\newblock Irregular {H}odge theory.
\newblock {\em arXiv:1511.00176}, 2015.

\bibitem{sabont}
C.~Sabbah and J-D. Yu.
\newblock On the irregular {H}odge filtration of exponentially twisted mixed
  {H}odge modules.
\newblock {\em Forum Math. Sigma}, 3:e9, 71, 2015.

\bibitem{sabfro}
Claude Sabbah.
\newblock Frobenius manifolds: isomonodromic deformations and infinitesimal
  period mappings.
\newblock {\em Expositiones mathematicae}, 16:1--58, 1998.

\bibitem{saimod}
M.~Saito.
\newblock Modules de {H}odge polarisables.
\newblock {\em Publ. Res. Inst. Math. Sci.}, 24(6):849--995 (1989), 1988.

\bibitem{saiont}
Morihiko Saito.
\newblock On the structure of brieskorn lattice.
\newblock {\em Ann. Inst. Fourier (Grenoble)}, 39(1):27--72, 1989.

\bibitem{simnon}
C.~T. Simpson.
\newblock Nonabelian {H}odge theory.
\newblock In {\em Proceedings of the {I}nternational {C}ongress of
  {M}athematicians, {V}ol.\ {I}, {II} ({K}yoto, 1990)}, pages 747--756. Math.
  Soc. Japan, Tokyo, 1991.

\bibitem{stelim}
J.~H.~M. Steenbrink.
\newblock Limits of {H}odge structures.
\newblock {\em Invent. Math.}, 31(3):229--257, 1975/76.

\bibitem{uedhom}
K.~Ueda.
\newblock Homological mirror symmetry for toric del pezzo surfaces.
\newblock {\em Communications in mathematical physics}, 264(1):71--85, 2006.

\bibitem{voihod}
C.~Voisin.
\newblock {\em Hodge theory and complex algebraic geometry. {I}}, volume~76 of
  {\em Cambridge Studies in Advanced Mathematics}.
\newblock Cambridge University Press, Cambridge, english edition, 2007.
\newblock Translated from the French by Leila Schneps.

\end{thebibliography}

  \end{document}